\documentclass[11pt, a4paper]{article}

\usepackage{geometry}
\usepackage{layout}

\usepackage{amsmath,amssymb,amsfonts,amsthm}
\usepackage[usenames]{color}
\usepackage{graphicx}

\usepackage[colorlinks=true,citecolor=black,linkcolor=black,urlcolor=blue]{hyperref}

\usepackage{float}

\usepackage{enumerate} 

\usepackage{etoolbox}
\makeatletter
\AfterEndEnvironment{thm}{\everypar{\setbox\z@\lastbox\everypar{}}}
\AfterEndEnvironment{theorem}{\everypar{\setbox\z@\lastbox\everypar{}}}
\AfterEndEnvironment{lemma}{\everypar{\setbox\z@\lastbox\everypar{}}}
\AfterEndEnvironment{definition}{\everypar{\setbox\z@\lastbox\everypar{}}}
\AfterEndEnvironment{remark}{\everypar{\setbox\z@\lastbox\everypar{}}}
\AfterEndEnvironment{proposition}{\everypar{\setbox\z@\lastbox\everypar{}}}
\AfterEndEnvironment{corollary}{\everypar{\setbox\z@\lastbox\everypar{}}}
\AfterEndEnvironment{fact}{\everypar{\setbox\z@\lastbox\everypar{}}}
\makeatother

\newenvironment{remark}{\noindent {\bf Remark}.}{\par\smallskip\par}

\newcommand{\Nat}{\mathbb{N}}

\newcommand{\prob}{\mathbb{P}}
\newcommand{\ex}[1]{\mathbb{E}\left[#1\right]}

\newcommand{\bigO}{{\mathcal O}}

\newcommand{\Gnh}{{G_{n, 1/2}}}

\newtheorem{firsttheorem}{Proposition}
\newtheorem{fact}[firsttheorem]{Fact}
\newtheorem{theorem}[firsttheorem]{Theorem}
\newtheorem{lemma}[firsttheorem]{Lemma}
\newtheorem{corollary}[firsttheorem]{Corollary}

\newtheorem{proposition}[firsttheorem]{Proposition}

\newtheorem{definition}[firsttheorem]{Definition}

\newtheorem*{lemma*}{Lemma}

\numberwithin{equation}{section}
\numberwithin{firsttheorem}{section}

\newcommand{\roundUp}[1]{\left\lceil#1\right\rceil}
\newcommand{\roundDown}[1]{\left\lfloor#1\right\rfloor}

\newcommand{\Gnm}{G_{n,m}}
\newcommand{\Gnp}{G_{n,p}}
\usepackage{marginnote}

\usepackage{bbm}
\usepackage[numbers]{natbib}

\newcommand{\kp}{\mathbf{k}}
\newcommand{\bfk}{\mathbf{k}}

\newcommand{\bfr}{\mathbf{r}}
\newcommand{\bfl}{{\boldsymbol{\ell}}}
\newcommand{\bft}{\mathbf{t}}
\newcommand{\mc}{\mathcal}
\newcommand{\parenth}[1]{\left(#1\right)}
\newcommand{\ceilnk}{{\left\lceil \frac{n}{k}\right\rceil}}
\newcommand{\floornk}{{\left\lfloor \frac{n}{k}\right\rfloor}}

\newcommand{\boldkappa}{{\boldsymbol{\kappa}}}
\newcommand{\boldlambda}{{\boldsymbol{\lambda}}}

\newcommand{\Pb}{{\mathbb{P}}}

\newcommand{\E}{\mathbb{E}}

\newcommand{\kS}{{k_\text{S}}}
\newcommand{\kL}{{k_\text{L}}}
\newcommand{\lS}{{\ell_\text{S}}}
\newcommand{\lL}{{\ell_\text{L}}}
\newcommand{\lambdaS}{{\lambda_\text{S}}}
\newcommand{\lambdaL}{{\lambda_\text{L}}}

\newcommand{\gid}{{g_\mathrm{id}}}
\newcommand{\nid}{{n_\mathrm{id}}}
\newcommand{\gt}{{g_\mathrm{tr}}}
\newcommand{\ntr}{{n_\mathrm{tr}}}

\newcommand{\MA}{{M_2}}
\newcommand{\MB}{{M_1}}

\newcommand{\boldk}{{\boldsymbol{k}}}

\renewcommand{\le}{\leqslant}
\renewcommand{\ge}{\geqslant}

\renewcommand{\epsilon}{\varepsilon}

\usepackage{xcolor}

\newcommand{\R}{\mathbb{R}}

\newcommand{\bp}{\mathbf{p}}
\newcommand{\tP}{\widetilde{P}}
\newcommand{\tL}{\widetilde{L}}

\begin{document}

\pagenumbering{roman}

\title{Colouring random graphs: Tame colourings}

\author{
Annika Heckel\thanks{Matematiska institutionen, Uppsala universitet, Box 480, 751 06 Uppsala, Sweden. Email: \texttt{annika.heckel@math.uu.se}. Funded by the Swedish Research Council, Starting Grant 2022-02829, and the European Research Council, ERC Grant Agreements 772606-PTRCSP and 676632-RanDM.}
\and Konstantinos Panagiotou\thanks{Mathematisches Institut, Ludwig-Maximilians-Universität München, Theresienstr.~39, 80333 München, Germany. Email: \texttt{kpanagio@math.lmu.de}. Funded by the European Research Council, ERC Grant Agreement 772606-PTRCSP.} 
}

\date{\today}
\maketitle
	
\begin{abstract}
Given a graph $G$, a colouring is an assignment of colours to the vertices of $G$ so that no two
adjacent vertices are coloured the same.
If all colour classes have size at most~$t$, then we call the colouring \emph{$t$-bounded}, and the \emph{$t$-bounded chromatic number of $G$}, denoted by $\chi_t(G)$, is the minimum number of colours in such a colouring. Every colouring of $G$ is then $\alpha(G)$-bounded, where $\alpha(G)$ denotes the size of a largest independent set in $G$, and we denote by $\chi(G) = \chi_{\alpha(G)}(G)$ the \emph{chromatic number} of $G$.

We study colourings of the binomial random graph $G_{n,1/2}$  
and of the corresponding uniform random graph $G_{n,m}$ with $m=\left \lfloor \frac 12 \binom{n}{2} \right \rfloor$ edges. We show that for $t = \alpha(G_{n,m})-2$, $\chi_t(G_{n,m})$ is maximally concentrated on at most two explicit consecutive values.
This behaviour stands in stark contrast to that of the normal chromatic number, 
which was recently shown by the first author and Oliver Riordan~\cite{HRHowdoes} not to be concentrated on any sequence of intervals of length $n^{1/2-o(1)}$. 
Moreover, when $t = \alpha(G_{n, 1/2})-1$ and if the expected number of independent sets of size $t$ is not too small, we determine an explicit interval of length $n^{0.99}$ that contains $\chi_t(G_{n,1/2})$ with high probability.
Both results have profound consequences: the former is at the core of the intriguing Zigzag Conjecture on the distribution of $\chi(G_{n, 1/2})$ and justifies one of its main hypotheses, while the latter is an important ingredient in the proof of a non-concentration result in~\cite{HRHowdoes} for $\chi(G_{n,1/2})$ that is conjectured to be optimal.

Our two results are consequences of a more general statement. We consider a specific class of colourings that we call \emph{tame} and which fulfil some natural, albeit technical, conditions. We provide tight bounds for the probability of existence of such colourings via a delicate second moment argument, and then we apply these bounds to the case of $t$-bounded colourings, where $t\in\{\alpha(G)-1,\alpha(G)-2\}$. As a further consequence of our main result, we prove two-point concentration of the equitable chromatic number of~$G_{n,m}$.
\end{abstract}

{
\small
\tableofcontents
}

\pagebreak

\pagenumbering{arabic}

\section{Introduction}
Given $n, m \in \Nat$ and $p \in [0,1]$, the \emph{binomial random graph} $\Gnp$ is the graph on $n$ labelled vertices where each possible edge is included independently with probability $p$. The \emph{uniform random graph} $G_{n,m}$ is the graph on $n$ labelled vertices with exactly $m$ edges chosen uniformly at random from all possible edge sets of size $m$. In this paper we consider \emph{colourings} of random graphs, where a colouring of a graph $G$ is an assignment of colours to the vertices of $G$ so that no two neighbouring vertices are coloured the same. The smallest number of colours for which this is possible is called the \emph{chromatic number} of $G$, and is denoted by $\chi(G)$. In the rest of the paper we will abbreviate without further reference
\[
    N = \binom{n}{2}, \quad q=1-p \quad\text{and}\quad b=1/q.
\]
We denote by $\Pb_p$, $\E_p$, $\Pb_m$, $\E_m$ probability and expectation in $\Gnp$ and $\Gnm$, respectively; if we simply write $\Pb$ and $\E$ this refers to $\Gnm$.
Moreover, we will be interested in asymptotic properties of the random graph, that is, when $n\to\infty$. As usual in this context, for a sequence of discrete probability spaces $(\Omega_n, \Pb_n)_{n\in \Nat}$ we say that a sequence $(E_n)_{n \in \Nat}$ of events holds \emph{with high probability (whp)} if $\Pb_n(E_n) \rightarrow 1$ as $n \rightarrow \infty$. For two functions $f(x), g(x)$, we write $f(x) \lesssim g(x)$ if $f(x) \le (1+o(1)) g(x)$ as $x \rightarrow \infty$.

\paragraph{Some History.} In one of the seminal papers initiating the study of random graphs~\cite{erdos1960evolution}, Erd\H{o}s and R\'enyi raised the question of the chromatic number of a random graph. While originally the focus was on sparse random graphs $G_{n,p}$, where $p$ is of order~$n^{-1}$ and thus the graph has bounded average degree, attention soon also turned to dense random graphs where~$p$ is constant, and in particular to the special case~$p= 1/2$ corresponding to the uniform distribution on all graphs on $n$ vertices.

In 1975, Grimmett and McDiarmid~\cite{grimmett1975colouring} determined the likely order of magnitude of~$\chi(\Gnp)$ when $p$ is constant, 
showing that whp~$\chi(G_{n, p})= \Theta \left( {n}/{\ln n} \right)$. 
In a landmark contribution in 1987~\cite{bollobas1988chromatic}, Bollob\'as pinned down the asymptotic value of the chromatic number when $p$ is constant, showing that, whp, 
\begin{equation}\label{eq:bollobas87}
 \chi(G_{n,p}) \sim \frac{n}{2 \log_b n} \enspace .
\end{equation}
This result was sharpened several times \cite{mcdiarmid1989method,panagiotou2009note,fountoulakis2010t}. The currently sharpest bounds were given by the first author in \cite{heckel2018chromatic}, showing that if $0<p\le1-1/e^2$ is constant, whp
\begin{equation}\label{eq:previousbounds}
  \chi(G_{n,p}) = \frac{n}{2 \log_b n - 2 \log_b \log_b n - 2} + o \left(\frac{n}{\log^2 n} \right),
\end{equation}
with a similar but slightly more complicated expression given for the case $1-1/e^2<p<1$.

The case  $p \rightarrow 0$ has also been studied. In particular, if $p \gg 1/n$, then results of {\L}uczak~\cite{luczak1991chromatic} imply that whp
\begin{equation*}
    \chi(G_{n,p}) \sim \frac{np}{2 \ln(np)}.
\end{equation*}
It was noted early on that the chromatic number of the random graph $G_{n,p}$ is sharply \emph{concentrated}. In 1987, Shamir and Spencer \cite{shamir1987sharp} showed that for any function $p=p(n)$, whp $\chi(\Gnp)$ is contained in an interval of length about $\sqrt{n}$. Alon \cite{alonspencer, scott2008concentration}
improved this to about $\sqrt{n}/\ln n$ when $p > 0$ is constant.
Later it was observed that if $p\to 0$ quickly enough, then this can be improved dramatically. Alon and Krivelevich \cite{alon1997concentration} proved that if $p<n^{-1/2-\epsilon}$ for some $\epsilon > 0$, $\chi(G_{n,p})$ takes one of at most \emph{two} consecutive values whp. In other words, in this range $\chi(G_{n,p})$ behaves almost deterministically. However, neither of these results gives any clue about the \emph{location} of the concentration intervals.  Achlioptas and Naor~\cite{achlioptas2005two} determined two \emph{explicit} values for $\chi(G_{n,p})$ when $p=d/n$ and where $d$ is constant, and Coja-Oghlan, Steger and the second author~\cite{coja2008chromatic} extended this to three explicit values when $p< n^{-3/4 - \epsilon}$.

From the late 1980s, Bollob\'as raised, and he and Erd\H{o}s popularised, the opposite question: are there any non-trivial examples where $\chi(G_{n,p})$ is \emph{not} very narrowly concentrated? In the appendix of the first edition of the standard textbook on the probabilistic method \cite{alonspencerfirstedition}, Erd\H{o}s asked for a proof that $\chi(G_{n, 1/2})$ is not whp contained in any sequence of intervals of constant length. Bollob\'as highlighted the problem \cite{bollobas:concentrationfixed}, asking for any non-trivial results showing a lack of concentration, and, more generally, for the correct concentration interval length of $\chi(G_{n,p})$ as a function of~$n$ and $p$. In 2019, the first author \cite{heckel2019nonconcentration} gave the first result of this type: it turns out that, at least for some subsequence of the natural numbers, $\chi(G_{n, 1/2})$ is not whp contained in any sequence of intervals of length $n^{1/4 - \epsilon}$, for any fixed $\epsilon >0$. Together with Oliver Riordan~\cite{HRHowdoes}, this was improved to intervals of length $n^{1/2 - \epsilon}$, essentially matching Shamir and Spencer's aforementioned  upper bound on the concentration interval length.

Notably, the results in \cite{heckel2019nonconcentration,HRHowdoes} give no lower bound on the concentration interval length for any particular $n$, let alone all $n$. The method that is developed `only' asserts that for every $n$, there must be  some nearby integer $n^* \sim n$ such that $\chi(G_{n^*, 1/2})$ is not too narrowly concentrated. Furthermore, for most $n$, the lower bound on the concentration interval length at $n^*$ from the proof in \cite{HRHowdoes} is actually far away from the best upper bound of $\sqrt{n}/{\ln n}$: depending on $n$, the lower bound varies between $n^{o(1)}$ and $n^{1/2 -o(1)}$. It therefore remains a challenging widely open and timely problem to establish the correct concentration behaviour --- or, ultimately, the limiting distribution --- of $\chi(G_{n, p})$ for $p=1/2$, and more generally, for all constant $0 < p < 1$.

\paragraph{The Zigzag Conjecture.} 
An intriguing conjecture about the concentration interval length of $\chi(G_{n, 1/2})$ was made by B\'ela Bollob\'as, Rob Morris, Oliver Riordan, Paul Smith as well as the present authors. 
The conjecture, on which we will elaborate more below, proposes that the concentration interval length behaves in a rather unusual way: roughly speaking, it is conjectured to be a function of $n$ that `zigzags' between $n^{1/4+o(1)}$ and $n^{1/2 + o(1)}$.   
We refer to~\cite{HRHowdoes} for a detailed exposition and many further related results, observations and conjectures.

Before we continue, let us quickly review some background.
The chromatic number of any graph $G$ on $n$ vertices is closely related to the distribution of \emph{independent sets} in~$G$.  We call an independent set of size $t$ a \emph{$t$-set}, and denote by $N_t$ the number of $t$-sets in $\Gnp$. Let
\begin{equation}\label{eq:expNt}
    \mu_t = \E_p [N_t] = \binom{n}{t} q^{\binom{t}{2}}
\end{equation}
be the expected number of $t$-sets in $\Gnp$, and set
\begin{equation}
\label{eq:alpha0}
    \alpha_0 = \alpha_0(n) := 2 \log_b n -2 \log_b\log_b n +2 \log_b(e/2)+1,
    \qquad \alpha = \alpha(n) := \left\lfloor \alpha_0(n)\right \rfloor.
\end{equation}
We denote by $\alpha(G)$ the \emph{independence number} of a graph $G$, which is the size of largest independent set. 
Bollob\'as and Erd\H os~\cite{erdoscliques} proved that if $p$ is constant, then whp
\begin{equation}
\label{eq:alphaGnp}
    \alpha(G_{n, p}) = \left\lfloor \alpha_0(n) +o(1)\right \rfloor,    
\end{equation}
pinning down $\alpha(G_{n,  p})$ to at most two consecutive values. 
In fact, even more is true: 
for almost all $n$, whp $\alpha(G_{n, p}) = \alpha$, that is, except for a set of density zero in the natural numbers, the independence number is known exactly whp.

The chromatic number is closely related to the independence number, since in any colouring all colour classes are independent sets. In particular, for any graph $G$ on $n$ vertices, $\chi(G) \ge n / \alpha(G)$.  This simple lower bound is asymptotically tight for random graphs: the result~\eqref{eq:bollobas87} of Bollob\'as implies that whp $\chi(G_{n, p}) \sim n / \alpha(G_{n,p})$ for constant $p$. Even more can be said: by looking closely at~\eqref{eq:previousbounds} and~\eqref{eq:alphaGnp}, we see that actually the {average} size of a colour class in an optimal colouring differs from $\alpha(G_{n,p})$ by at most a mere constant.

With this preparation at hand we can delve into the Zigzag Conjecture. Our primary goal is to understand the distribution of $\chi(G_{n,p})$. One immediate and promising approach to do so is to enumerate colourings of $G_{n,p}$ with a specific number $k = k(n) \sim n/2\log_b n$ of colours; let's call this (random) number $X_k$. If the expectation $\mathbb{E}_p[X_k]$ vanishes, then whp $X_k = 0$ and so $\chi(G_{n,p}) > k$. On the other hand, if $\mathbb{E}_p[X_k]$ is large and $X_k$ is concentrated around its expectation, then whp $X_k > 0$ and so $\chi(G_{n,p}) \le k$. 

Note that in the latter case, where $\mathbb{E}_p[X_k]$ is large, we actually have some wiggle room. Indeed, we may replace $X_k$ by some other quantity $X'_k$ such that $\mathbb{E}[X'_k]$ is also large and $X'_k > 0$ implies that $X_k > 0$, that is, there is a colouring with $k$ colours. This may sound innocent, but it allows us to switch from a random variable (in this case $X_k$) with  concentration properties that are potentially difficult to handle or even non-existent to another random variable ($X'_k$), where the distribution is easier to study or that has better properties. In particular, instead of enumerating all colourings, we may restrict to a subset of them that possess a specific \emph{profile} $\mathbf{k} = (k_u)_{1 \le u \le \alpha}$: a profile fixes the number $k_u$ of colour classes with~$u$ vertices, for all $1 \le u \le \alpha$. Obviously $\mathbf{k}$ should be chosen in a way so that the expected number of such colourings is as large as possible. Of course at this point, it is not obvious at all why, in expectation, there is a plethora of such colourings when $\mathbb{E}[X_k]$ is large, but we may expect that typical colourings of $G_{n,p}$ have a typical `shape', in the sense that their profile is close to the aforementioned maximizing profile. So let us write $X'_k$ for the number of colourings with profile $\mathbf{k}'$, where $\mathbf{k}'$ maximizes, among all profiles, the expected value.

The crucial question is about concentration properties of $X'_k$, and there is something that we can immediately observe. As it turns out and as we will make more precise later, the profile $\mathbf{k}'$ is such that $k_u$ is proportional to $k$ for all $u$ close to $\alpha$, that is, the number of colour classes of each of the sizes $\alpha-i$, where $i = \bigO(1)$, is linear in the total number $k$ of colours.  
However, it is well-known that for most $n$, the number $N_\alpha$ of largest independent sets is $o(k)$, see also Section~\ref{ssec:isets} below.
In other words, $\mathbb{E}[X'_k]$ is dominated by an atypical event regarding~$N_\alpha$ and thus $X'_k$ \emph{is not} concentrated (actually, $X'_k = 0$ whp).
There is also another issue.
The number $N_{\alpha-1}$ is (much) larger than the number $k$ of colours, but its standard deviation is~$o(k)$. So requiring that $G_{n,1/2}$ has a specific colouring with profile $\mathbf{k}'$ alters the distribution of the graph significantly, since it then has an atypically large number of $(\alpha-1)$-sets.
This provides a second reason why $X'_k$ does not concentrate. On the positive side, there are no such issues with the number of $t$-sets where $t \le \alpha-2$, since they are so numerous and fluctuate so much that `planting' $\Theta(k)$ of them \emph{should not really} alter the distribution of~$G_{n,1/2}$.

The aforementioned observation is at the core of the Zigzag Conjecture, which postulates that the variations in $N_t$, $t \le \alpha-2$, play no role in a very strong sense: as soon as we have decided on the part of a colouring that fixes the colour classes of size $\alpha$ and $\alpha-1$,
$\chi(G_{n,1/2})$ should concentrate extremely sharply on \emph{one} point only, at least for most $n$, resembling the behaviour of the independence number.  Based on this assumption, the conjecture formulates the claimed precise dependency of $\chi(G_{n,1/2})$ on the variations in $N_\alpha$ and $N_{\alpha-1}$, see~\cite{HRHowdoes}.

\paragraph{Our Results.}
In this paper we take the first step towards confirming the Zigzag Conjecture, and towards ultimately determining not just the concentration interval length but also the limiting distribution of the chromatic number. One consequence of our main result is Theorem~\ref{twopointrestricted} below, which suggests rather strongly that colour classes of size $\alpha$ and $\alpha-1$ are indeed the only potential sources of non-concentration of $\chi(G_{n,m})$. Roughly speaking, we show that if we colour the random graph and only allow colour classes of size $ \le \alpha-2$, then the required number of colours is maximally concentrated --- it takes one of at most two consecutive values whp. 

Let us introduce some terminology. We say that a colouring is \emph{$t$-bounded} if no colour class is larger than $t$. So a colouring of a graph $G$ is always $\alpha(G)$-bounded. We call the minimum number of colours for which a $t$-bounded colouring exists the \emph{$t$-bounded chromatic number} of~$G$, and denote it by $\chi_t (G)$.
\begin{theorem}
\label{twopointrestricted}
Let $p=1/2$, $m= \left \lfloor pN \right \rfloor$ and $a=a(n)=\alpha(n)-2$. 
Then there is a function $k=k(n)$ such that, whp,
\[
    \chi_{a}(G_{n,m})\in \{k, k+1\}.
\]
\end{theorem}
Actually, we can be a bit more precise than that. Let $E_{n,k,t}$ denote the expected number of unordered $t$-bounded $k$-colourings of $G_{n, 1/2}$, where `unordered' refers to the fact that we count colourings up to permutations of the colours; see also the discussion after Lemma~\ref{lem:expxk}.
Then the \emph{$t$-bounded first moment threshold} is defined as
\begin{equation}\label{ktdef}
    \boldk_t(n) := \min\{k: E_{n,k,t} \ge 1 \}.
\end{equation}
Note that $E_{n,k,t}$ is increasing in $k$ for $k\le n$, so that $\boldk_t(n)$ is well-defined. With this notation at hand, we could obtain an even more precise statement: we show in the proof of Theorem~\ref{twopointrestricted} in Section~\ref{sec:proofsmainthms} that we can actually pick 
\[
    k(n) \in \{\boldk_{\alpha-2}(n)-1, \boldk_{\alpha-2}(n) \}.
\]
Let us remark briefly on the use of $G_{n,m}$ in Theorem~\ref{twopointrestricted}, rather than $G_{n, 1/2}$ which we discussed earlier. The reason for switching to $G_{n,m}$ is that the delicate second moment argument in the forthcoming proof only applies to $G_{n,m}$. As pointed out by Alex Scott (see the remarks at the end of \cite{heckel2019nonconcentration}), there is a simple coupling of the random graphs $G_{n, 1/2}$ and $G_{n,m}$ with $m=\left \lfloor N/2 \right \rfloor$ so that whp their chromatic numbers (or $(\alpha-2)$-bounded chromatic numbers) differ by at most $\omega(n) \ln n$, where $\omega(n) \rightarrow \infty$ is arbitrary. Thus it is almost equivalent whether to study $G_{n, 1/2}$ or $G_{n,m}$ in the context of the concentration or limiting distribution of the chromatic number. Theorem~\ref{twopointrestricted} implies that $\chi_{\alpha-2} (G_{n,1/2})$ is concentrated on $\omega(n) \ln n$ values, although we believe that a two-point concentration result could likely be achieved by allowing $m$ in Theorem~\ref{twopointrestricted} to vary slightly around $N/2$.

Theorem~\ref{twopointrestricted} is a consequence of our main result Theorem~\ref{theorem:general}, which is more general. There we establish a lower bound for the probability of existence for a large class of possible colourings that we call \emph{tame colourings}. We postpone the precise formulation of the setup to Section~\ref{sec:tamegeneral}. Before we come to that, we give two further applications. The second application is Theorem~\ref{theorem:announcedbounds} below, which roughly states that the first moment threshold is a good approximation for the typical values of $\chi_a(G_{n,1/2})$, provided that the expected number $\mu_a$ of $a$-sets, see~\eqref{eq:expNt}, lies within certain bounds. Setting $a=\alpha-1$ and using the fact that (for most $n$) whp $\chi(G_{n, 1/2})$ and $\chi_{\alpha-1}(G_{n, 1/2})$ differ by less than about $\mu_\alpha$ (see Lemma 44 in \cite{HRHowdoes}), this result can be used to improve the previously best explicit bounds \eqref{eq:previousbounds} for $\chi(G_{n, 1/2})$, which had pinned down its value up to an error term of size $o(n/\ln^2 n)$.

Assuming (a special case of) Theorem~\ref{theorem:announcedbounds}, it was shown  in~\cite{HRHowdoes}  that~$\chi(G_{n, 1/2})$ is not concentrated on any sequence of intervals of length ${\sqrt{n} \ln \ln n}/{\ln^3 n}$. This lower bound comes within a power of $\ln n$ of Alon's aforementioned upper concentration bound of ${\sqrt{n}}/{\ln n}$. 
Even more significantly, in~\cite{HRHowdoes} it is argued heuristically why this lower concentration bound should be best possible; see the finer conjectures in \S1.3.2 of~\cite{HRHowdoes}. With Theorem~\ref{theorem:announcedbounds} below, the proof of this lower concentration bound is now complete.
\begin{theorem} \label{theorem:announcedbounds}
Let $G \sim G_{n, 1/2}$, and let $a=a(n)$ be a sequence of integers such that
\[
    n^{1.1} < \mu_a < n^{2.9}.
\]
Then, whp,
\[
    \chi_{a} (G) = \boldk_{ a} + \bigO( n^{0.99}).
\] 
\end{theorem}

Let us remark briefly on the bounds on $\mu_a$ above. It is well-known that they imply in particular that $a \in \{\alpha(n)-1, \alpha(n)-2\}$, see also Lemma~\ref{prop:basicboundsisets} below. The upper bound~$n^{2.9}$ has no significance beyond ensuring $a \ge \alpha-2$ and could be relaxed to $n^{3-\epsilon}$ for any $\varepsilon > 0$. However, the lower bound $n^{1.1}$ \emph{cannot} be pushed towards $n^{1+\epsilon}$: curiously, the behaviour of the optimal $a$-bounded colouring profile changes significantly when $\mu_a \le n^{1+x_0}$ for some small constant $x_0 \approx 0.02905$, and our proof does not apply in this case. We elaborate more on this in Section~\ref{section:partialcolourings} and after Lemma~\ref{lemma:kstartame}.

The last application of our main result concerns \emph{equitable} colourings. We say that a colouring is equitable if the colour class sizes differ by at most $1$. Let $\chi_=(G)$ denote the \emph{equitable chromatic number} of a graph $G$, that is, the smallest $k$ such that there exists an equitable $k$-colouring. The famous  Hajnal-Szemer{\'e}di Theorem \cite{hajnal1970proof} states that if $G$ has maximum degree $\Delta(G)$, then $\chi_=^*(G) \le \Delta(G)+1$. For random graphs, Krivelevich and Patk\'os~\cite{krivelevich2009equitable} showed that if $n^{-1/5+\epsilon} \le p \le 0.99$, whp $\chi_= (G_{n,p}) \sim \chi(G_{n,p})$. In \cite{heckel:equitable}, the first author proved that for constant $p <1-1/e^2$, there is a subsequence of the integers where $\chi_=(G_{n,m})$ with $m=\left \lfloor pN \right \rfloor$ is maximally concentrated on only {one} value. We extend this concentration result to all integers $n$ (and two consecutive values) when $p<1-1/e \approx 0.63$.
\begin{theorem}
\label{twopointequitable}
Let $0 < p <1-1/e$ be fixed and $m= \left \lfloor pN \right \rfloor$. 
Then there is a sequence $k=k(n)$ such that, whp,
\[
    \chi_=(\Gnm)\in \{k, k+1\}.
\]
\end{theorem}

\paragraph{Outline.} The paper is roughly structured as follows. In Section~\ref{sec:tamegeneral} we present our main result, Theorem~\ref{theorem:general} and set up the appropriate framework. Section~\ref{section:preliminaries} is a collection of auxiliary technical statements and known results that we shall use repeatedly. In the subsequent Section~\ref{section:proofoverview} we guide through the proof of our main result and, in particular, define the random variable $Z_\bfk$ that we will study. Section~\ref{section:firstmoment} is devoted to the analysis of the first moment of $Z_\bfk$, while in Section~\ref{section:secondmoment}, the technically most elaborate part of the paper, we determine asymptotics for the second moment of $Z_\bfk$. Optimal profiles in the case of $t$-bounded colourings for $t\in\{\alpha-1, \alpha-2\}$ and $p=1/2$ are characterised in Section~\ref{section:optimalprofiles}, where we also establish their tameness. Finally, in the last regular section we prove Theorems~\ref{twopointrestricted}, \ref{theorem:announcedbounds} and~\ref{twopointequitable}. All proofs omitted in the main part can be found in the appendix.

\section{Tame colourings and the general result}
 \label{sec:tamegeneral}
		
\subsection{Colouring profiles}	
We will view a colouring as an ordered partition of $V=[n]$ into independent sets: let
\[
    \Pi=(V_1, \dots, V_k)
\]
be an ordered partition of the vertex set $V$ into non-empty parts. If all $V_i$ are independent, then we call $\Pi$ a \emph{colouring}. By a \emph{partial ordered partition}, we mean an ordered collection
\[
    \Pi=(V_1, \dots, V_\ell)
\]
of disjoint sets. If all $V_i$ are independent, then we say that $\Pi$ is a \emph{partial colouring} of $G$. The sets $V_i $ are called the \emph{colour classes} of the (complete or partial) colouring~$\Pi$. In the following definition we put some structure on the set of partitions, describing them in terms of how many parts of any given size they contain.
\begin{definition}
Let $n$, $k$ and $t$ be integers.
\begin{enumerate}[a)]
\itemsep2pt 
    \item A \emph{$k$-colouring profile} is a sequence $\bfk=(k_u)_{1 \le  u \le n}$ such that 
    \[
        \sum_{1\le u \le n} k_u = k \quad \text{and} \quad \sum_{1\le u \le n} uk_u \le n.
    \]
    If $\sum_{1\le u \le n} uk_u = n$, we call $\bfk$ a \emph{complete (colouring) profile}, otherwise a \emph{partial profile}.
    \item A colouring profile $\bfk$ is called \emph{$t$-bounded} if $k_u = 0$ for all $u >t$.
    \item A (complete or partial) ordered partition $\Pi=(V_1, \dots, V_k)$ of $V$ \emph{has profile $\bfk$} if it contains exactly $k_u$ parts of size $u$ for all $1\le u \le n$, and the parts are decreasing in size, that is, $|V_1| \ge |V_2| \ge \dots \ge |V_k|$. Moreover, a colouring with colours in $\{1, \dots, k\}$ has profile $\bfk$ if the induced ordered partition given by its colour classes has profile $\bfk$.
    \item Given a colouring profile $\bfk$, let $ \kappa_u$ be the \emph{fraction of vertices} in sets of size $u$, that is,
    \begin{equation}
        \label{eq:kappak}
        \kappa_u = \frac{ uk_u}{n}, \,\,
	\quad
	\text{and set}
	\quad
	\boldkappa=(\kappa_u)_{1 \le u \le n}. 
    \end{equation}
    Let $\kappa = \sum_{1\le u \le n} {\kappa_u} \le 1$. Note that $\kappa=1$ if and only if $\bfk$ is a complete colouring profile. 
    \item Given a colouring profile $\bfk$ or equivalently $\boldkappa$, let 
    \[
        X_\mathbf{k} \text{ or } X_\boldkappa
    \]
    count the number of valid colourings in the random graph $G$ with profile $\textbf{k}$, and let
    \[
        \bar{X}_\bfk =\bar{X}_\boldkappa= \frac{X_\mathbf{k}}{\prod_{1\le u \le n} k_u!}
    \]
    be the number of \emph{unordered} colourings with profile $\textbf{k}$.
    \item Given $\boldkappa = (\kappa_u)_{1\le u\le n}$ and $\boldsymbol{\lambda}=(\lambda_u)_{1\le u\le n}$, we write $\boldsymbol{\lambda} \le \boldsymbol{\kappa}$ if $\lambda_u \le \kappa_u$ for all $u$.
    \end{enumerate}
\end{definition}
	
We will use both $\bfk$ and $\boldkappa$ to refer to colouring profiles, and specifying one of these sequences defines the other via~\eqref{eq:kappak}. We will often use the letters $i,j$ to index colour classes $V_i$, $V_j$, and the letters $u, v$ often refer to colour class sizes, which may in principle range from $1$ to $n$, but are usually $t$-bounded for some $t = \alpha + \bigO(1)$. To keep notation compact, we will often just write $\sum_u$, $\prod_u$, $\sum_i$ and so on when the range of $u$, $i$ is clear.
	
We will \emph{tend} to use $k$ for the number of colours in a complete colouring, and $\ell$ for the number of colours in a partial colouring. The number of colour classes of size $u$ will then be referred to as $k_u$ or $\ell_u$, and it will often be the case that $\ell_u \le k_u$ and that $k_u$ is the number of colour classes of size $u$ in a complete $k$-colouring, and $\ell_u$ is the number of colour classes of size $u$ in a partial $\ell$-colouring. Usually $\kappa_u = k_u u / n$, $\kappa = \sum_u \kappa_u$, $\lambda_u = \ell_u u / n$, $\lambda = \sum_u \lambda_u$.

From the definitions we readily obtain the following simple lemma about the expected number of colourings with a given profile.
\begin{lemma}
\label{lem:expxk}
Let $n,k\in\mathbb{N}$ and let $\bfk = (k_u)_{1\le u \le n}$ be a $k$-colouring profile. Then
\[
    \mathbb{E}_p[X_\bfk] = P_\bfk \, q^{f_\bfk}, ~ p\in[0,1]
    \quad
    \text{and}
    \quad
    \mathbb{E}_m[X_\bfk] = P_\bfk \, \frac{{N-f_\bfk \choose m}}{{N \choose m}}, ~m \in\mathbb{N},
\]
where
\begin{equation*}
\label{defP} 
    P_\kp = \frac{n!}{\prod_{1 \le u \le n} u!^{k_u}}
    \quad
    \text{and}
    \quad
    f_\mathbf{k}= \sum_{1 \le u \le n} {u \choose 2} k_u.    
\end{equation*}
\end{lemma}
\begin{proof}
The number of ways to choose an ordered vertex partition of profile $\bfk$ equals $P_\bfk$. Moreover, once we have chosen such a partition, it is a colouring if and only if none of the edges within the parts appears in the graph; the number of such forbidden edges is 
$f_\bfk$.
\end{proof}
Let $p \in [0,1]$ and $k,t\in \mathbb{N}$. 
With Lemma~\ref{lem:expxk} at hand we can also determine the expected number of unordered $t$-bounded $k$-colourings $E_{n,k,t,p}$ of $\Gnp$, and we abbreviate, as before, $E_{n,k,t} = E_{n,k,t,1/2}$. Indeed, if we denote by $P_{n,k,t}$ the set of all $t$-bounded $k$-colouring profiles on $n$ vertices, then
\begin{equation*}
    E_{n,k,t,p} = \sum_{\bfk \in P_{n,k,t}} \mathbb{E}_p[ \bar X_\bfk]
    = \sum_{\bfk \in P_{n,k,t}} \frac{\mathbb{E}_p[X_\bfk]}{\prod_{1 \le u \le t} k_u!}
    = \sum_{\bfk \in P_{n,k,t}} \frac{P_{\bfk} \, q^{f_\bfk}}{\prod_{1 \le u \le t} k_u!}.
\end{equation*}
In \S\ref{section:optimalprofiles} we will study $E_{n,k,t}$ and in particular the `optimal' $t$-bounded colouring profile that maximizes the (unordered) expectation in very great detail.

\subsection{Tame colourings}
\label{section:tamecolourings}
	
We specify some technical smoothness properties of colouring profiles. These properties (among others) will occur naturally in the `optimal' colouring profiles under consideration later, and so they impose no real restriction. 
\begin{definition}\label{deftame}
Let $0<p<1$ be fixed. Set $m = \left \lfloor pN \right \rfloor$, and let $a = a(n) =\alpha_0(n) - \bigO(1)$ with $a \le \alpha$ be a sequence of integers, and $\mathbf{k}=\mathbf{k}(n)$ a sequence of complete $a$-bounded $k$-colouring profiles. Then $\bfk$ is called \emph{tame}  if there is a constant $ c\in(0,1)$ and an increasing function $ \gamma: \mathbb{N}_0\rightarrow \mathbb{R}$ with  $\gamma(x)\rightarrow \infty$ as $x \rightarrow \infty$ 
  so that both of the following hold.
    \begin{enumerate}[a)]
    \itemsep2pt 
    \label{tame:a0}
    \item $\kappa_u = {u k_u}/{n} < b^{-(\alpha-u)\gamma(\alpha-u)}$ for all $1 \le u \le a$ and sufficiently large $n$, and
    \label{tame:tail}
    \item \label{largeexp} $ \ln \E_m \big[\bar{X}_\kp\big] \gg -n^{1-c}$.
    \end{enumerate}
\end{definition}
Some comments on the conditions are in place. Condition \ref{tame:tail}) on the tail of the $\kappa_u$'s is a convenient assumption that facilitates our computations; we believe that our results are true under weaker conditions (perhaps only requiring a `second moment'-type condition such as $\sum_u(a-u)^2 \kappa_u$ being bounded), but the case treated here is sufficient for all intended applications. 
In particular, we will see that the relevant optimal $(\alpha-1)$- or $(\alpha-2)$-bounded colourings of $G_{n,m}$ have the property that $\kappa_u \approx b^{-(a-u)^2/2}$, so \ref{tame:tail}) will be satisfied with plenty of room to spare.

Condition \ref{largeexp}) requires that the expected number of unordered colourings is \emph{not too small}. Obviously, under such a condition we cannot in general establish that $X_\kp>0$ whp, and our main result Theorem~\ref{theorem:general} actually makes the stronger assumption that $\ln \E_m \big[\bar{X}_\kp\big] \gg \ln n$, i.e., the expected number of colourings is at least \emph{superpolynomial}. However, assuming just \ref{largeexp}) enables us to formulate intermediate proof steps and results in a more general form that may facilitate future use.

We will require a similar condition for the expected number of certain partial colourings in Theorem~\ref{theorem:general} -- see \eqref{eq:lowerboundbeta}. It is easy to see that some condition of this kind is also necessary so that $X_\bfk > 0$ whp: it is possible to construct profiles $\bfk$ so that the expectation of $\bar{X}_{\bfk}$ is large, but the expectation of $\bar{X}_\bfl$ tends to $0$ for some $\bfl \le \bfk$. (For example, we may boost $\E[\bar X_\bfk]$ by allowing for an atypically large number of $\alpha$-sets, as we described in the introduction.) 
But then whp $\bar{X}_{\bfl} = 0$, which implies that also whp $X_{\bfk}=0$, since any colouring with profile $\bfk$ contains a partial colouring with profile $\bfl$. Thus, a lower bound on $\E_p[\bar X_{\bfl}]$ for all ${\bfl} \le \bfk$ is also necessary, which we will stipulate in \eqref{eq:lowerboundbeta}, but to keep our second moment calculations in as general a setting as possible, in the definition of a tame colouring we do not require it yet.

Let us collect some quick properties of tame profiles $\bfk$ that we will use several times.
\begin{proposition}
\label{prop:proptame}
Let $\bfk$ be tame. Then there is a function $u^* \sim \alpha \sim 2\log_bn$ so that $k_u = 0$ for all $u < u^*$. Moreover, for any $s \in \mathbb{N}$,
\[
    \sum_{u^* \le u \le a}(\alpha-u)^s\kappa_u = \bigO(1).
\]
\end{proposition}
\begin{proof}
Let $1 \le u \le a$. From condition \ref{tame:tail}) we obtain that $k_u = \kappa_u n / u < b^{-(\alpha-u)\gamma(\alpha-u)}n/u$, and since $\gamma \to \infty$, we may choose $\alpha - u^* = o(\ln n)$ so that $k_u < 1$ for all $u < u^*$, as claimed. Moreover, for $x$ large enough we have $b^{-\gamma(x)} \le 1/10$ and so
\[
    \sum_{u^* \le u \le a}(\alpha-u)^s\kappa_u 
    \le 
    \bigO(1) + \sum_{u^* \le u \le a}(\alpha-u)^s 10^{-\alpha + u}
    \le 
\bigO(1)+    \sum_{v \ge 0}v^s 10^{-v} = \bigO(1).
\]
\end{proof}
\begin{remark}
In the following, whenever we consider a tame profile $\bfk$ we will tacitly assume functions $a= \alpha_0 - \bigO(1)$ with $a \le \alpha$ and $u^* \sim \alpha$ such that $k_u =0$ unless $u^* \le u \le a$, and we will use these functions $a, u^*$ without further reference or explanation. Moreover, without loss of generality, we will assume that $u^* \le \alpha-2$ and that $\gamma (\alpha - u^*) \ge 10$ (since $\gamma(x) \rightarrow \infty$ as $x \rightarrow \infty$). In particular, whenever we write $\sum_u, \prod_u$ and the profile is tame, then the range of $u$ will be $u^* \le u \le a$.
\end{remark}

\begin{remark}
If $\bfk$ is tame, then we obtain that
\[
    n = \sum_{u} u k_u \sim 2 \log_b n \sum_{u} k_u,
\]
and so the total number of colours $k = \sum_u k_u \sim {n}/{2 \log_b n}$. We will also use this fact several times without explicit reference.
\end{remark}

\subsection{The general result}

Our main (technical) result provides  a lower bound for the probability that an $a$-bounded colouring with a tame profile exists, provided that there are many $a$-sets  and that the expected number of colourings and partial colourings is large enough.
\begin{theorem}\label{theorem:general}
Let $p \in (0,1)$ and $\epsilon>0$ be fixed and $m=\left \lfloor Np \right \rfloor$. Let $a=a(n)$ be an integer sequence such that $a=\alpha_0 - \bigO(1)$ and 
\begin{equation}
\label{eq:lowerboundmu}
    \mu_a \ge n^{1+\epsilon}.
\end{equation}	
Suppose that $\kp=\kp(n)$ is a tame $a$-bounded sequence of complete colouring profiles so that
\begin{equation} \label{eq:lowerboundexpectation}
    \ln \E_m[\bar X_\bfk] \gg \ln n.
\end{equation}
Furthermore, suppose that for any fixed ${\delta}>0$, if $n$ is large enough, then for all colouring profiles $\boldlambda \le \boldkappa$ such that ${\delta} \le \sum_{1 \le u \le a} \lambda_u \le 1-{\delta}$ (and letting $\ell_u = n \lambda_u /u$),
    \begin{equation}\label{eq:lowerboundbeta}
       \E_p\left[\bar{X}_{\boldlambda}\right] \ge \exp \Big( \ln^6 n \Big) \prod_{1 \le u \le a}{\binom{k_u}{\ell_u}}^2 .    \end{equation}
Then
\begin{align}
    \Pb_m \left( X_\bfk >0 \right) &\gtrsim \exp \Big(   - \frac{k_a^2}{\mu_{a}} -\bigO(\MB)\Big),
    \quad\text{where}\quad
    \MB = \frac{k_a^4 \ln^2 n }{n \mu_{a}^2}. \label{eq:defMB}
\end{align}
\end{theorem}
Let us comment briefly on the assumptions~\eqref{eq:lowerboundexpectation} and \eqref{eq:lowerboundbeta}.
We will see later in Lemma~\ref{lemma:kstartame} that assuming $\E_m[\bar X_\bfk] \gg \ln n$ poses no real restriction, as in the range of $k$ that we consider, increasing the number of colours by just one boosts the expectation by a factor of $e^{\Omega(\ln^2n)}$ (or for equitable colourings, by $e^{\Omega(\ln n\ln\ln n)}$ in the worst case).
Moreover, as we already mentioned after Definition~\ref{deftame}, for $X_\bfk >0$ to hold it is necessary that $\E[\bar X_\bfl]$ 
is large enough; the binomial coefficients account for the choice that we have when choosing the partial profiles.
As we will see, the requirement \eqref{eq:lowerboundbeta} will hold with room to spare, as for the `optimal' colouring profile under consideration later, the expected numbers of sub-colourings in question will actually be exponential in $n$.

Using that $k_u \le k = \bigO(n/\ln n)$ for all $u$, we immediately obtain the  following statement that provides conditions for the whp-existence of a colouring.
\begin{theorem}
\label{maintheorem}
Under the conditions of Theorem \ref{theorem:general}, if additionally $\mu_a \gg n^2 / \ln^2 n$, then with high probability, $G_{n,m}$ has a colouring with profile $\bfk$.
\end{theorem}
As we will see below in \eqref{expectationu}, 
$\mu_a \gg n^2/\ln^2 n$ is equivalent to $\mu_{a+2} \rightarrow \infty$. In particular, for $p=1/2$, this is the case if $a \le \alpha(n)-2$, which will be crucial in the verification of Theorem~\ref{twopointrestricted}.

\section{Preliminaries}
\label{section:preliminaries}

In this section we collect several facts and basic bounds concerning independent sets in random graphs and  the expected number of colourings with certain properties. The proofs are mostly routine (but sometimes involved) and are all contained in the appendix. 

\subsection{Independent sets in random graphs}
\label{ssec:isets}

Fix $p \in (0,1)$, and let $N_t$ be the number of $t$-sets --- that is, independent vertex sets of size $t$ -- in $G_{n,p}$. Recall the definitions of $\alpha_0$ and $\alpha$ from~\eqref{eq:alpha0}, which we repeat here for convenience:
\[
    \alpha_0 = \alpha_0(n) = 2 \log_b n -2 \log_b\log_b n +2 \log_b(e/2)+1,
    \qquad \alpha = \alpha(n) = \left\lfloor \alpha_0(n)\right \rfloor.
\]
As we already saw, for almost all $n$ the independence number of $G_{n,p}$ satisfies $\alpha(G_{n,p}) = \alpha$ whp. Moreover, $\alpha_0$ is the approximate value of~$u$ for which the expected value $\mu_u = \E_p[N_u] = \binom{n}{u}q^{\binom{u}{2}}$, interpreted suitably for $u \in \mathbb{R}$, is equal to 1. 
The following statement, proved in the appendix, gives us bounds on the expected number of $u$-sets when $u$ is close to $\alpha_0$.
\begin{lemma}
\label{prop:basicboundsisets}
Let $0 < p < 1$ be fixed. Let $u(n) \le \alpha_0(n)$ be an integer sequence such that $x = x(n) := \alpha_0 - u = \bigO(1)$. Then 
\begin{equation}
\label{eq:mut}
    \mu_u
    = n^{x+\bigO(\ln \ln n / \ln n)}.
\end{equation}
Moreover, if $0 < p < 1-e^{-4}$,
\begin{equation}\label{eq:muinfinity}
    \mu_{\alpha} = \E_p[N_\alpha] \rightarrow \infty.
\end{equation}
\end{lemma}
Now define $\theta= \theta(n)$ by letting
\begin{equation}
    \label{eq:defx2}
    \mu_\alpha = n^\theta.
\end{equation}
Then, by \eqref{eq:mut},
\begin{equation}\label{eq:theta}
    \theta=\theta(n)= \alpha_0 - \alpha + \bigO\left( \frac{\ln \ln n}{\ln n} \right) \in [o(1),1+o(1)],
\end{equation}
so that $\theta$ is essentially the fractional part of $\alpha_0$. In the lemma below, which is also proved in the appendix, we give bounds for $\mu_u$ for a large range of $u$.
\begin{lemma}
\label{prop:muurelativetomua}
Let $0 < p < 1$ be fixed. Let $a=a(n)=\alpha_0-\bigO(1)$ and $u = u(n)$ be integer sequences such that $0.1 \alpha \le u \le 10 \alpha$. Then, uniformly in $u$, 
\begin{equation}
\label{expectationu}
    \mu_{u}
    = \mu_a \left(\Theta\Big( \frac{n}{\ln n}b^{-(a-u)/2}\Big) \right)^{a-u}.
\end{equation}
\end{lemma}

\subsection{Bounds on the expected number of colourings}
\label{section:firstmomentlemmas}

This section contains a few basic but important lemmas that determine or provide bounds for the expected number of colourings. 
 We start with a bound that determines the leading exponential order of the expected number of colourings with a given profile. 
\begin{lemma} \label{expectationlemma}
Let $0 < p < 1$ be fixed. Let $\boldkappa = \boldkappa(n)$ be a sequence of (complete or partial) colouring profiles such that $\kappa_u =0$ for all $u< 0.1 \alpha$ and for all $u>10\alpha$, and such that $\sum_u \kappa_u (\alpha-u)^2 = \bigO(1)$. Then 
\[
    \E_m\big[\bar{X}_\boldkappa\big], \E_p\big[\bar{X}_\boldkappa\big] 
    = \exp\left(\varphi(\boldkappa)n + \bigO\Big(\frac{n \ln \ln n}{\ln n} \Big)\right),
\]
where
\begin{align*}
    \varphi(\boldkappa) &= -\parenth{1-\kappa} \ln\parenth{1-\kappa} +\frac{\ln b}{2} \sum_{u} \kappa_u \Big(\alpha_0-1-\frac{2}{\ln b}-u\Big),
    \quad
    \kappa = \sum_u \kappa_u.
\end{align*}
If $\boldkappa$ is complete and so $\kappa =1$, then we let $\parenth{1-\kappa} \ln\parenth{1-\kappa} = 0$ in the definition of $\varphi$. 
\end{lemma}
For tame profiles we have $k_u \neq 0$ only for $u \sim 2\log_bn$ and so $\kappa_u = k_u u/ n \sim k_u/k$, uniformly in $u$. In this case, if in addition $\E_p\big[\bar{X}_\boldkappa\big] \rightarrow \infty$ but not too quickly, the lemma readily implies that $\varphi(\kappa)=o(1)$ and so $\sum_{u}\kappa_u u = \alpha_0-1-{2}/{\ln b}+o(1)$. In that case the average colour class size $n/k$ satisfies
\[
    \frac{n}k
    = \frac{\sum_u k_uu }{k}
    \sim \sum_u \kappa_u u
    = \alpha_0-1-\frac 2 {\ln b} +o(1).
\]
The next lemma provides an upper bound for the expectation when the average colour class size is larger than this,  which also holds if $\sum_u \kappa_u u^2$ is unbounded.
\begin{lemma} \label{lemmaupperbound}
Let $0 < p < 1$ and $C >0$ be fixed. Let $\bfk$ be a sequence of complete colouring profiles so that $k_u=0$ for all $u<0.1 \alpha$ and for all $u>10 \alpha$. If
\[
    \frac{n}{k} > \alpha_0-1-\frac{2}{\ln b}+C,
\]
then
\[
    \E_m\big[\bar{X}_\bfk\big], \E_p\big[\bar{X}_\bfk\big] 
    < \exp \Big (-\frac{C\ln b}{2}n + o(n) \Big).
\]
\end{lemma}
In the next lemma we give a lower bound for the expected number of $a$-bounded colourings if $\mu_a$ is large enough, with a very small error term.
\begin{lemma}
\label{expectationlowerbound}
Let $0<p<1, \epsilon> \delta >0$ be fixed, let $a=a(n)=\alpha_0-\bigO(1)$ be a sequence of integers such that $\mu_{a} \ge n^{1+\epsilon}$, and suppose that $u^*=u^*(n) \sim 2 \log_b n$. Then for all sequences~$\boldkappa$ of $a$-bounded colouring profiles such that $\kappa_u=0$ for all $u < u^*$, uniformly
\[
    \E_m\big[\bar{X}_\bfk\big], \E_p\big[\bar{X}_\bfk\big]
    \ge \exp\left(\tilde \varphi_\delta(\kappa) n+\bigO( \ln^2 n)\right),
    \quad
    \text{where}
    \quad
    \kappa=\sum_{u} \kappa_u,
\]
and $\tilde \varphi_\delta:[0,1] \rightarrow \mathbbm{R}$ satisfies
\[
\tilde \varphi_\delta(x)
=-(1-x)\ln(1-x) -\left(1-\frac{\ln b}{2}\delta \right) \,x,
\]	
where, as before, we set $0 \ln 0 := 0$ in the case $\kappa=1$.
\end{lemma}
The following statement 
will be useful in the verification of Theorem~\ref{theorem:general}, giving a lower bound on the expected number of partial profiles for a range of $\lambda= \sum_u \lambda_u$ where the lower bound \eqref{eq:lowerboundbeta} does not apply.
\begin{lemma} \label{lemma:smallsubprofiles}
Assume the conditions of Lemma \ref{expectationlowerbound}. Then 
there is a $C= C(\varepsilon) > 0$ so that the following holds. If $\boldkappa$ is a sequence of $a$-bounded colouring profiles so that $\kappa_u=0$ for all $u <u^*$, then uniformly for all sequences $\boldlambda=(\lambda_u)_{u^* \le u \le a}$ with $0 \le \lambda_u \le \kappa_u$ and $\ln^{-3} n \le \sum_u \lambda_u \le C$, 
\[
\E_p\left[\bar{X_{\boldlambda}}\right]
\ge \exp \Big(\Theta\Big(\frac{n}{\ln^3 n}\Big)\Big) \prod_{u}{\binom{k_u}{\ell_u}}^2.
\]
\end{lemma}

\subsection{Basic bounds}

In this section we collect some basic bounds that will be useful later. We begin with the following lemma that gives us the asymptotic probability that a given set of $x$ edges is not present in $\Gnm$. 
\begin{lemma} \label{gnmlemma}
Let $p \in (0,1)$ be constant and $m = Np + \bigO(1)$. If $x=x(n)=o(n^{4/3})$, then 
\[
    \frac{{N-x \choose m}}{{N \choose m}} \sim q^{x} \exp \left(- \frac{(b-1)x^2}{n^2}\right)
\]
where $q=1-p$ and $b=1/q$.
\end{lemma}
The following bound is a handy technical estimate for factorials.
\begin{lemma} \label{ablemma}
Let $a=a(n)$, $b=b(n)$ be two nonnegative integer sequences such that $0 \le b-a \rightarrow \infty$. Then
\[
    \frac{(b-a)!}{b!} \lesssim b^{-a} e^{a^2/b}. 
\]
\end{lemma}
The following statement of analytic flavour will be needed in the proof of Lemma~\ref{lemma:kstartame}.
\begin{lemma} \label{lemma:technicalsequence}
For every $x \in [0,1]$, let $(s_i(x))_{i \ge 1}$ be a sequence of real numbers so that $s_i(x)$ is a continuous function of $x \in [0,1]$ for all $i \ge 1$. Suppose further that for all $x \in [0,1]$, $s_i(x) \rightarrow1$ as $i \rightarrow \infty$. Then $\min_{x \in [0,1]} s_i(x) \rightarrow 1$ as $i \rightarrow \infty$.
\end{lemma}

\section{Proof overview for Theorem~\ref{theorem:general}}
\label{section:proofoverview}
We will prove Theorem \ref{theorem:general} in \S\ref{section:proofoverview}, \S\ref{section:firstmoment} and \S\ref{section:secondmoment}. Throughout these sections, let $G \sim G_{n,m}$ with $m=\left \lfloor pN \right \rfloor$ where $p \in (0,1)$ is constant, fix a sequence $a=a(n)= \alpha_0(n)-\bigO(1)$ and let $\bfk$ be a tame sequence of $a$-bounded colouring profiles. Let $c>0$ be the constant and $\gamma(n)$ be the function from Definition~\ref{deftame} for the tame sequence $\bfk$, together with a function~$u^*=u^*(n) \sim a$ as mentioned in the remarks (since $\gamma(x) \rightarrow \infty$ as $x \rightarrow \infty$) so that $k_u = 0$ for all $u < u^*$.

Recall the Paley-Zygmund inequality, which states that for any non-negative random variable $Z$ with finite variance,
\begin{equation}\label{eq:payleyzygmund}
    \prob(Z>0) \ge \E[Z]^2/\E[Z^2].
\end{equation}
To prove Theorem \ref{theorem:general}, it suffices find a non-negative random variable $Z_\kp$ so that $Z_\kp>0$ implies $X_\bfk >0$, and so that under the conditions of Theorem \ref{theorem:general},
\begin{equation}
    \E[Z_\kp^2]/\E[Z_\kp]^2 \lesssim \exp \left( \frac{k_a^2}{\mu_{a}} +\bigO(\MB)\right). \label{smratio}
\end{equation}
The most natural candidate for $Z_\kp$ is the number of unordered colourings with profile $\kp$.
However, determining the second moment of this random variable is a complex and elaborate task, and there is no a priori reason to expect a statement like~\eqref{smratio}. 
	
Fortunately, we are able to take a different route. We greatly simplify the underlying structure of the possible pairs of colourings by only considering colourings that are \emph{well-separated} in a certain sense. More specifically, we show that it is sufficient to consider only colourings~$\pi$ with the property that if a colour class from another colouring has significant overlap with a colour class in~$\pi$, the two colour classes must be essentially identical. This  property guarantees that if we look at  any pair of colourings and remove the (essentially) identical colour classes, then what remains looks completely random; in essence all correlations disappear and this is the most important ingredient in verifying~\eqref{smratio}. The 
identical colour classes of size $a$ 
 shared between pairs of colourings will account for the terms $k_a^2 / \mu_a+\bigO(\MB)$.

\subsection{Definition of $Z_\bfk$} \label{sectiondefinitionrandomvariable}

Before defining the random variable $Z_\bfk$ that we will consider, let us begin with a simple notion that quantifies how a vertex set is put together from the sets of a vertex partition.
\begin{definition}\label{def:tcomposed}
    For an ordered partition $\pi=(V_1, \dots, V_k)$ of $V$ and a set $S \subset V$, we write
    \[
    z(S, \pi):= \big| \{i: V_i \cap S \neq \emptyset \big\}|.
    \]
    If $z(S, \pi)=z$, we say that $S$ is \emph{$z$-composed} with respect to $\pi$.  
\end{definition}
With this definition at hand, we will define some events for vertex partitions $\pi$ in $\Gnm$ below. This will ensure that any two partitions $\pi, \pi'$ which define colourings in $G_{n,m}$ and for which these events hold are rather \emph{dissimilar} in the following sense. As we will see, for any two such partitions $\pi = (V_1, \dots , V_k)$  and $\pi' = (V'_1, \dots, V'_k)$ which are colourings in $\Gnm$,  the colour classes $V'_i, 1 \le i \le k$, up to a small number of exceptions, have no non-trivial intersection with the $V_i$'s. More specifically, essentially every~$V'_i$ is either identical to a colour class of $\pi$, or it is composed from $(1-o(1))|V'_i|$ colour classes of $\pi$. This is a quite useful property: in the study of the second moment $\E[Z_\bfk]$, after taking the identical parts into consideration, it will allow us to restrict our considerations to pairs $(\pi, \pi')$ that do not interfere much. Indeed, if $\pi,\pi'$ are dissimilar as just described, then knowing that $\pi$ induces a colouring of $\Gnm$ doesn't really tell us much about the chances that $\pi'$ is a colouring, as only a small number of edges within the $V'_i$'s are known not to be contained in the graph.  So the events that $\pi$ and $\pi'$ both induce colourings are almost independent, and then the second moment is under control. Let us proceed with the precise list of properties that we will assume.
\begin{definition} \label{defrandomvariable} 
Let $\pi = (V_1, \dots, V_k)$ be an ordered complete partition of $V$ with tame profile~$\bfk$. Define the following events.
\begin{itemize}
\itemsep1pt
	\item[$A_\pi:$] $\pi$ is a colouring, that is, the sets $V_1, \dots, V_k$ are independent. 
	\item[$B_\pi:$] If $S \subset V$ is an independent set 
 with $u^* \le |S| \le a$, then
    \[
        z(S,\pi)\le 2  ~~\text{ or }~~ z(S,\pi) \ge |S| - 2 (\alpha - |S|)-1.
    \]
    That is, $S$ is composed of vertices from at most two parts of $\pi$, or from at least $|S| - 2 (\alpha - |S|)-1$ parts.
    			
	\item[$C_\pi:$] If $S\subset V$ is an independent set 
 such that $u^*\le |S|\le a$ and such that $z(S,\pi)= 2$, then there are parts $V_i$ and $V_j$ such that 
	\[
        |S \cap V_{i}|= 1 ~~\text{ and }~~ |S \cap V_j|=|S|-1.
    \]
	That is, any 2-composed independent set $S$ has a simple structure: it is obtained by taking one vertex from a part $V_{i}$ of $\pi$, and all other vertices from another part $V_{j}$.
			
		\item[$D_\pi:$] There are at most $\ln^3 n$ independent sets $S$ of size between $u^*$ and $a$ such that 
		\[
			z(S,\pi)= 2
            \quad
			\text{and}
			\quad
			\text{there is a part $V_i$ such that  $|S \cap V_i| \ge |V_i|-1$}.
		\]
		\end{itemize}
		\end{definition}
We are now ready to define the random variable $Z_\kp$ which we will apply the second moment method to: let
\[
    Z_\kp = \sum_{\pi \text{ of profile }\kp} \xi_\pi,
    \quad
    \text{where} 
    \quad
    \xi_{\pi}= \mathbbm{1}_{A_\pi \cap B_\pi \cap C_\pi \cap D_\pi}.
\]	
Obviously $Z_\kp>0$ implies $X_\kp>0$, because $\xi_\pi>0$ implies $A_\pi$. The first, crucial,  point is that switching from $X_\kp$ to the more restrictive $Z_\kp$ has little impact on the first moment, and therefore on the expectation threshold.
\begin{proposition}
    \label{lem:XkZk} 
    Let $\bfk$ be tame. Then $\ex{Z_\kp} \sim \ex{X_\kp}$.
\end{proposition}
The second important point is that the heavy structural restrictions (`dissimilarity') imposed by $B_\pi \cap C_\pi \cap D_\pi$ on possible pairs of colourings allow us to obtain a sufficient upper bound on the second moment of $Z_\kp$.
\begin{proposition}  
\label{secondmomentgoal}
Let $\bfk$ be tame. Then, under the conditions of Theorem~\ref{theorem:general} and recalling the definition~\eqref{eq:defMB} of $\MB$,
\[
    \frac{\ex{Z^2_\kp}}{\ex{Z_\kp}^2} \le  \exp \left(\frac{k_a^2}{\mu_{a}}+\bigO(\MB)+o(1)    \right).
    \]
\end{proposition}
Theorem \ref{theorem:general} then follows immediately from Proposition~\ref{secondmomentgoal} and Payley-Zygmund~\eqref{eq:payleyzygmund}. We will prove Proposition~\ref{lem:XkZk} in \S\ref{section:firstmoment} and Proposition~\ref{secondmomentgoal} in \S\ref{section:secondmoment}.
		
\section{The first moment}
\label{section:firstmoment}

In this section we prove Proposition~\ref{lem:XkZk}, showing that the expectations of $X_\bfk$ and $Z_\bfk$ are asymptotically the same. 
To this end, let $\pi(n)$ be a partition with profile $\bfk(n)$ for each $n\in\mathbb{N}$, where $\bfk$ is tame. We will show in Lemmas~\ref{lem:Bpi},~\ref{lem:Cpi} and~\ref{lem:Dpi} that
\[
    \Pb(\overline{B_\pi} ~|~ A_\pi), \quad 
    \Pb(\overline{C_\pi} ~|~ A_\pi), \quad
    \Pb(\overline{D_\pi} ~|~ A_\pi)
    = o(1),
\]
from which Proposition~\ref{lem:XkZk} immediately follows. Since it is no additional effort, we will also prove the same statements for the probabilities in $G_{n,p}$. In the rest of this section we will write without further reference
\[
    N_\pi := \binom{n}{2} - \sum_{u^* \le u \le a} k_u\binom{u}{2} = N - \Theta(n\log n)
\]
for the number of admissible edges in $G$ given $A_\pi$, that is, given that $\pi$ is a colouring. Recall that $G \sim G_{n,m}$ has \[m = \lfloor pN \rfloor\] edges in total. We continue right away with the proof for $B_\pi$. 
\begin{lemma}
    \label{lem:Bpi}
    With the notation in this section, $\Pb(\overline{B_\pi} ~|~ A_\pi) = o(1)$ and $\Pb_{p}(\overline{B_\pi} ~|~ A_\pi)$.
\end{lemma}
\begin{proof}
    Let $\pi = (V_1, \dots, V_k)$ (where we suppress the dependence on $n$). We begin with a simple observation. For $c>0$, let $E(c)$ be the event that there is an independent set $S$ such that $|S \cap V_i| \cdot |S \cap V_j| \ge c \ln n$ for some $1 \le i < j \le k$. There are at most $n^2$ ways to choose $i,j$ and at most $2^{|V_i| + |V_j|} \le 2^{4\log_b n}$ ways to choose subsets from $V_i$ and $V_j$. If $E(c)$ occurs, then the edges with one endpoint in $S \cap V_i$ and the other in $S \cap V_j$ are not contained in $G_{n,m}$. Thus
			\[
			\Pb(E(c)~|~A_\pi)
			\le n^2 \cdot 2^{4\log_b n} \cdot {\binom{N_\pi - c \ln n}{m}} / {\binom{N_\pi}{m}}.
			\]
By applying Lemma~\ref{gnmlemma} twice we obtain that the ratio of the binomial coefficients is $\sim q^{c\ln n}$. Thus
			\[
			\Pb(E(c)~|~A_\pi)
			\lesssim   n^2 \cdot 2^{4\log_b n} \cdot q^{c\ln n}.
			\]
			It is clear that we can choose $c$ independent of $n$ such that this probability is $o(1)$; let $c_0$ be such a choice and let us write (in a slight abuse of notation) $E = E(c_0)$. Then the complementary event $\overline{E}$  asserts that all independent sets $S$ in $G_{n,m}$ have the property that $|S\cap V_i|\cdot|S \cap V_j| \le c_0\ln n$ for all $i<j$, that is, there is no significant overlap of any independent set with two (or more) colour classes in $\pi$.  With this at hand, denote by
            $$
			I_{u,z}, \quad u^*\le u \le a, \quad z \in \{3, \dots, u-2(\alpha-u)- 2\}
		$$
		the event that there is an independent set $S$ with $u$ vertices that is $z$-composed. We will show that
			\[
			\sum_{u=u^*}^{a} \sum_{z=3}^{u-2(\alpha-u)-2} P(I_{u,z} \cap \overline{E}~|~A_\pi) = o(1),
			\]
			from which the claim in Lemma~\ref{lem:Bpi} follows immediately.  
			
			Let $u,z$ be fixed in the respective intervals. We estimate the probability of 
			$I_{u,z} \cap \overline{E}$ as follows. There are at most $\binom{n}{z}$ ways to choose $z$ vertices in $z$ different parts of $\pi$. Let us consider a specific choice such that these vertices are contained in $V_{i_1}, \dots, V_{i_z}$, where $1 \le i_1 < i_2 < \dots < i_z \le k$. We pick non-negative integers $x_1, \dots, x_z$ that sum up to $u-z$ and such that $(1+x_i)(1+ x_j) \le c_0 \ln n$, and then we select for each $1\le j \le z$ another $x_j$ vertices from $V_j$. The number of such choices is at most $\prod_{1\le j \le z} \binom{a}{x_j}$. Finally, the selected subsets from $V_{i_1}, \dots, V_{i_z}$ form an independent set of size $u$ if and only if the $\sum_{1 \le i < j \le z} (1+x_i)(1+x_j)$ edges between them are not included in $G_{n,m}$, conditional on $A_\pi$. So abbreviating 
			\[
			{\cal X}_{u,z} = \big\{(x_1, \dots, x_z) \in \mathbb{N}_0^{z}: x_1 + \dots + x_z = u-z, (1+x_i)(1+x_j) \le c_0 \ln n \big\},
			\]
			we obtain that
			\[
			\Pb(I_{u,z} \cap \overline{E} ~|~ A_\pi)
			\le
			\binom{n}{z} \sum_{(x_1, \dots, x_z)\in{\cal X}_{u,z}} \prod_{1\le j \le z} \binom{a}{x_j} \cdot \frac{\binom{N_\pi - \sum_{1\le i < j \le z}(1+x_i)(1+x_j)}{m}}{\binom{N_\pi}{m}}.
			\]
Using Lemma~\ref{gnmlemma} twice we obtain that the fraction of binomial coefficients in the previous expression is $\sim q^{\sum_{1 \le i < j \le z}(1+x_i)(1+x_j)}$. Moreover, for each $(x_1, \dots, x_z)\in{\cal X}_{u,z}$,
\[
    \sum_{1 \le i < j \le z}(1+x_i)(1+x_j)
= \frac{u^2}{2} - \frac12 \sum_{1\le j \le z} (1+x_j)^2.
\]
			This is minimized when all $x_j = 0$, except for one that equals $u-z$; in that case the value is 
			\[
			\frac12\big(u^2 - (u-z+1)^2 - z + 1\big) = (z-1)(u-z/2).
			\]
			By putting everything together we obtain that
			\[
			\Pb(I_{u,z} \cap \overline{E} ~|~ A_\pi)
			\lesssim 
			\binom{n}{z} \cdot q^{(z-1)(u-z/2)} \cdot |{\cal X}_{u,z}| \cdot \max_{(x_1, \dots, x_z)\in{\cal X}_{u,z}} \prod_{1\le j \le z} \binom{a}{x_j}.
			\]
			In the next steps we will estimate the binomial coefficient and the exponent of $q$. Since it will be quite convenient, let us write
			\[
			z = u - 2(\alpha-u) - d, \quad \text{where} \quad 2 \le d \le u-2(\alpha-u) + 3.
			\]
			Regarding the binomial coefficient, note that 
			\[
			\binom{n}{z}
			\le \Big(\frac{en}z\Big)^z
			= q^{-z \log_b n + z \log_b z - z \log_b e}.
			\]
			Moreover, recalling from~\eqref{eq:alpha0} that $\alpha \ge 2\log_bn - 2\log_b \log_b n + 2\log_b(e/2)$, we obtain that
			\[
			(z-1)(u-z/2)
			= (z-1)(2\alpha -u-d)/2
			\ge (z-1)\Big(\log_b n - \log_b\log_b n + \log_b(e/2) + \frac{\alpha-u+d}2 
			\Big).
			\]
			Combining these bounds, and using $z \ge 3 > e$, yields \begin{equation}
				\label{eq:globalBound}
				\Pb(I_{u,z} \cap \overline{E} ~|~ A_\pi)
				\lesssim 
				n q^{(z-1)\big(\log_b(z/2\log_b n) + (\alpha-u+d)/2 \big)}  \cdot |{\cal X}_{u,z}| \cdot \max_{(x_1, \dots, x_z)\in{\cal X}_{u,z}} \prod_{1\le j \le z} \binom{a}{x_j}.
			\end{equation}
			From here on we will distinguish three cases. Consider first the `few parts' case, where $z \le (\ln\ln n)^2$; or equivalently $d\ge u - 2(\alpha-u) -  (\ln\ln n)^2$. In particular $d$ is of logarithmic order. Let us look more closely at a sequence ${\bf x} =  (x_1, \dots, x_z) \in {\cal X}_{u,z}$. Since the $x_j$'s sum up to $u-z$, there is at least one one index $j^*$ with $x_{j^*} \ge (u-z)/z = \Omega((\ln n) / z)$. As ${\bf x} \in {\cal X}_{u,z}$, we obtain that there is a constant $c_1 > 0$ such that $x_j \le \lfloor c_1 z \rfloor$ for all $j \neq j^*$. With this observation at hand, we can bound $|{\cal X}_{u,z}|$ as follows. There are at most $z$ ways to choose $j^*$, $a$ ways to pick $x_{j^*}$ and $c_1 z$ ways to pick each of the $x_j$ with $j \neq j^*$. So as $z \le (\ln \ln n)^2$,
\[
    |{\cal X}_{u,z}|
    \le z a (c_1 z)^{z-1}
    = n^{o(1)}.
\] 
The monotonicity and log-concavity of binomial coefficients, the bound $\binom{a}{\lfloor c_1 z \rfloor} \le a^{c_1z}$, and the fact that $u=a-o(\log n)$ guarantee for such $\bf x$ that
\[
    \prod_{1\le j \le z} \binom{a}{x_j} \le \binom{a}{ \lfloor c_1 z \rfloor}^{z-1}\binom{a}{u-z-\lfloor c_1z \rfloor(z-1)}
    \le a^{c_1 z^2} \cdot \binom{a}{a - o(\log n)}
    = n^{o(1)}.
\]
Note also that  $\log_b(z/2\log_b n) \ge -\log_b\log_b n = -o(d)$ as $d$ is of logarithmic order, and that $\alpha-u+d\sim \alpha \sim 2 \log_b n$. 
Since $n^{o(1)} = q^{-o(\ln n)}$ and $z \ge 3$, we obtain in the `few parts' case from~\eqref{eq:globalBound}
			\begin{equation}
				\label{eq:few}
				\begin{split}
					\Pb(I_{u,z} \cap \overline{E}~|~ A_\pi)
					& \le 
					n^{1+o(1)} \cdot q^{(1-o(1))(z-1)(\alpha-u+d)/2} \\
					& \le n^{1+o(1)}\cdot q^{(2-o(1))\log_b n}
					= n^{-1+o(1)},
					\quad 3 \le z \le (\ln\ln n)^2.
				\end{split}
			\end{equation}
Consider the `intermediate parts' case, where $(\ln\ln n)^2 \le d < u-2(\alpha-u)-(\ln\ln n)^2$; in particular $z > (\ln\ln n)^2$. We again use that $\log_b(z/2\log_b n) \ge -\log_b\log_b n$. Moreover, since  ${\cal X}_{u,z}$ is a subset of the solutions in the non-negative integers to $x_1 + \dots + x_z = u-z$,
			\begin{equation}
				\label{eq:allCount}
				|{\cal X}_{u,z}| \le \binom{u-1}{z-1} \le u^{z-1} \le a^{z-1}.
			\end{equation}
			This bound is true for all $z$ and we shall use it again later.
			Consider also the crude bound
			\begin{equation}
				\label{eq:crudeBoundProd}
				\prod_{1\le j \le z} \binom{a}{x_j}
				\le a^{\sum_{1 \le j \le z} x_j} =a^{u-z}= a^{2(\alpha-u) + d},
				\quad (x_1, \dots, x_z) \in {\cal X}_{u,z},
			\end{equation}
			valid also for all $z$. By plugging all these estimates into~\eqref{eq:globalBound} we obtain in the `intermediate parts' case that
			\[
			\Pb(I_{u,z} \cap \overline{E} ~|~ A_\pi)
			\lesssim 
			n \cdot q^{(z-1)\big(-\log_b\log_b n + (\alpha-u+d)/2+o(1) - \left(1 + (2(\alpha-u) + d) / (z-1)\right) \log_b a\big)}
			\]
			and since $z = \omega(\ln\ln n)$, $d=\omega(\ln \ln n)$, $\log_b a = \log_b\log_b n + \bigO(1)$, with room to spare
			\begin{equation}
				\label{eq:intermediate}
				\Pb(I_{u,z} \cap \overline{E} ~|~ A_\pi)
				\lesssim 
				n \cdot q^{(z-1)(d/2 - o(d))}
				\le  n^{-1}, \quad (\ln\ln n)^2 \le d \le u-2(\alpha-u)-(\ln\ln n)^2.
			\end{equation}
			Finally, let us consider the `many parts' case where $2 \le d \le (\ln\ln n)^2$, that is, $z \ge u-2(\alpha-u)-(\ln\ln n)^2 \sim 2\log_b n$. Then ${z}/{2 \log_b n}	= 1- o(1)$ and so $\log_b({z}/{2 \log_b n}) = o(1)$. From~\eqref{eq:globalBound} we obtain that
			\[
			\Pb(I_{u,z} \cap \overline{E} ~|~ A_\pi)
			\le n^{1-(1+o(1))(\alpha-u+d)} 
			\cdot |{\cal X}_{u,z}| \cdot \max_{(x_1, \dots, x_z)\in{\cal X}_{u,z}} \prod_{1\le j \le z} \binom{a}{x_j}.
			\]
			Moreover, from the bound in~\eqref{eq:crudeBoundProd} we obtain that 
			\[
			\prod_{1\le j \le z} \binom{a}{x_j}
			\le a^{2(\alpha-u) +d} = n^{o(\alpha-u+d)}.
			\]
			Regarding $|{\cal X}_{u,z}|$, we obtain from~\eqref{eq:allCount} that 
			\[
			{\cal X}_{u,z}
			\le \binom{u-1}{z-1} = \binom{u-1}{2(\alpha-u)+d}
			\le u^{2(\alpha-u)+d}\le a^{2(\alpha-u)+d}
			= n^{o(\alpha-u+d)}.
			\]
			Thus, again with room to spare and since $\alpha-u \ge 0$ and $d \ge 2$,
			\[
			\Pb(I_{u,z} \cap \overline{E}~|~A_\pi)
			\le n^{1-(1+o(1))(\alpha-u+d)} \cdot n^{o(\alpha-u+d)}
			\le n^{-1+o(1)}, \quad 2 \le d \le (\ln\ln n)^2.
			\]
			Combining this bound with~\eqref{eq:few} and~\eqref{eq:intermediate} we obtain
			\[
			\sum_{u^*\le u \le a} \sum_{z=3}^{u-2(\alpha-u)-2} \Pb(I_{u,z} \cap \overline{E}~|~A_\pi) \le a\cdot a \cdot n^{-1+o(1)} = o(1),
			\]
			as desired.

   For the corresponding statement for $\Pb_p$, observe that throughout the proof, we estimated all probabilities in $G_{n,m}$ to be $(1+o(1))$ times the corresponding probability in $G_{n,p}$ via Lemma~\ref{gnmlemma}. So the conclusion in $G_{n,p}$ follows in exactly the same way.
		\end{proof} 
		
		\begin{lemma}
			\label{lem:Cpi}
			With the notation in this section $\Pb(\overline{C_\pi} ~|~ A_\pi) = o(1)$ and $\Pb_{p}(\overline{C_\pi} ~|~ A_\pi)$.
		\end{lemma}
		\begin{proof}
			Let $\pi = (V_1, \dots, V_k)$ (where we suppress the dependence on $n$). Suppose that $\overline{C_\pi}$ occurs. Then there is a $2$-composed independent set $S$ with $|S| \ge u^*$ and there are indices $1 \le i < j \le k$ such that $2 \le |S \cap V_i|,  |S \cap V_j| \le |S|-2$.
			As $S$ is independent, this means that $|S \cap V_i| \cdot |S \cap V_j|$
			edges are excluded from $G_{n,m}$ (conditional on $A_\pi$). We obtain that
			\[
			\Pb(\overline{C_\pi} ~|~ A_\pi)
			\le
			\sum_{1 \le i <j \le k} ~ \sum_{\substack{\ell, \ell' \ge 2 \\ u^* \le \ell + \ell' \le a}} \binom{|V_i|}{\ell} \binom{|V_j|}{\ell'} \frac{\binom{N_\pi - \ell \ell'}{m}} {\binom{N_\pi}{m}}.
			\]
Applying Lemma~\ref{gnmlemma} (twice) to the ratio of the binomial coefficients and using that $k \le n$ and $|V_i| \le a$ for all $i$ we obtain that  
			\[
			\Pb(\overline{C_\pi} ~|~ A_\pi)
			\lesssim
			n^2 ~ \sum_{\substack{\ell, \ell' \ge 2 \\ a-u^* \le \ell + \ell' \le a}} \binom{a}{\ell} \binom{a}{\ell'} q^{\ell \ell'}.
			\]
			Let us first consider the terms with $2 \le \ell \le \ln\ln n$ or $2 \le \ell' \le \ln\ln n$. Then $\ell \ell' \ge 2(u^*-2) = (4-o(1))\log_b n$ and moreover, the log-concavity of the binomial coefficients and the fact that $\ell+\ell'\sim 2\log_b n \sim a$ guarantees that
			\[
			\binom{a}{\ell} \binom{a}{\ell'}
			\le \binom{a}{\ln\ln n} \binom{a}{\ell+\ell' - \ln\ln n}
			= n^{o(1)}.
			\]
			Thus,
			\[
			\binom{a}{\ell} \binom{a}{\ell'} q^{\ell \ell'}
			\le n^{-4 + o(1)}, \quad 2 \le \ell \le \ln\ln n \text{ or } 2 \le \ell' \le \ln\ln n.
			\]
			On the other hand, if both $\ell, \ell' \ge \ln\ln n$, then as $\ell+\ell' \sim a = \Theta( \log n)$, we have $\ell\ell' = \omega(\ln n)$. As all binomial coefficients are $\le 2^a = n^{\bigO(1)}$, we obtain
			\[
			\Pb(\overline{C_\pi} ~|~ A_\pi)
			\le n^2 \cdot a^2 \cdot (n^{o(1) - 4} + n^{\bigO(1)}q^{-\omega(\ln n)})
			= o(1).
			\]
      For the corresponding statement for $\Pb_p$, observe that, as in the last lemma, we only ever calculated probabilities in $G_{n,m}$ to be $(1+o(1))$ times the corresponding probability in $G_{n,p}$ via Lemma~\ref{gnmlemma}. So again the conclusion in $G_{n,p}$ follows in the same way.
		\end{proof}

\begin{lemma}
\label{lem:Dpi}
    With the notation in this section $\Pb(\overline{D_\pi} ~|~ A_\pi) = o(1)$ and $\Pb_{p}(\overline{D_\pi} ~|~ A_\pi)$.
\end{lemma}
\begin{proof}
Let $\pi = (V_1, \dots, V_k)$ (where we suppress the dependence on $n$).  Let us write $Y$ for the (random) number of independent sets $S$ that are 2-composed and have the property that there is a $1 \le i \le k$ with $|S \cap V_i| \ge |V_i|-1$ \text{and} $|S| = |S \cap V_i|+1$ (that is, $S$ contains exactly one more vertex outside of $V_i$). We will show that
\begin{equation}
    \label{eq:boundSimplerCount}
    \Pb(Y \ge \ln^3 n ~|~ A_\pi) = o(1),
\end{equation}
from which the statement in the lemma follows immediately from Lemma~\ref{lem:Cpi}:
\[
    \Pb(\overline{D_\pi} ~|~ A_\pi) 
    \le \Pb(\overline{D_\pi} \cap C_\pi~|~ A_\pi) + \Pb(\overline{C_\pi} ~|~ A_\pi)
    \le \Pb(Y \ge \ln^3 n ~|~ A_\pi) + o(1).
\] 
We shall first compute the expectation of $Y$. 
For a vertex $v\in[n]$ let us write $c(v)$ for the unique index $j$ with $v \in V_j$. Define the families of indicator random variables
\[
    (Y_{v, i})_{v \in [n], i \in [k]\setminus \{c(v)\}}, 
    \qquad
    (Y_{v,w})_{v\in [n], w \in [n] \setminus V_{c(v)}},
\]
where $Y_{v, i}$ equals one if $\{v\} \cup V_i$ is an independent set, and $Y_{v,w}$ equals one if $\{v\} \cup V_{c(w)} \setminus \{w\}$ is independent.  Then, by definition, $Y$ is the sum of all these variables. From now on we will write $v$ for a vertex, $i$ for an index in $[k]\setminus \{c(v)\}$ and $w$ for a vertex in $[n] \setminus V_{c(v)}$. Note that
\[
    \E_m[Y_{v, i} ~|~ A_\pi] = \binom{N_\pi - |V_j|}{m} / \binom{N_\pi}{m}.
\]
By applying Lemma~\ref{gnmlemma} twice (by now we should have enough routine) we obtain
\[
    \E_m[Y_{v,i} ~|~ A_\pi] = (1+o(1))q^{|V_i|}
    ~\text{ and similarly }~
    \E_m[Y_{v, w} ~|~ A_\pi] = (1+o(1))q^{|V_{c(w)}|-1}	
\]
There are $n$ choices for $v$. Moreover, we can choose $w$ by selecting an $i \neq c(v)$ and then selecting a vertex from $V_i$. We obtain that
\[
\begin{split}
    \E_m[Y ~|~ A_\pi]
    & = \sum_{v \in [n], i \in [k]\setminus \{c(v)\}} \E_m[Y_{v, i} ~|~ A_\pi] + \sum_{v\in [n], w \in [n] \setminus V_{c(v)}} \E_m[Y_{v, w} ~|~ A_\pi] \\
    & = (1+o(1)) n \left(\sum_{1 \le i \le k} q^{|V_i|} + q^{-1}\sum_{1 \le i \le k} |V_i|q^{|V_i|}\right),
\end{split}
\]
and since $k \sim n/2\log_b n$ and $|V_i| \sim 2\log_b n$, recalling $\kappa_u= k_u u/n$ and that $\bfk$ is tame (so we may apply Proposition~\ref{prop:proptame}),
\[
    \E_m[Y ~|~ A_\pi]
    \le (q^{-1}+o(1)) n^2  q^a \sum_{u^* \le u \le a} \kappa_u q^{-a+u}
    \le \bigO(\ln^2 n) \sum_{u^* \le u \le a} \kappa_u q^{-a+u} = \bigO(\ln^2 n).
\]
The claim follows by applying Markov's inequality.

      For the corresponding statement for $\Pb_p$, as in the previous two lemmas, we only ever calculated probabilities and expectations in $G_{n,m}$ to be $(1+o(1))$ times the corresponding expressions in $G_{n,p}$ via Lemma~\ref{gnmlemma}. So again the conclusion in $G_{n,p}$ follows in the same way.
\end{proof}

\section{The second moment}
\label{section:secondmoment}
In this section we prove Proposition \ref{secondmomentgoal}. Throughout, we fix a sequence $\kp=\kp(n)$ of tame $a$-bounded $k$-colouring profiles with $a= \alpha- \bigO(1)$ and  $c$, $\gamma(n)$, $u^*(n)$ as in Definition \ref{deftame} and the remarks thereafter. Whenever we index a sum, product or sequence with $u$ or $v$, this means $u^*\le u \le a$ or $u^* \le v \le a$ unless otherwise specified. 
		
Recall Definition \ref{defrandomvariable} of the events $B_\pi$, $C_\pi$, and $D_\pi$. We say that a pair $(\pi, \pi')$ of ordered partitions is \emph{relevant} if, roughly speaking, $\pi'$ being a colouring would not violate the events $B_\pi$, $C_\pi$, and $D_\pi$, and $\pi$ being a colouring would not violate the events $B_{\pi'}$, $C_{\pi'}$, and $D_{\pi'}$. We spell out what this means in the following definition.
		
\begin{definition}\label{def:tamepairs}
Let $\pi=(V_i)_{i=1}^k$ and $\pi'=(V'_j)_{j=1}^k$ be ordered partitions with profile $\kp$. Then we say that $(\pi, \pi')$ is \emph{relevant} if all of the following hold.
\begin{enumerate}[a)]
    \item 
    \begin{enumerate}[1.]
        \item If $V_j'$ is a part of $\pi'$ of size $u$, then 
        \[z(V_j', \pi) \le 2\quad \text{or} \quad z(V_j', \pi) \ge u-2(\alpha-u)-1.\]
        \item If $z(V_j', \pi) = 2$ for some part $V_j'$ of $\pi'$, then there are parts $V_{i_1}$ and $V_{i_2}$ of $\pi$ such that 
        \[|V_j' \cap V_{i_1}|=1\quad \text{and} \quad |V_j'\cap V_{i_2}|=|V_j'|-1 .\]
        \item There are at most $\ln^3 n$ parts $V_j'$ of $\pi'$ such that 
        \[z(V_j', \pi)=2, \quad \text{and} \quad \text{there is a part $V_i$ of $\pi$ such that } |V_j' \cap V_i| \ge |V_i|-1.\] 
    \end{enumerate}
    \item All of the above also hold with $\pi$ and $\pi'$ swapped.
\end{enumerate}
    Let $\Pi_\mathrm{rel}= \Pi_\mathrm{rel}(\kp)$ denote the set of all relevant pairs  $(\pi, \pi')$ of vertex partitions.
\end{definition} 
		
		The crucial point of this definition is that $A_\pi \cap B_{\pi}\cap C_\pi \cap D_\pi$ and $A_{\pi'} \cap B_{\pi'}\cap C_{\pi'} \cap D_{\pi'}$ can only hold at the same time if $(\pi, \pi') \in \Pi_{\mathrm{rel}}$. Therefore, 
		\[
		\ex{Z_\kp^2}= \sum_{\pi,\pi'} \ex {\xi_{\pi} \xi_{\pi'}}  = \sum_{(\pi,\pi') \in \Pi_\mathrm{rel}} \ex {\xi_{\pi} \xi_{\pi'}} \le \sum_{(\pi,\pi') \in \Pi_{\mathrm{rel}}} \Pb\left(A_{\pi} \cap A_{\pi'}\right).
		\]
		So to verify Proposition~\ref{secondmomentgoal}, it suffices to prove that under the conditions of Theorem~\ref{theorem:general},
		\begin{equation}
			\frac{1}{\ex{Z_\textbf{k}}^2} \sum_{(\pi, \pi') \in \Pi_{\mathrm{rel}}} \Pb\left(A_{\pi} \cap A_{\pi'}\right) \le \exp\left(\frac{k_a^2}{\mu_{a}}+\bigO(\MB)+o(1)\right). \label{eq:assum}
		\end{equation}
The following observations are immediate from Definition \ref{def:tamepairs} and the fact that all parts of the partitions we consider are of size between $u^*$ and $a$, where $u^* \sim a \sim 2 \log_b n$. Recall that we say that a set $S$ is $z$-composed (with respect to $\pi$) if $z(S,\pi)=z$ as in Definition~\ref{def:tcomposed}.
\begin{fact}
\label{tamelemma}
Suppose that $(\pi, \pi') \in \Pi_\mathrm{rel}$, and that $V_i$ is a part of $\pi$. Then exactly one of the following five cases applies.
    \begin{enumerate}[a)]
    \itemsep0pt
        \item $V_i=V'_j$ for some part $V'_j$ of $\pi'$. \label{case1}
        \item \label{case5} $V_i$ is at least $(u-2(\alpha-u)-1)$-composed (with respect to $\pi'$), where $u=|V_i|$. 
        \item \label{case2} $V_i$ is $1$-composed  and $V_i \subset V'_j$ for some part $V'_j$ of $\pi'$, and $|V_j'\setminus V_i|=1$.
        \item \label{case3} $V_i$ is $2$-composed and $V_i \supset V'_j$ for some part $V'_j$ of $\pi'$, and $|V_i\setminus V_j'|=1$.
        \item \label{case4} There is a part $V_j'$ of $\pi'$ such that $|V_j'\setminus V_i|=|V_i\setminus V_j'|=1$ (so both $V_i$, $V_j'$ are $2$-composed).
    \end{enumerate}
\end{fact}
If case \eqref{case1} of Fact~\ref{tamelemma} applies, we say that $V_i$ is \emph{identical} (with respect to $(\pi, \pi')$). In case \eqref{case5}, we say that $V_i$ is \emph{scrambled} (with respect to $(\pi, \pi')$). If one of the cases \eqref{case2}, \eqref{case3} or \eqref{case4} applies, we call $V_i$ \emph{exceptional} (with respect to $(\pi, \pi')$). It follows from Definition~\ref{def:tamepairs} that $\pi$ has at most $2 \ln^3n$ exceptional parts. Finally, if $V_i$ is either scrambled or exceptional (i.e., if it is not identical), we say that $V_i$ is \emph{transmuted}.

As $\pi$ and $\pi'$ have the same profile $\bfk$, the number of identical and transmuted parts of each size (with respect to the other partition) is the same in both $\pi$ and $\pi'$. So, given $(\pi, \pi') \in \Pi_\mathrm{rel}$, let us define the following quantities that all, without always mentioning it explicitly, depend directly on $(\pi, \pi')$. Set
\[
    \bfl = (\ell_u)_{u^* \le u \le a},
    ~~
    \ell_u =  \text{number of identical parts of size $u$},
    ~~
    \ell = \sum_{u^* \le u \le a} \ell_u.
\]
Moreover, we define the `normalised' versions of these quantities,
\[
    \lambda_u = \ell_u u/n,
    \quad
    \lambda = \sum_{u^* \le u \le a} \lambda_u.
\]
In addition to that, we set
\[
    t_u =  \text{ number of transmuted parts of size $u$},
    \quad
    t = \sum_{u^*\le u \le a}  t_u,
\]
and write
\[
    s = \text{total number of scrambled parts}.
\]
Note that
\[
    \ell_u+t_u
    = k_u\quad \text{and} \quad \ell+t=k.
\]
Furthermore, as there are at most $2 \ln^3 n$ exceptional parts,
\[
    s \le t \le s+ 2 \ln^3 n.
\]
Note that since
\[
    \sum_{u^* \le u \le a} \ell_u u 
    = \sum_{u^* \le u \le a} \lambda_u n
    =  \lambda n,
\]
the quantity $\lambda$ is the fraction of vertices contained in identical parts. The calculations for the proof of Proposition \ref{secondmomentgoal} will be split up into three cases depending on the value of~$\lambda$. Recall the constant $c \in (0,1)$ from the Definition~\ref{deftame} of the tame sequence $\bfk$, set $c_0=c/3 \in (0, \frac 13)$, and let
\begin{equation}
\label{eq:defofscrambledetc}
\begin{split}
    \Pi_{\text{scrambled}}
    &= \left\{ (\pi, \pi') \in \Pi_{\mathrm{rel}} \mid \lambda(\pi, \pi') < {\ln^{-3} n} \right\},  \\
    \Pi_{\text{middle}}
    &= \left\{ (\pi, \pi') \in \Pi_{\mathrm{rel}} \mid  {\ln^{-3} n} \le \lambda(\pi, \pi') \le 1-n^{-{c_0}}\right\},  \\
    \Pi_{\text{similar}}
    &= \left\{ (\pi, \pi') \in \Pi_{\mathrm{rel}} \mid \lambda(\pi, \pi') > 1- n^{-{c_0}}\right\}.
\end{split}    
\end{equation}
Then the three lemmas below readily  imply \eqref{eq:assum} and thus Proposition \ref{secondmomentgoal}. Some of the lemmas here, as well as later on, are formulated in a slightly more general way than what is needed in the proof  Theorem~\ref{theorem:general}. In particular, those lemmas do not assume the additional conditions \eqref{eq:lowerboundmu},  \eqref{eq:lowerboundexpectation} and \eqref{eq:lowerboundbeta} from Theorem~\ref{theorem:general} unless stated explicitly. This does not complicate their proofs, but makes them more readily available for reuse in future work.
\begin{lemma}
\label{lemmascrambled} Let $a=\alpha_0-\bigO(1)$ and let $\kp$ be a tame $a$-bounded profile, then
		\begin{equation*}\frac{1}{\ex{Z_\kp}^2}\sum_{(\pi,\pi') \in \Pi_{\textup{scrambled}}} \Pb\left(A_{\pi} \cap A_{\pi'}\right) \lesssim \exp \left(  \frac{k_a^2}{\mu_{a}}+\bigO(M)\right)
		\end{equation*}
with $M=\MB+\MA$, where 
\begin{align}
    \MB = \MB(n) &=  \frac{k_a^4 \ln^2 n }{n \mu_{a}^2}, \nonumber\\
    \MA =  \MA(n) &=\frac{n}{\mu_{a} \ln n } + \frac{k_a^2 \ln^3 n+k_a k_{a-1} \ln^2 n}{ \mu_{a}n}+\frac{ ( k_{a-2} \ln n + k_{a-1} \ln^2 n+k_a \ln^3 n)^2 }{\mu_{a}n^2}. \label{eq:defM}
\end{align}
 \end{lemma}
	Note that in the context of Theorem~\ref{theorem:general}, we have $\mu_a \ge n^{1+\epsilon}$, and so $\MA=o(1)$ since $k_u \le k = \bigO(n/\ln n)$ for all $u$.
\begin{lemma} \label{lemmamiddle} 
Let $a=\alpha_0-\bigO(1)$ and let $\kp$ be a tame $a$-bounded profile. Furthermore, suppose that conditions \eqref{eq:lowerboundmu} and \eqref{eq:lowerboundbeta} from Theorem~\ref{theorem:general} hold. Then
\[
    \frac{1}{\ex{Z_\kp}^2}\sum_{(\pi,\pi') \in \Pi_{\textup{middle}}} \Pb\left(A_{\pi} \cap A_{\pi'}\right) =o(1).
\]
\end{lemma}
	
\begin{lemma}
\label{lemmasimilar}
Let $a=\alpha_0-\bigO(1)$ and let $\kp$ be a tame $a$-bounded profile, 
then
\[
    \frac{1}{\ex{Z_\kp}^2}\sum_{(\pi,\pi') \in \Pi_{\textup{similar}}} \Pb\left(A_{\pi} \cap A_{\pi'}\right) \le  \frac{n^{\bigO(1)}}{\E[\bar X_\bfk]} \enspace .
\]
\end{lemma} 
Note that in Theorem~\ref{theorem:general}, \eqref{eq:lowerboundexpectation}, we assume that $\E[\bar X_\bfk] \gg \ln n$, so in that case the bound from Lemma~\ref{lemmasimilar} is $o(1)$. Lemmas \ref{lemmascrambled} and \ref{lemmamiddle} are the most difficult to prove; this is accomplished in \S\ref{sectionscrambled} after some quick preliminaries in \S \ref{section:smprelims}. We prove Lemma~\ref{lemmasimilar} in \S\ref{sectionsimilar} with a short standalone argument.

\subsection{General preliminaries}
\label{section:smprelims}

Let $\bfk$ be a tame colouring profile.  Recall from Lemma~\ref{lem:expxk} that 
$\ex{X_\kp} = P {{N-f \choose m}}/{{N \choose m}}$, where
\begin{equation}
\label{defPf} 
    P = P_\kp = \frac{n!}{\prod_{u^* \le u \le a} u!^{k_u}}
    \quad\text{and}\quad
    f=f_\mathbf{k}= \sum_{u^* \le u \le a} {u \choose 2} k_u.
\end{equation}
Here $P$ denotes the number of ordered vertex partitions of profile $\kp$, and $f$ is the number of \emph{forbidden edges}, that is, edges with both endpoints  in the same part of any partition with profile $\bfk$. From Proposition~\ref{lem:XkZk} and Lemma~\ref{gnmlemma}, since $f = \bigO(n \ln n)$ for any tame profile, we obtain that
\begin{align}\label{expectation}
    \ex{Z_\kp} \sim \ex{X_\kp} = P \frac{{N-f \choose m}}{{N \choose m}} \sim P \, q^{f} \exp \left( -\frac{(b-1)f^2}{n^2} \right).
\end{align}
For two partitions $\pi$ and $\pi'$ of profile $\kp$, let
\[g=g_{\pi, \pi'}\]
denote the number of forbidden edges that $\pi$ and $\pi'$ have in common, that is, the number of pairs of vertices that are in the same part in both $\pi$ and $\pi'$. Using Lemma~\ref{gnmlemma} as before, the probability that both $\pi$ and $\pi'$ induce a valid colouring
is  
\begin{align}\label{jointprobability} 
    \Pb\left(A_{\pi} \cap A_{\pi'}\right) = {{{N-2f+g} \choose {m}}}/{{N \choose m}} \sim q^{2f-g} \exp\left(-\frac{(b-1)(2f-g)^2}{n^2}\right).
\end{align}

\subsection{Proof of Lemmas \ref{lemmascrambled} and \ref{lemmamiddle}} \label{sectionscrambled}
	
	Throughout \S\ref{sectionscrambled}, we assume that $\kp$ is tame and $a=\alpha-\bigO(1)$, but we do not assume the additional conditions \eqref{eq:lowerboundmu}, \eqref{eq:lowerboundexpectation} and  \eqref{eq:lowerboundbeta} from Theorem~\ref{theorem:general} unless explicitly stated.

\subsubsection{Notation and main lemma}
\label{section:notationsecondmoment}	

Recall from the start of \S \ref{section:secondmoment} that given $(\pi,\pi') \in \Pi_{\mathrm{rel}}$, we denote by $\bfl=(\ell_u)_u$ and $\bft=(t_u)_u$ the numbers of \emph{identical} and \emph{transmuted} parts of size $u$, respectively, and let $\ell = \sum_u \ell_u$ and $t=\sum_u t_u$. Then $k_u=\ell_u+t_u$ for all $u$, and $k=\ell+t$. 
In order to prove Lemma \ref{lemmascrambled} we need a fine-grained description of how two partitions~$\pi$,~$\pi'$ overlap. So, in addition to $\bfl=(\ell_u)_u$, we define the \emph{overlap sequence} $\bfr = \bfr(\pi, \pi')$ encoding the overlap between $\pi$ and $\pi'$ within the transmuted parts.
	
\begin{definition} Let $(\pi, \pi') \in \Pi_{\mathrm{rel}}$, where $\pi=(V_1, \dots, V_k)$ and $\pi'=(V_1', \dots, V_k')$.
\begin{enumerate}[a)]
\itemsep0mm
    \item An \emph{overlap block} is a non-empty maximal set of vertices that are in the same transmuted part of $\pi$ and in the same transmuted part of $\pi'$. In other words, an overlap block is the non-empty intersection of two transmuted parts of $\pi$ and $\pi'$, respectively.
    \item Let $r_x^{u,v}$ denote the \emph{number} of overlap blocks of size exactly $x$ given by the intersection of a part of size $u$ in $\pi$ and and a part of size $v$ in $\pi'$. Formally, 
    \begin{align*}
        r_x^{u,v} = \Big| \Big\{(V_i, V'_j) \,\,\big|\,\,& \text{$V_i$ transmuted part of $\pi$, $V_j'$ transmuted part of $\pi'$, }  \\
        &   |V_i|=u, \,\, |V_j'|=v, \,\, |V_i \cap V_j'|=x \Big\} \Big|.
    \end{align*}
    \item For $x \ge 1$, let
    \begin{align*}
        r_x = \sum_{u,v}r_x^{u,v}.
    \end{align*}
    \item We call $\bfr=\bfr(\pi, \pi')=(r_x^{u,v})_{x,u,v}$ the \emph{overlap sequence} of the pair $(\pi,\pi')$.
\end{enumerate}
\end{definition}
We make several simple observations, some of which are summarised in the lemma below. Suppose that $(\pi, \pi') \in \Pi_{\mathrm{rel}}$.
First, note that a $z$-composed transmuted part $V_i$ of $\pi$ contains exactly $z$ pairwise disjoint overlap blocks.
In particular, if $V_i$ is scrambled and $|V_i|=u$, then $V_i$ contains at least $u-2(\alpha-u)-1$ disjoint overlap blocks.
So any overlap block in a scrambled part contains at most $2(\alpha-u) +2 \le 4(\alpha-u^*)$ vertices.
	
Any overlap block consisting of more than $4(\alpha-u^*)$ vertices is in an exceptional (transmuted but not scrambled) part $V_i$ of $\pi$, and contains $|V_i|\ge u^*$ or $|V_i|-1 \ge u^*-1$ vertices. Each exceptional part contains only one such block, so as there are at most $2 \ln^3 n$ exceptional parts, there are at most $2\ln^2 n$ overlap blocks with more than $4(\alpha-u^*)$ vertices. 
	
There are no overlap blocks of size $a$, because such an overlap block would  have to form a complete part in both $\pi$ and $\pi'$, but then this part would be identical, not transmuted. For the same reason, an overlap block of size $a-1$ cannot be contained in parts of size $a-1$ in both $\pi$ and $\pi'$.

\begin{lemma}\label{lem:ra-1}
Suppose that $(\pi, \pi')\in \Pi_{\mathrm{rel}}$. Then for all $4(\alpha-u^*)<x<u^*-1$, we have $r_x=0$. Furthermore,
        \[ \pushQED{\qed} 
        \sum_{x\ge u^*-1} r_x \le 2 \ln^3 n \quad \text{and} \quad   
    r_a = r_{a-1}^{a-1,a-1}=0. \qedhere
    \popQED
\]
\end{lemma}
We now organise pairs $(\pi, \pi')$ of partitions according to the sequences $\bfl$, $\bfr$.
\begin{definition}
Suppose that $\bfl=(\ell_u)_u$ is a possible sequence encoding the numbers of identical parts of size $u$, and that $\bfr=(r_x^{u,v})_{x,u,v}$ is a possible overlap sequence. Then we let
\[
    \Pi_{\bfl, \bfr}
    = \left\{ (\pi, \pi')\in \Pi_\mathrm{rel} \mid \text{ $\pi$ and $\pi'$ overlap according to $\bfl$ and $\bfr$} \right\}
\]
and
\[
    P_{\bfl, \bfr} = |\Pi_{\bfl, \bfr}|  .   
\]
\end{definition}
	
If $(\pi,\pi') \in \Pi_{\bfl, \bfr}$, they share exactly
\begin{align}
\label{defg}
    g= g_{\bfl, \bfr}  = \sum_{u=u^*}^a {u \choose 2}\ell_u + \sum_{z=2}^{a-1} {z \choose 2} r_z
\end{align}
forbidden edges. Combined with \eqref{expectation} and \eqref{jointprobability}, this  implies that
\begin{equation}
\label{umformung}
\begin{split}
\frac{1}{\ex{Z_\kp}^2}\sum_{(\pi,\pi') \in \Pi_{\bfl, \bfr}} \Pb\left(A_{\pi} \cap A_{\pi'}\right) &\sim  \frac{P_{\bfl, \bfr}}{P^2}\, q^{-g_{\bfl, \bfr}}\, \exp \Big(-\frac{(b-1)}{n^2} \left((2f-g_{\bfl, \bfr})^2-2f^2\right) \Big) \\
& \le \frac{P_{\bfl, \bfr}}{P^2}\, q^{-g_{\bfl, \bfr}}\, \exp \Big(-\frac{2(b-1)f}{n^2} \left(f-2g_{\bfl, \bfr}\right) \Big) \enspace . 
\end{split}    
\end{equation}
Given $\bfl$ and $\bfr$, we will also use the following notation:
	\begin{align*}
		\nid &= \sum_u u\ell_u = \lambda n, \,\,\text{ the number of vertices in identical parts}\\
		\ntr&=\sum_u ut_u = \sum_{x\ge 1} x r_x =  n-n_\mathrm{id} = (1-\lambda)n, \,\, \text{ the number of vertices in transmuted parts}\\
		\eta &= 1- \frac{r_1}{\ntr}, \,\, \text{ the fraction of such vertices in overlap blocks of size at least $2$}\\
		\gid &= \sum_u \ell_u{ u \choose 2}, \,\, \text{ the number of shared forbidden edges in identical parts}\\
		\gt & = \sum_z r_z {z \choose 2}, \,\, \text{ the number of shared forbidden edges in transmuted parts.}
	\end{align*}
	Note that $n=\nid+\ntr$ and $g=\gid+\gt$. By \eqref{eq:defofscrambledetc}, for $(\pi, \pi') \in \Pi_\mathrm{middle} \cup \Pi_\mathrm{scrambled}$ with overlap $\bfl$ and $\bfr$, 
\begin{equation}\label{eq:ntrbound1}
    \lambda = \frac{\nid}{n} \le 1-n^{-c_0},
    \quad \text{ or equivalently }
    \quad \ntr>n^{1-c_0}.
\end{equation}
We are now ready to state the main lemma which bounds the contribution to (\ref{secondmomentgoal}) from all pairs $(\pi, \pi') \in \Pi_{\bfl, \bfr}$ with $\bfr$, $\bfl$ so that  $\ntr>n^{1-c_0}$.
\begin{lemma}
\label{contributionlemma}
With the notation in this section, let $\bfl=\bfl(n)$ and $\bfr=\bfr(n)$ be possible overlap sequences of a relevant pair of partitions such that $\ntr \ge n^{1-c_0}$.  Then
\begin{align}
        \frac{1}{\ex{Z_\kp}^2}\sum_{(\pi,\pi') \in \Pi_{\bfl, \bfr}} \Pb\left(A_{\pi} \cap A_{\pi'}\right) \,\, \lesssim \,\,\,& F_1(\bfl) F_2(\bfl,\bfr)  \exp \parenth{F_3\parenth{\bfl, \bfr}}, \label{asldkjfalk}
    \end{align}
    where
    \begin{align*}
        F_1(\bfl) \,\,&= \,\,   \frac{\prod_u{k_u \choose \ell_u}^2}
        { {
                \E_p \left[{\bar{X}_{\boldsymbol{\bfl}}}\right]
        } },
        \\
        F_2(\bfl, \bfr) \,\,&=\,\, \prod_{x \ge 2, u,v} \frac{T(x,u,v)^{r^{u,v}_x}}{r^{u,v}_x!}\, \text{ with }\, T(x,u,v)= \frac{{t_ut_v{{u \choose x}}{v \choose x}}e^{x \eta}}{{\ntr \choose x}q^{x \choose 2}} \,\text{ and }\, \eta=\frac{\ntr-r_1}{\ntr},\\
        F_3\parenth{\bfl, \bfr} \,\,&=\,\, -\frac{2(b-1)f \left(f-2g\right)+2 (f -  \gid- a(\ntr-r_1))^2}{n^2} .
    \end{align*}
\end{lemma}
We prove this lemma in \S\ref{sectionlemmaproof} below. Very roughly speaking, the term $F_1$ quantifies the contribution from identical parts and $F_2$ the contribution from smaller overlaps in scrambled parts. The term $\exp(F_3)$ contains some corrective terms, for example from calculating in $G_{n,m}$ rather than $G_{n,p}$, which will eventually cancel out with the summed contributions in $F_2$ from overlap blocks of size $2$, that is, joint forbidden edges of $\pi$ and $\pi'$ (which is the reason why we work in $G_{n,m}$ rather than in $G_{n,p}$).   
Subsequently, in \S\ref{summinssection}, we sum \eqref{asldkjfalk} over all possible sequences~$\bfl$, $\bfr$ corresponding to $\Pi_\mathrm{middle}$ and $\Pi_{\mathrm{scrambled}}$ in order to obtain Lemmas~\ref{lemmascrambled} and~\ref{lemmamiddle}.
	
\subsubsection{Proof of Lemma \ref{contributionlemma}} \label{sectionlemmaproof}
We start by counting $P_{\bfl, \bfr}$. Given some pair $(\pi, \pi')$ of partitions with $\pi=(V_1, \dots, V_k)$ and $\pi'=(V_1', \dots, V_k')$, we define the $k \times k$ \textit{overlap matrix} $\mc{M}= \mc{M}(\pi, \pi')=\left( M_{ij} \right)_{1 \le i, j, \le k}$ by letting
	\[
	M_{ij}=|V_i \cap V'_j|.
	\]
Note that the entries in the $i^\mathrm{th}$ row of $\mc{M}$ sum to $|V_i|$, and the entries in the $j^\mathrm{th}$ column sum to~$|V'_j|$. So for every $u^* \le u \le a$, exactly $k_u$ rows and columns of $\mc{M}$ sum to $u$.  Furthermore, the row and column sums decrease in size (as we defined ordered partitions as having decreasing part sizes $|V_1| \ge |V_2| \ge ... \ge |V_k|$), so the first $k_a$ rows and columns sum to $a$, the next $k_{a-1}$ rows and columns sum to $a-1$ and so on. The sum of all entries in the matrix is $n$.
	
	For every $u^* \le u \le a$, $(\pi, \pi')$ has exactly $\ell_u$ identical parts of size $u$. This corresponds to~$\ell_u$ entries of the number $u$ in $\mc{M}$ so that all other entries in the same row or column are $0$. In addition, for any $1 \le x \le a-1$ there are $r_x$ entries of the number $x$. All other entries of $\mc{M}$ are $0$.
	
	Now conversely, given a matrix with these properties, the number of corresponding pairs $(\pi$, $\pi')$ of partitions is given by the multinomial coefficient 
	\begin{equation*}
		\frac{n!}{ \prod_u u!^{\ell_u}\prod_x x!^{r_x}}.
	\end{equation*}
	This is because for each positive entry $M_{ij}$ of $\mc{M}$, we may choose a vertex set of size $M_{ij}$ for $V_i \cap V'_j$, and this uniquely defines $\pi$ and $\pi'$. So let $M_{\bfl,\bfr}$ denote the number of matrices with row and column sums according to $\bfk$ and entries corresponding to $\bfl$ and $\bfr$ and with the properties above, then
\begin{equation}
\label{relationpm}
    P_{\bfl,\bfr} \le \frac{n!}{\prod_u u!^{\ell_u}\prod_x x!^{r_x} } \, M_{\bfl,\bfr}.
\end{equation}
Note that this is an upper bound because not every matrix $\mc{M}$ with the properties above corresponds to relevant pairs $(\pi, \pi')$, which fulfil some additional conditions.

\begin{lemma}
\label{matrixlemma}
With the notation in this section,
\begin{align*}
    M_{\bfl,\bfr} \lesssim r_1! \, &\prod_u \left({k_u \choose \ell_u}^2\frac{\ell_u!}{{u!^{2t_u}}} \right)  \prod_{x \ge 2, u,v} \frac{\left(t_u t_v {{u \choose x}{v \choose x}x!^2}\right)^{r^{u,v}_x}}{r^{u,v}_x!}  \exp\left( -\frac{2 \big(f -  {\gid}- a{(\ntr-r_1)}\big)^2}{n^2}\right).
\end{align*}
\end{lemma}	
\begin{proof}
Start with an empty $k \times k $ matrix. We first choose places for the $\ell_u$ entries of the number $u$, $u^*\le u \le a$, and write $0$ for all other entries in the same row or column. Each entry $u$ has to go into one of the $k_u$ rows and columns which sum to $u$. As there cannot be two such entries $u$, $v$ in the same row or column, there are exactly
\begin{align}\label{placelargest}
    \prod_u {k_u \choose \ell_u}^2 \ell_u!
\end{align}
choices. All other non-zero entries must be placed in the other $t=k-\ell$ 
rows and columns. Recall that $t_u=k_u-\ell_u$. We next place $r_x^{u,v}$ entries of the number $x$ for all $2 \le x \le a-1$, for which there are at most
\begin{align}\label{placelarge}
    \prod_{x\ge 2, u, v} {t_u t_v \choose r^{u,v}_x} &\le \prod_{x \ge 2, u,v} \frac{(t_u t_v)^{r^{u,v}_x}}{r^{u,v}_x!}
\end{align}
possibilities. All remaining entries are $0$ or $1$, and their placement needs to be in accordance with the given row and column sums of the matrix. To bound the number of choices, we use the following theorem of McKay~\cite{mckay1984asymptotics} 
that pins down the  asymptotic number of $0$-$1$ matrices with prescribed row and column sums.
\begin{theorem}[\cite{mckay1984asymptotics}]
    \label{mckaymatrices}
    Let $N(\boldsymbol{\sigma}, \boldsymbol{\tau})$ be the number of $m \times n$ 0-1 matrices with row sums $\boldsymbol{\sigma}=(\sigma_1, \dots, \sigma_m)$ and column sums $\boldsymbol{\tau}=(\tau_1, \dots, \tau_n)$. Let $S= \sum_{i=1}^m \sigma_x$, $\sigma = \max_i \sigma_i$, $\tau = \max_j \tau_j$, $S_2 = \sum_{i=1}^m \sigma_i (\sigma_i -1)$ and $T_2 = \sum_{j=1}^n \tau_j (\tau_j -1)$. If $S \rightarrow \infty$ and $1 \le \max\{\sigma,\tau\}^2 < dS$ for some $d < 1/6$, then
    \[
    N(\boldsymbol{\sigma}, \boldsymbol{\tau}) = \frac{S!}{\prod_{i=1}^m \sigma_i! \, \prod_{j=1}^n \tau_j!} \exp \left( - \frac{S_2 T_2}{2S^2} + \bigO \left( \frac{\max\{\sigma,\tau\}^4}{S}\right) \right).
    \]
\end{theorem}

In our setting, we want to count the number of $k\times k$ $0$-$1$-matrices where the row and column sums are given by the profile $\bfk$, minus the entries we have already placed in the matrix. We use the notation $\boldsymbol{\tau}=(\tau_1, \dots, \tau_k)$ and $\boldsymbol{\sigma}=(\sigma_1, \dots, \sigma_{k})$ for the prescribed row and column sums as in McKay's theorem, where the exact values for the $\sigma_i$ and $\tau_j$ depend on where we placed the larger entries in the matrix. Note that $ N(\boldsymbol{\sigma}, \boldsymbol{\tau}) $ overcounts the total number of ways to complete the matrix as we cannot place any entries $1$ where we have already written another number, but this will be insignificant.

The total number of entries of the number $1$ is exactly $r_1$ -- this is $S$ in the notation of Theorem \ref{mckaymatrices}. If $n$ is large enough then 
\begin{equation}
r_1 \ge n^{0.5} \label{eq:r1}
\end{equation}
 (because by assumption $\ntr > n^{1-c_0}=n^{1-c/3}$ where $c \in (0,1)$ is the constant from Definition~\ref{deftame}, and there are at most $2 \ln^3 n$ exceptional parts, so $\pi$ has at least $n^{0.5}$ scrambled parts, each of which contains overlap blocks of size $1$). 
Furthermore, all $\sigma_i$'s and $\tau_j$'s are at most $a=\bigO(\ln n)$. So  certainly $1 \le \max\{\sigma,\tau\}^2 < S/10$, say, and Theorem \ref{mckaymatrices} gives
\begin{align}\label{01matrix}
    N(\boldsymbol{\sigma}, \boldsymbol{\tau}) \sim \frac{r_1!}{\prod_{i=1}^k \sigma_i! \, \prod_{j=1}^k \tau_j!} \exp \left( - \frac{\sum_{i=1}^k \sigma_i (\sigma_i -1) \sum_{j=1}^k \tau_j (\tau_j -1)}{2r_1^2} \right).
\end{align}
The sequence $\sigma_1, \dots, \sigma_k$ can be obtained from the sequence $a, a, \dots, u^*, u^*$ ($k_u$ times $u$ for each $u^* \le u \le a$) in the following way. First, for each $u^* \le u \le a$ replace exactly $\ell_u$ members $u$ of the sequence by $0$ (this corresponds to the $\ell_u$ entries of the number $u$ in the matrix). Then for all $u^* \le u \le a$, successively subtract the number $2$ from $\sum_{u^*\le v \le a} r_2^{u,v}$ members of the sequence which were originally $u$, the number $3$ from $\sum_{u^* \le v \le a} r_3^{u,v}$ members of the sequence which were originally $u$, and so on (with repetition).

Accordingly, the product $\prod_{i=1}^{ k} \sigma_i!$ can be obtained from the product $\prod_{u^* \le u \le a} u!^{k_u}$ by removing the corresponding factors of the factorials. Then $\prod_{i=1}^{ k} \sigma_i!$ is minimised if these factors are as large as possible, which happens if each number is subtracted from a member of the sequence that is as large as possible. Recalling that $t_u=k_u-\ell_u$, this gives
\begin{align*}
    \prod_{i=1}^{ k} \sigma_i! &\ge \prod_{u} u!^{k_u-\ell_u} \prod_{x \ge 2, u, v}  \frac{\left( u-x\right)!^{r_x^{u,v}}}{u!^{r_x^{u,v}}}= \frac{\prod_{u} u!^{t_u}}{ \prod_{x \ge 2, u,v} \left(x!{{u\choose x}}\right)^{r^{u,v}_x}}.
\end{align*}
Of course there is a corresponding bound for $\boldsymbol{\tau}=(\tau_1, \dots, \tau_{k})$, so
\begin{align} \label{bound2i}
    \prod_{i=1}^{ k} \sigma_i! \prod_{j=1}^{ k} \tau_j! \ge \frac{\prod_{u} u!^{2t_u}}{ \prod_{x \ge 2, u,v} \left(x!^2{{u \choose x}{v \choose x}}\right)^{r^{u,v}_x}}.
\end{align}
A similar argument gives a lower bound  for $\sum_{i=1}^k \sigma_i (\sigma_i -1)$. Recalling that $f=\sum_u k_u{u \choose 2}$,
\begin{align*}
    \sum_{i=1}^k \sigma_i (\sigma_i -1)
    = 2 \sum_{i=1}^k  {\sigma_i \choose 2} 
    &\ge 2\sum_u {u \choose 2 }(k_u-\ell_u) - 2\sum_{x\ge2} r_x \parenth{{a \choose 2}-{a-x \choose 2}}\\
    &= 2f - 2 \sum_u {u \choose 2} \ell_u-  \sum_{x\ge2}r_x (2ax-x-x^2) \\
    &\ge 2f -  2 {\gid} - 2a \sum_{x\ge2} x r_x  = 2f -  2{\gid}- 2a{(\ntr-r_1)}.
\end{align*}
As the same lower bound also holds for $\sum_{j=1}^k \tau_j (\tau_j -1)$, we obtain that
\begin{align*}
    \sum_{i=1}^k \sigma_i (\sigma_i -1)\sum_{j=1}^k \tau_j (\tau_j -1)&\ge  4(f -  {\gid}- a{(\ntr-r_1)})^2 .
\end{align*}
Plugging this and (\ref{bound2i}) into (\ref{01matrix}) gives
\begin{align*}
    N(\boldsymbol{\sigma}, \boldsymbol{\tau}) \lesssim \frac{r_1!}{\prod_{u} u!^{2t_u}} \prod_{x \ge 2, u,v} \left(x!^2{{u \choose x}{v \choose x}}\right)^{r^{u,v}_x} \exp\left( -\frac{2 (f -  {\gid}- a{(\ntr-r_1)})^2}{r_1^2}\right).
\end{align*}
Together with  (\ref{placelargest}) and (\ref{placelarge}),  bounding $r_1 \le n$, we have \begin{align*}
    M_{\bfl,\bfr} \lesssim r_1! \, &\prod_u \left({k_u \choose \ell_u}^2\frac{\ell_u!}{{u!^{2t_u}}} \right)  \prod_{x \ge 2, u,v} \frac{\left(t_u t_v x!^2{{u \choose x}{v \choose x}}\right)^{r^{u,v}_x}}{r^{u,v}_x!}    \exp\left( -\frac{2 (f -  {\gid}- a{(\ntr-r_1)})^2}{n^2}\right),
\end{align*}
completing the proof of Lemma \ref{matrixlemma}.
\end{proof}	
We need one final estimate before combining everything into one big formula. Recall from \S\ref{section:notationsecondmoment} that
\[
    n=\nid+\ntr \, \text{  and }\, \eta
    = \frac{\ntr-r_1}{\ntr},
\]
and so by Lemma \ref{ablemma} (noting that $r_1 \rightarrow \infty$ by \eqref{eq:r1}),
\begin{equation}
\label{eq:sdl}
\begin{split}
    \frac{r_1!}{n!}
    &=\frac{\ntr!}{n!} \cdot \frac{(\ntr-{(\ntr-r_1)})!}{\ntr!} 
    \lesssim \frac{(n-{\nid})!}{n!}  \ntr^{-\ntr+r_1}e^{\eta {(\ntr-r_1)}} \\
    &\le \frac{(n-{\nid})!}{n!} \cdot \prod_{x \ge 2} \left(\frac{e^{\eta x} }{x!{\ntr \choose x}} \right)^{r_x}, 
\end{split}
\end{equation}
using in the last step that ${\ntr-r_1}=\sum_{x\ge 2} x r_x$, and that ${\ntr \choose x} \le n_{\mathrm{tr}}^x/x!$. Now we are ready to put everything together. We continue from (\ref{umformung}), using (\ref{relationpm}), Lemma \ref{matrixlemma}, (\ref{defPf}),  \eqref{eq:sdl} and the fact that $k_u=\ell_u+t_u$:
	
	\[\frac{1}{\ex{Z_\kp}^2}  \sum_{(\pi, \pi') \in \Pi_{\bfl, \bfr}} \Pb\left[A_{\pi} \cap A_{\pi'}\right] \sim q^{-g} \, \frac{P_{\bfl, \bfr}}{P^2}\, \exp \Big(-\frac{2(b-1)f}{n^2} \left(f-2g\Big) \right)\]
	\begin{align}
		\le & \, \,\, \, M_{\bfl,\bfr} \, \frac{ \prod_{u} u!^{2k_u-\ell_u}}{ n! \, q^{g}\,\prod_{x\ge 2} x!^{r_x}}\, \exp \left(-\frac{2(b-1)f}{n^2} \left(f-2g\right) \right)\nonumber\\
		\lesssim & \, \,\, \, \frac{r_1!}{n! \, q^g} \, \prod_u \left({k_u \choose \ell_u}^2 \ell_u!u!^{\ell_u}\right)  \prod_{x \ge 2, u,v} \frac{\left(x! t_u t_v {{u \choose x}{v \choose x}}\right)^{r^{u,v}_x}}{r^{u,v}_x!} \nonumber \\
		& \cdot  \exp\left( -\frac{2(b-1)f}{n^2} \left(f-2g\right)-\frac{2 (f -  {\gid}- a{(\ntr-r_1)})^2}{n^2}\right)  \nonumber\\
		\lesssim & \, \,\, \, F_1(\bfl) F_2(\bfl,\bfr)  \exp \parenth{F_3\parenth{\bfl, \bfr}} ,
	\end{align}
	as required, bounding $r_1 \le n-\nid$ and recalling that $\E_p[\bar{X}_\bfl]= \frac{n!}{(n-\nid)!} \prod_u \left( \frac{q^{{u \choose 2} \ell_u}}{\ell_u! u!^{\ell_u}}\right)$.
	\qed
	
\subsubsection{Summing the terms --- proof of Lemmas \ref{lemmascrambled} and \ref{lemmamiddle}}\label{summinssection}
		
\subsubsection*{Overview}
With Lemma~\ref{contributionlemma} at hand, we know almost exactly how much pairs $(\pi, \pi') \in \Pi_{\bfl, \bfr}$ contribute to the sums that we want to bound in Lemmas~\ref{lemmascrambled} and \ref{lemmamiddle}. We will now sum these contributions over all possible overlap sequences $\bfl$ and $\bfr$ encoding pairs in $\Pi_\mathrm{scrambled}\cup \Pi_\mathrm{middle}$.
	
Consider the expression $F_2(\bfl, \bfr)$ from Lemma~\ref{contributionlemma}. The form of the terms will make it easy to sum $F_2(\bfl, \bfr)$ over all integers $r_x$ using the well-known identity $\sum_{n\ge 0 }{z^n}/{n!} = e^z$. Roughly speaking, this will result in terms of the form $\exp(\sum_{u,v,x}T(x,u,v) )$. In view of this, we make the following definition.
	
Recall that by Lemma  \ref{lem:ra-1}, $r_{a} =r_{a-1}^{a-1,a-1}=0$, which is why we will not need to take sums over these parameters. Furthermore, note that the terms $T(x,u,v)$ only depend on $\bfl$ (via $t_u$ and $t_v$) and $r_1$ (via $\eta = (\ntr - r_1)/ \ntr$, where $\ntr = \sum_u ut_u= (1-\lambda)n$).
	\begin{definition}\label{def:Tsums}
		For $2 \le x \le a-2$, let
		\[
		T(x) = T_{\bfl,r_1} (x) = \sum _{u^* \le u,v \le a} T(x,u,v) = \frac{e^{x\eta}\left(\sum_{u} t_u {u \choose x}  \right)^2}{{\ntr \choose x}q^{x \choose 2}}.
		\]
		Furthermore, let
		\begin{align*}
			T(a-1) &= T(a-1,a,a)+T(a-1,a,a-1)+T(a-1,a-1,a) 
			= \frac{e^{(a-1)\eta} \big(a^2 t_{a}^2+2at_{a}t_{a-1} \big)}{{\ntr\choose a-1} q^{a-1 \choose 2}}.
		\end{align*} 
	\end{definition}
	We continue with some technical lemmas bounding various terms.  
All proofs are straightforward calculations that can be found in the appendix. 
\begin{lemma}
\label{easylemma1}
Let $\bfl$ and $\bfr$ encode the overlap of a pair $(\pi, \pi')\in \Pi_{\textrm{scrambled}} \cup \Pi_{\textrm{middle}}$, in particular such that $\lambda= {\sum_u \ell_u u}/{n} \le 1-n^{-c_0}$. Then the following holds uniformly for all such sequences $\bfl, \bfr$.
\begin{enumerate}[a)]
    \item $ F_3(\bfl, \bfr)= \bigO (\ln^2 n)$. \label{lemparta}
    \item $T(2)= \bigO(\ln^2 n)$. \label{lempartb1}
    \item Let $3 \le x \le 4(\alpha-u^*)$ and recall that $t=\sum_u t_u = k-\ell$. Then \label{lempartb}
    \[
        T(x)  = \bigO \left( \frac{\ln^3 n}{t}\right)=o(1)
        ~\text{ and }~
        \sum_{x=3}^{4(\alpha-u^*)} T(x) =o(1).
    \] 
    \item \[\prod_{x >4(\alpha-u^*), u,v} \frac{T(x,u,v)^{r^{u,v}_x}}{r^{u,v}_x!}  \le \exp\left( \bigO(\ln ^5 n)\right).\label{lempartc}\]
\end{enumerate}
\end{lemma}
The next lemma will be used in the case $\lambda < \ln^{-3}n$, where we have to be more precise. 
\begin{lemma} \label{easylemma2}
If $\bfl$, $\bfr$ are such that $\lambda,\eta < \ln^{-3}n$, then, uniformly,
		\begin{enumerate}[a)]
			\item $F_3(\bfl, \bfr) = -2b f^2/{n^2}+o(1)$, \label{teil1}
			\item $T(2)= {2b f^2}/{n^2}+o(1)$, \label{teil2}
			\item  $
			\sum_{x= u^*-1}^{a-3} T(x) \le \frac{n^{-1+o(1)}}{\mu_a} = o(1),
			$\label{teil3} 
			\item $T(a-2)= \bigO \left(  \frac{ ( k_{a-2} \ln n + k_{a-1} \ln^2 n+k_a \ln^3 n)^2 }{\mu_{a}n^2} \right)$ \label{teil4a},
			\item $T(a-1) = \bigO \left( \frac{k_{a}^2 \ln^3 n+k_{a} k_{a-1} \ln^2 n}{n\mu_{a}}\right)$. \label{teil4}
		\end{enumerate}
In particular, together with Lemma~\ref{easylemma1} and recalling the definition \eqref{eq:defM} of $\MA$, 
\[
    \sum_{x=3}^{4(\alpha-u^*)}T(x)+\sum_{x=u^*-1}^{a-1}T(x)=\bigO(\MA)+o(1).
\]
\end{lemma}

\subsubsection*{The middle case --- proof of Lemma \ref{lemmamiddle} }
	
Suppose that $\bfl\le \bfk$ is a sequence such that $\lambda \le 1-n^{-c_0}$, where as usual $\lambda = \sum_u \lambda_u = \sum_u {\ell_u u}/{n}$. From Lemma~\ref{contributionlemma}, if we fix $\bfl$ and sum (\ref{asldkjfalk}) over all possible overlap sequences $\bfr$ encoding the overlap of a relevant pair, using Lemma \ref{easylemma1} and the identity $\sum_{n\ge 0} {z^n}/{n!}= e^z$, 
	\begin{align}
		\sum_{\bfr \text{ relevant}}  & \Big(F_1(\bfl) F_2((\bfl,\bfr)) \exp \parenth{F_3\parenth{\bfl, \bfr}} \Big)  \le F_1(\bfl)  \sum_{\bfr \text{ relevant}}  \left(\prod_{x\ge 2,u,v} \frac{T(x,u,v)^{r^{u,v}_x}}{r^{u,v}_x!} \right) \exp (\bigO(\ln^2 n)) \nonumber\\
		& \le F_1(\bfl)  \sum_{\bfr \text{ relevant}}  \left(\prod_{2\le x \le 4(\alpha-u^*),u,v} \frac{T(x,u,v)^{r^{u,v}_x}}{r^{u,v}_x!} \right) \exp (\bigO(\ln^5 n)) \nonumber\\
		&\le F_1(\bfl)\left(\prod_{2\le x \le 4(\alpha-u^*),u,v} \left(\sum_{r_x^{u,v}=0}^\infty  \frac{T(x,u,v)^{r^{u,v}_x}}{r^{u,v}_x!} \right)\right)\exp (\bigO(\ln^5 n))\nonumber \\
		&= F_1(\bfl) \exp \left(  \sum_{x=2}^{4(\alpha-u^*)} T(x)  + \bigO(\ln^5 n) \right) \le \frac{\prod_u{k_u \choose \ell_u}^2}
		{ {
				\E_p \left[\bar{X}_{\boldsymbol{\bfl}}\right]
		} } \exp \left( \bigO(\ln^ 5 n) \right). \label{eq:umform}
\end{align} 
We will use this relationship, together with Lemma~\ref{contributionlemma}, to prove the following three lemmas in the appendix. Together with the additional conditions \eqref{eq:lowerboundmu} and \eqref{eq:lowerboundbeta} from Theorem~\ref{theorem:general}, they readily imply Lemma \ref{lemmamiddle} (by setting $\beta(n)=\exp(\ln^6 n)$ in Lemma~\ref{lemma:newlemma2}, so that \eqref{eq:betadelta} is given by \eqref{eq:lowerboundbeta}).
\begin{lemma}
\label{lemma:newlemma1}
Let $\epsilon>0$ be constant and let $a=a(n)=\alpha_0-\bigO(1)$ be an integer so that $\mu_a \ge n^{1+\epsilon}$. Let~$\kp$ be a tame $a$-bounded sequence of profiles. Then there is a constant $C(\epsilon)>0$ so that 
\[
    \frac{1}{\ex{Z_\kp}^2}\sum_{(\pi,\pi') \in \Pi_{\textup{middle,1}}} \Pb\left(A_{\pi} \cap A_{\pi'}\right) 
    \le \exp \Big( - \Theta\Big( \frac{n} {\ln^3 n}\Big) \Big),
\]
where
\[
    \Pi_{\textup{middle,1}} = \left\{ (\pi, \pi') \in \Pi_{\textup{middle}} \mid {\ln^{-3} n} \le \lambda \le C(\epsilon)\right\}.
\]
\end{lemma}

\begin{lemma}\label{lemma:newlemma2}	Let $a=a(n)=\alpha_0-\bigO(1)$ be an integer, $\kp$ a tame $a$-bounded sequence of profiles, and $\delta>0$ a constant. Furthermore, suppose that $\beta(n)$ is a function so that for all colouring profiles $\boldlambda \le \boldkappa$ such that ${\delta} \le \sum_{1 \le u \le a} \lambda_u \le 1-{\delta}$, 
	\begin{equation}\label{eq:betadelta}
		\E_p\left[\bar{X}_{\boldlambda}\right] \ge \beta(n) \prod_{1 \le u \le a}{k_u \choose \ell_u}^2 .    \end{equation}
	Then 
\[
    \frac{1}{\ex{Z_\kp}^2}\sum_{(\pi,\pi') \in \Pi_{\textup{middle,2}}} \Pb\left(A_{\pi} \cap A_{\pi'}\right) 
    \le \frac{1}{\beta(n)} \exp \big(\bigO(\ln^5 n) \big),
\]
where 
\[
    \Pi_{\textup{middle,2}} = \left\{ (\pi, \pi') \in \Pi_{\textup{middle}} \mid \delta \le \lambda \le 1-\delta\right\}.
\]
\end{lemma}
\begin{lemma}\label{lemma:newlemma3}
Let $a=a(n)=\alpha_0-\bigO(1)$ be an integer, and $\bfk$ a tame sequence of $a$-bounded colouring profiles. There is a constant $\delta_0 > 0$ so that 
\[
    \frac{1}{\ex{Z_\kp}^2}\sum_{(\pi,\pi') \in \Pi_{\textup{middle,3}}} \Pb\left(A_{\pi} \cap A_{\pi'}\right)
    \le \exp \Big(-\Theta(n^{1-c_0}) \Big),
\]
where
\[
\Pi_{\textup{middle,3}} = \left\{ (\pi, \pi') \in \Pi_{\textup{middle}} \mid 1-\delta_0 \le \lambda \le 1-n^{-c_0}\right\},
\]
and $c_0=c/3$, where $c$ is the constant from the Definition \ref{deftame} of the tame sequence $\bfk$.
\end{lemma}

\subsubsection*{The scrambled case --- proof of Lemma \ref{lemmascrambled}}
		
We now come to the most delicate case $\lambda<{\ln^{-3} n}$. We will sum \eqref{asldkjfalk} over all sequences~$\bfl, \bfr$ encoding the overlap between pairs of partitions in $\Pi_{\text{scrambled}}$. We will start by fixing a sequence $\bfl$ such that $\lambda<{\ln^{-3} n}$, and sum over all $\bfr$ such that $(\bfl, \bfr)$ encodes the overlap between pairs of partitions in $\Pi_{\text{scrambled}}$. Thereafter we will sum the result over all possible sequences $\bfl$. So recall that $\eta = 1-r_1 / \ntr $ and let 
\begin{align*}
    R_1(\bfl)&= \big\{ \bfr \mid (\bfl, \bfr) \text{ encodes the overlap of a pair in $\Pi_{\text{scrambled}}$ and $\eta \ge {\ln^{-3} n}$}\big\} \\
    R_2(\bfl)&= \big\{ \bfr \mid (\bfl, \bfr) \text{ encodes the overlap of a pair in $\Pi_{\text{scrambled}}$ and $\eta < {\ln^{-3} n}$}\big\},
\end{align*}
then the following lemmas bound the sum over $\bfr$ of \eqref{asldkjfalk} for a fixed sequence $\bfl$.
\begin{lemma}\label{lemma:sumR1}
Let $\bfl$ be a possible sequence encoding the identical parts in a pair of partitions of profile $\kp$, and suppose that $\lambda = \sum_u \ell_u u/n < {\ln^{-3} n}$. Then
\[
    \frac{1}{\ex{Z_\kp}^2}\sum_{\bfr \in R_1(\bfl)}\sum_{(\pi,\pi') \in \Pi_{\bfl, \bfr}} \Pb\left(A_{\pi} \cap A_{\pi'}\right)
    = o(F_1(\bfl)).
\]
\end{lemma}
	
\begin{lemma}\label{lemma:sumR2}
Under the conditions of Lemma~\ref{lemma:sumR1}, and recalling the definition~\eqref{eq:defM} of $\MA$,
\[
    \frac{1}{\ex{Z_\kp}^2}\sum_{\bfr \in R_2(\bfl)}\sum_{(\pi,\pi') \in \Pi_{\bfl, \bfr}} \Pb\left(A_{\pi} \cap A_{\pi'}\right) \lesssim F_1(\bfl)\, \exp ( \bigO(\MA)).
\]
\end{lemma}
	
	\begin{proof}[Proof of Lemma~\ref{lemma:sumR1}] 
		Let $\bfr \in R_1(\bfl)$. Recall that
		\[
		\eta=1-\frac{r_1}{\ntr}= \frac{\sum_{x \ge 2} x r_x}{\ntr}=\frac{\sum_{x \ge 2,u,v} x r^{u,v}_x}{\ntr}.
		\]
		So if $\eta > {\ln^{-3} n}$, then as $\ntr=(1-\lambda)n \sim n$ and $u,v,x\le a=\bigO(\ln n)$, if $n$ is large enough there exist $x_0 \ge 2, u_0, v_0$ such that 
		\[
		r^{u_0,v_0}_{x_0} \ge {n}/{\ln^8 n}.
		\]
		By Lemma~\ref{lem:ra-1}, we have $\sum_{x>4(\alpha-u^*)} r_x \le 2 \ln^3 n$, so $x_0 \le 4(\alpha-u^*)$. Then, by Lemma \ref{easylemma1},
		\[
		T (x_0, u_0, v_0) \le T(x_0) \le \max_{2 \le x \le 4(\alpha-u^*)}T(x) = \bigO \left( \ln^2 n \right).
		\]
		Therefore, if $n$ is large enough,
\begin{equation}
\label{eq:ksdjhf}
    \frac{T (x_0, u_0, v_0)^{r^{u_0,v_0}_{x_0}}}{r^{u_0,v_0}_{x_0}!}\le \left(\frac{e T (x_0, u_0, v_0)}{r^{u_0,v_0}_{x_0}} \right)^{r^{u_0,v_0}_{x_0}}
    \le \left(\frac{\ln^{11} n}{n} \right)^{{n}/{\ln^8 n}} \le \exp\left( -\ln^6 n\right).
\end{equation}
		Using (\ref{asldkjfalk}), we wish to sum $F_1(\bfl) F_2(\bfl,\bfr)  \exp \parenth{F_3\parenth{\bfl, \bfr}}$ over all $\bfr \in R_1(\bfl)$. Using Lemma \ref{easylemma1}, we can bound $F_3(\bfl, \bfr)=\bigO (\ln^2 n)$, and furthermore, using \eqref{eq:ksdjhf}, Lemma \ref{easylemma1} and the bound $T^r/r! \le e^T$, uniformly in $\bfr$,
		\begin{align*}
			F_2(\bfl, \bfr) &= \frac{T (x_0, u_0, v_0)^{r^{u_0,v_0}_{x_0}}}{r^{u_0,v_0}_{x_0}!} \prod_{\substack{2\le x \le 4(\alpha-u^*), u,v \\ (x,u,v) \neq (x_0, u_0, v_0)}} \frac{T(x,u,v)^{r^{u,v}_x}}{r^{u,v}_x!} \prod_{x >4(\alpha-u^*), u,v} \frac{T(x,u,v)^{r^{u,v}_x}}{r^{u,v}_x!} \\
			&\le \exp\left( -\ln^6 n + \bigO(\ln^5 n)\right).
\end{align*}
Note that as $u, v, x \le a = \bigO(\ln n)$, and all $r_x^{u,v}$ are at most $k <n$, and using Lemma~\ref{easylemma1} again, we can bound crudely for sufficiently large $n$:
\begin{align*}
    \sum_{\bfr \in R_1(\bfl)} F_1(\bfl) F_2(\bfl,\bfr)  & \exp \parenth{F_3\parenth{\bfl, \bfr}} \le 	F_1(\bfl) \sum_{\bfr \in R_1(\bfl)}\exp\left(-\ln^6n+\bigO(\ln^5 n)\right)\\
    &\le F_1(\bfl) n^{a^3}\exp\left(-\ln^6n+\bigO(\ln^5 n)\right) 
    \le F_1(\bfl) \exp \left(-\frac{1}{2}\ln^6 n \right),
\end{align*}
and this, by \eqref{asldkjfalk}, implies the claim.
\end{proof}

\begin{proof}[Proof of Lemma~\ref{lemma:sumR2}]
Let $\bfr \in R_2(\bfl)$, so both $\lambda, \eta<{\ln^{-3} n}$. Similarly to the proof of Lemma~\ref{lemma:sumR1}, using \eqref{asldkjfalk} we need to sum $F_1(\bfl) F_2(\bfl,\bfr)  \exp \parenth{F_3\parenth{\bfl, \bfr}}$ over all $\bfr \in R_2(\bfl)$.
Using Lemmas~\ref{lem:ra-1} and \ref{easylemma2}, 
switching sums and products, and recalling that there are no overlap blocks of size between $4(\alpha-u^*)+1$ and $u^*-2$, we obtain
\begin{align*}
    \sum_{\bfr \in R_2(\bfl)} & F_1(\bfl) F_2(\bfl,\bfr)  \exp \parenth{F_3\parenth{\bfl, \bfr}} \sim F_1(\bfl)\exp\left( -\frac{2b f^2}{n^2}\right)\sum_{\bfr \in R_2(\bfl)}F_2(\bfl, \bfr)\\
    &\sim F_1(\bfl)\exp\left( -\frac{2b f^2}{n^2}\right)\sum_{\bfr \in R_2(\bfl)}\prod_{\substack{2 \le x \le 4(\alpha-u^*)  \\ \text{or } u^*-1 \le x \le a-1, \\  u,v}} \frac{T(x,u,v)^{r^{u,v}_x}}{r^{u,v}_x!} \\
    &\le  F_1(\bfl)\exp\left( -\frac{2b f^2}{n^2}\right) \prod_{\substack{2 \le x \le 4(\alpha-u^*)  \\ \text{or } u^*-1 \le x \le a-1, \\  u,v}}\left(\sum_{r_{x}^{u,v} \ge 0} \frac{T(x,u,v)^{r_x^{u,v}}}{r_x^{u,v}!}\right).
\end{align*}
Using the identity $\sum_{r=0}^\infty T^r/r! = e^T$ and Lemma~\ref{easylemma2} again, we obtain that this is at most
\begin{align*}
    \pushQED{\qed} 
    & F_1(\bfl) \exp \left(-\frac{2b f^2}{n^2} + \sum_{x=2}^{4(\alpha-u^*)}T(x)+\sum_{x=u^*-1}^{a-1} T(x) \right)\\
    \lesssim & ~ F_1(\bfl) \exp \left(-\frac{2b f^2}{n^2}  + \frac{2b f^2}{n^2}+ \bigO(\MA) \right) = F_1(\bfl) \exp \left(\bigO(\MA) \right).
    \qedhere
    \popQED
\end{align*}
\renewcommand{\qedsymbol}{}
\end{proof} 	
\vspace{-7mm}
We proceed with Lemma~\ref{lemmascrambled} by summing over all sequences $\bfl$ so that $\lambda < {\ln^{-3} n}$. Along the way, it will be almost no additional effort to observe the following statement. This will not be used further here, but we plan to use it in future work. 
\begin{lemma}\label{lemma:futurework}
Let $a=\alpha_0-\bigO(1)$ and $\kp$ a tame $a$-bounded profile. Assume that $k_a^2/\mu_a \rightarrow \infty$ and define
\[\Pi_{\mathrm{scrambled}}'  = \left\{ (\pi, \pi') \in \Pi_{\textup{scrambled}}  \mid \ell_a(\pi, \pi') \ge  \frac{10 k_a^2}{\mu_a} \right\}.
\]
Then
\[
    \frac{1}{\ex{Z_\kp}^2}\sum_{(\pi,\pi') \in \Pi'_{\textup{scrambled}}} \Pb\left(A_{\pi} \cap A_{\pi'}\right) \le  \exp \Big(- \frac{10 k_a^2}{\mu_a}+\bigO(\MB+\MA)\Big)\enspace.
\]
\end{lemma}

\begin{proof}[Proof of Lemmas~\ref{lemmascrambled} and \ref{lemma:futurework}]
Combining Lemmas~\ref{lemma:sumR1} and \ref{lemma:sumR2}, to prove Lemma~\ref{lemmascrambled} it suffices to show that
\begin{equation} \label{eq:sumoverbfl}
    \sum_{\bfl \le \bfk: \lambda <\ln^{-3} n} F_1(\bfl)\lesssim \exp \left(\frac{k_a^2}{\mu_a}+\bigO(\MB+\MA) \right),
\end{equation}
and for Lemma~\ref{lemma:futurework} we need the  corresponding upper bound when we only sum over $\bfl$ with $\lambda< {\ln^{-3} n}$ and $\ell_a \ge 10 k_a^2 / \mu_a$.  
So suppose that $\bfl \le \bfk$ with $\lambda < {\ln^{-3} n}$. Then
\begin{equation}
\label{boundforh1}
    {\nid}=\sum_u u \ell_u=\lambda n < \frac{n}{\ln^3 n}.
\end{equation} 	
Note that
\[
    \nid^2 \le \ell_a^2 a^2 +  2\nid (\nid - \ell_aa)  = \ell_a^2 a^2 + \bigO\left(\frac{n}{\ln^3n} \right) \sum_{u^* \le u \le a-1} u \ell_u,
\] 
	and so with Lemma \ref{ablemma}, 
	\[
	\frac{(n-{\nid})!}{n!} \lesssim n^{-{\nid}}e^{{\nid}^2/n} \le n^{-{\nid}}e^{{\ell_a^2 a^2}/{n} + \bigO\left(\frac{1}{\ln^ 3n} \right) \sum_{u^* \le u \le a-1}u\ell_u}.
	\]
Therefore, since $\E[X_\ell] = \frac{n!}{(n-\nid)!} \prod_u  \big( q^{{u \choose 2}\ell_u} / (u!^{\ell_u}\ell_u!) \big)$, bounding ${k_u \choose \ell_u} \le k_u^{\ell_u}/\ell_u!$,
\begin{align}
    F_1(\bfl) &= \frac{\prod_u {k_u \choose \ell_u}^2}{\E_p[\bar X_\ell]} \lesssim \frac{e^{\ell_a^2 a^2/n} }{\ell_a!} \left( \frac{k_a^2 a!}{n^{a}q^{{a\choose 2}}}\right)^{\ell_a}\prod_{u^* \le u \le a-1} \frac{1}{\ell_u!} \left( \frac{k_u^2 u !e^{\bigO\left(u / \ln^3 n\right)}}{n^{u}q^{{u\choose 2}}}\right)^{\ell_u}. \label{bound1}
\end{align}
	Recall that $\mu_{u}={n \choose u} q^{u \choose 2}$ is the expected number of independent $u$-sets in $G_{n,p}$. Note that for all $u^* \le u \le a$,
	\begin{equation}\label{bound2a}
		\frac{k_u^2 u!}{n^{u}q^{{u \choose 2}}} =\frac{k_u^2}{{n \choose u}q^{u \choose 2}} \left(1+\bigO\left(\frac{a^2}{n}\right) \right)= \frac{k_u^2}{\mu_{u}}+ \bigO \left( \frac{k_u^2 \ln^2 n}{\mu_{u} n} \right).
	\end{equation}
In particular, we obtain uniformly for $u^* \le u \le a$ that
	\begin{equation}\label{bound2}
		\frac{k_u^2 u!e^{\bigO({u}/{\ln^3 n})}}{n^{u}q^{{u \choose 2}}}  \sim \frac{k_u^2}{\mu_{u}}.
	\end{equation}
It follows from (\ref{expectationu}) that for any $t=\alpha_0-\bigO(1)$ and $u^* \le u \le a$,
	\[
    \mu_{u}
= \mu_t \left(\Theta\Big( \frac{n}{\ln n}b^{-(t-u)/2}\Big) \right)^{t-u}.
	\]
	Therefore, as $k_u \le k = \bigO\left({n}/{\ln n}\right)$ for $u^*\le u \le a-1$, 
	\begin{equation*}
		\sum_{u=u^*}^{a-1}  \frac{k_u^2}{\mu_{u}}= \bigO\left( \frac{n^2}{\ln^2 n } \sum_{u=u^*}^{a-1}  \frac{1}{\mu_{u}}\right)=\bigO \left(\frac{n^2}{\mu_{a-1}\ln^2 n} \right)=\bigO \left(\frac{n}{\mu_{a}\ln n} \right)=\bigO(\MA).  
	\end{equation*}
	So by this and (\ref{bound2}), switching sums and products and using $e^z = \sum_n {z^n}/{n!}$,
	\begin{align}
		\sum_{\ell_{u^*}, \dots, \ell_{a-1}} \prod_{u^*\le u \le a-1} \frac{1}{\ell_u!} \left( \frac{k_u^2 u!e^{O\left(\frac{u}{\ln^3 n}\right)}}{n^{u}q^{{u\choose 2}}}\right)^{\ell_u} = \exp \left((1+o(1))\sum_{u=u^*}^{a-1}  \frac{k_u^2}{\mu_{u}} \right)=\exp \left( \bigO(\MA) \right). \label{eq:summingu2}
	\end{align}
	We have to be more careful when summing \eqref{bound1} over $\ell_a$. Noting that when $z \rightarrow \infty$, the main contribution to $e^z=\sum_n {z^n}/{n!}$ comes from terms where $n \approx z$, and that $\ell_a \le \nid = \lambda n < {n}/{\ln^3 n}$, together with~\eqref{bound2a} we get
	\begin{align*}
		\sum_{\ell_a=0}^{n/\ln^3 n}  \frac{e^{\ell_a^2 a^2/n}}{\ell_a!} \left( \frac{k_a^2 a!}{n^{a}q^{{a\choose 2}}}\right)^{\ell_a} &\le  \exp \left(\frac{k_a^2 a!}{n^{a}q^{{a\choose 2}}}+\bigO\Big( \Big( \frac{k_a^2 a!}{n^{a}q^{{a\choose 2}}}\Big)^2\frac{a^2}{n}\Big) \right) \nonumber \\
		&= \exp \left( \frac{k_a^2}{\mu_{a}}+ \bigO \Big( \frac{k_a^4 \ln^2 n}{\mu_a^2 n}+ \frac{k_a^2 \ln^2 n}{\mu_{a} n}  \Big) \right) \nonumber \\
		&= \exp \left(\frac{k_a^2}{\mu_{a}}+ \bigO \left(\MB+\MA\right)  \right). 
	\end{align*}
Note further that, if we only take the sum over $\ell_a \ge 10 k_a^2 / \mu_a \sim 10k_a^2 a! / (n^a q^{a \choose 2})$ (and $\ell_a < n/\ln^3 n$),  then using Stirling's formula, we can similarly bound by  $\exp\big(-10 k_a^2 / \mu_a+\bigO(\MB+\MA)\big)$, assuming $k_a^2/\mu_a \rightarrow \infty$. 
	Using \eqref{bound1} and \eqref{eq:summingu2}, we can now bound the sum \eqref{eq:sumoverbfl}:  switching sums and products we obtain that
	\begin{align*}
		\sum_{\bfl \le \bfk: \lambda < \ln^{-3} n}  F_1(\bfl)  \lesssim &  \left(\sum_{\ell_a=0}^{n/\ln^3 n}  \frac{e^{\frac{\ell_a^2 a^2}{n}}}{\ell_a!} \left( \frac{k_a^2 a!}{n^{a}q^{{a\choose 2}}}\right)^{\ell_a} \right)  \left(\sum_{\ell_{u^*}, \dots, \ell_{a-1}} \prod_{u^* \le u \le a-1} \frac{1}{\ell_u!} \left( \frac{k_u^2 u!e^{\bigO\left(\frac{u}{\ln^3 n}\right)}}{n^{u}q^{{u\choose 2}}}\right)^{\ell_u}  \right)  \\
		\lesssim &\exp \left(\frac{k_a^2}{\mu_{a}}+ \bigO \left(\MB+\MA\right)  \right).
	\end{align*}
This is \eqref{eq:sumoverbfl} as required, completing the proof of Lemma \ref{lemmascrambled}. For Lemma~\ref{lemma:futurework} we proceed in the same way while restricting the sum to $10 k_a^2/\mu_a \le \ell_a \le n/\ln^3 n$, completing also that proof.
\end{proof}

\subsection{Proof of Lemma \ref{lemmasimilar}} \label{sectionsimilar} 
Let us remark at the beginning that the following proof in fact works for any constant $c_0>0$, not just for $c_0=c/3$.	Recall the definitions of $\bfl, \lambda, s, t, P, f, g$ from the beginning of \S\ref{section:secondmoment} and~\S\ref{section:smprelims}. By \eqref{expectation} and \eqref{jointprobability},
\begin{align}&\frac{1}{\ex{Z_\kp}^2}\sum_{(\pi,\pi') \in \Pi_{\text{similar}}} \Pb\left(A_{\pi} \cap A_{\pi'}\right) \nonumber\\
    & \sim \frac{1}{\ex{Z_\kp}P}\sum_{(\pi,\pi') \in \Pi_{\text{similar}}}\Big( q^{f-g_{\pi, \pi'}} \exp \Big(-\frac{b-1}{n^2} \left((2f-g_{\pi, \pi'})^2-f^2 \right) \Big) \Big)\nonumber \\
    &\le \frac{1}{\ex{Z_\kp} P}\sum_{(\pi,\pi') \in \Pi_{\text{similar}}} q^{f-g_{\pi, \pi'}}\nonumber .
\end{align}
As $\kp$ is tame, by Proposition \ref{lem:XkZk} we have 
\[
    \E[Z_\kp] \sim \E[X_\kp] = \E[\bar X_\kp] \cdot \prod_u k_u!  \enspace .
\]
So to prove Lemma \ref{lemmasimilar}, it suffices to show 
\begin{equation}
\label{sdfsedf}
    \sum_{(\pi,\pi') \in \Pi_{\text{similar}}} q^{f-g_{\pi, \pi'}} \le P \,\prod_u k_u!   \exp \Big(\bigO \big(
		\ln n \big)\Big).
\end{equation}
In order to achieve this, we will generate and count all pairs of partitions $(\pi, \pi')\in \Pi_{\text{similar}}$ while bounding $f-g_{\pi, \pi'}$ from below. 
	
	First note that if $(\pi, \pi') \in \Pi_{\text{similar}}$, then there are exactly $\lambda n$ vertices in identical parts, where $\lambda=\lambda(\pi, \pi')>1-n^{-c_0}$. As each part contains at least one vertex (in fact, at least $u^*$ vertices), the number of scrambled parts is at most
	\begin{equation} \label{eq:boundS}
		s(\pi, \pi') \le t(\pi, \pi') \le  (1-\lambda)n < n^{1-c_0}.
	\end{equation}
	Given $\pi$ and $\pi'$, let 
	\begin{align*}
		z_u(\pi, \pi') &= \text{ number of $2$-composed parts of $\pi$ of size $u$ (with respect to $\pi'$)},\\
		z(\pi, \pi') &= \text{ total number of $2$-composed parts of $\pi$ (with respect to $\pi'$)}. 
	\end{align*}
Fix some integers $s\le n^{1-c_0}$  and $(z_u)_{u^*\le u \le a}$ and let $z = \sum_u z_u$. We will first generate and count all pairs $(\pi, \pi') \in \Pi_\mathrm{rel}(\bfk)$ such that $s(\pi, \pi')=s$ and $z_u(\pi, \pi')=z_u$ for all $u$, and then take the sums over all $s\le n^{1-c_0}$  and $(z_u)_u$ in order to obtain \eqref{sdfsedf}. Without loss of generality we only count those pairs $(\pi, \pi')$ where $z(\pi, \pi') \ge z(\pi', \pi)$, 
	as by symmetry this decreases the left-hand side of (\ref{sdfsedf}) by at most a factor ${1}/{2}$.
	
	Any pair $(\pi, \pi') \in \Pi_\mathrm{rel}(\bfk)$ corresponding to $s$ and $(z_u)_u$ can be obtained in the following way. First choose $\pi$, for which there are $P$ possibilities, and then pick the scrambled and the $2$-composed parts, for which there are at most
	\[
	{k \choose s}\prod_u{k_u \choose z_u} \le \frac{k^s}{s!}\prod_u\frac{k_u^{z_u}}{z_u!} 
	\]
	choices. As $(\pi, \pi')$ is relevant, by Fact~\ref{tamelemma} all remaining parts of $\pi$ must be $1$-composed. From every $2$-composed part $V_i$, pick the exceptional vertex $v_i \in V_i$ which is in a different part of $\pi'$ than the other vertices in $V_i$, and remove it from $V_i$. There are at most
	\[
	   a^{z}
	\]
	choices for the exceptional vertices. Next, choose the vertices which will be the exceptional vertices in $2$-composed parts of $\pi'$. As w.l.o.g.\ $z(\pi', \pi) \le z(\pi', \pi)$, there are at most~$z$ of them. In $\pi$, those vertices must either be in one of the $s$ scrambled parts or they are amongst the $z$ exceptional vertices of $\pi$, so there are at most
\[
       2^{as+z}
\]
choices. Now assign these vertices to $1$- and (formerly) $2$-composed parts of $\pi$ in order to create the $2$-composed parts of $\pi'$. This can be done in at most
\[
   k^{z}
\]
ways.
We have generated all $1$-composed and $2$-composed parts of $\pi'$, and now only need to rearrange the remaining vertices to form the scrambled parts of $\pi'$. As each scrambled part of $\pi'$ contains at most $a$ vertices, there are at most $sa$ vertices left, so there are at most
\[
    s^{sa}
\]
ways to group them into $s$ parts. Finally, we choose one of the
\[
    \prod_u k_u!
\]
possible orderings of the parts, and this defines $\pi'$ completely. 
Overall, given $s$ and $(z_u)_u$, if $n$ is large enough there are at most
\begin{align}
    P \frac{\left(2^a s^a k\right)^s}{s!}(2ak)^{z} \prod_u \left(\frac{k_u^{z_u}}{z_u!} k_u! \right) \le \left( P \prod_u k_u! \right) \frac{\left(2^a s^a k\right)^s}{s!} \prod_u \left(\frac{\left(3ak^2\kappa_u\right)^{z_u}}{z_u!} \right)  \label{numberpiprime}
\end{align}
choices for the pair $(\pi, \pi')$ (note that $\kappa_u={k_u u}/{n} \sim {k_u}/{k}$ for $u^* \le u \le a$).
	
	We also need a lower bound for $f-g_{\pi, \pi'}$, the number of forbidden edges in~$\pi$ (i.e., pairs of vertices that are in the same part of $\pi$) that are not forbidden in $\pi'$. In $\pi$, each of the $z_u$ exceptional vertices in a $2$-composed part of size $u$ is the endpoint of exactly $u-1$ forbidden edges in $\pi$ which are not forbidden edges of $\pi'$. All of these forbidden edges are distinct. Furthermore, each of the $s$ scrambled parts of $\pi$ is at least $(u^*-2(a-u^*)-1)$-composed. As $a \sim u^* \sim 2 \log_b n$, if $n$ is large enough then each scrambled part contains at least 
	\[{u^*-2(a-u^*)-1 \choose 2} = {(2-o(1))\log_b n \choose 2} \ge (2-c_0) \log^2_b n\]
	forbidden edges which are not forbidden in $\pi'$. Hence,
	\begin{align*}
		f-g_{\pi, \pi'} &\ge \sum_u (z_u (u-1)) + (2-c_0) s \log_b^2 n.
	\end{align*}
	As $a=\alpha_0-\bigO(1)$ and $b^{a} = \Theta\left( {n^2}/{\ln^2 n}\right)=\Theta(k^{2})$, 
	\[
	q^{f-g_{\pi, \pi'}} =b^{-(f-g_{\pi, \pi'})}\le  b^{(-2+c_0)s \log^2_b n}\prod_u \bigO \left(k^{-2} b^{a-u}\right)^{z_u}.
	\]
	Together with (\ref{numberpiprime}), the contribution to (\ref{sdfsedf}) from relevant $(\pi$, $\pi')$ corresponding to the fixed values $s$ and $(z_u)_u$ is at most
	\begin{align}
		\left( P \prod_u k_u! \right) \frac{\left(2^a s^a k b^{(-2+c_0) \log^2_b n}\right)^s}{s!} \prod_u \left(\frac{\bigO\left(a\kappa_u b^{a-u}\right)^{z_u}}{z_u!}  \right). \label{zwischen}
	\end{align}
	As we picked $s \le n^{1-c_0}$ and since $a \sim \alpha_0 \sim 2 \log_b n$, if $n$ is large enough we have 
	\[
	2^a s^a k b^{(-2+c_0) \log^2_b n}\le b ^{(-2 + c_0 + o(1) ) \log^2_b n+ a (1-c_0) \log_b n}= b^{(-c_0+o(1))\log_b^2 n} \le 1.
	\]
So, summing \eqref{zwischen} over all possible choices for integers $s \le n^{1-c_0}$ and $(z_u)_u$, using once more that $e^z = \sum_{n \ge 0} z^n/n!$, we obtain
	\begin{align*}
		\sum_{(\pi,\pi') \in \Pi_{\text{similar}}} q^{f-g(\pi_0, \pi')} & \lesssim P \prod_{u^*\le u \le a} k_u! \exp \left(1+ \bigO(a) \sum_{u^*\le u \le a} \kappa_u b^{a-u} \right)  \\
		&\le P\prod_{u^*\le u \le a} k_u!  \exp\parenth{\bigO\big(a\big)} ,
	\end{align*}
using the property $\kappa_u < b^{-(\alpha-u)\gamma(\alpha-u)}$ from Definition \ref{deftame}~\eqref{tame:tail} in the last step. As $a=\bigO(\ln n)$, this is~\eqref{sdfsedf} as required.	
\qed

\section{Optimal profiles}\label{section:optimalprofiles}
		
In this section we study the colouring profiles that maximize the expectation $\E[\bar X_\bfk]$ amongst all $t$-bounded colouring profiles $\bfk$ with $k$ colours. This analysis will be used in the verification of Theorems \ref{twopointrestricted} and \ref{theorem:announcedbounds}. We set $p = 1/2$, so that $b=1/(1-p)=2$ throughout.

The main goal of this section is to establish that we may apply Theorem~\ref{theorem:general} to an optimal $t$-bounded profile, at least in the cases required to verify Theorems~\ref{twopointrestricted} and \ref{theorem:announcedbounds}. That is, we establish that the optimal profile is tame and, furthermore, meets the additional constraints \eqref{eq:lowerboundexpectation} and \eqref{eq:lowerboundbeta} from Theorem~\ref{theorem:general}. This is ultimately achieved in Lemma~\ref{lemma:kstartame}. The main difficulty is to show that the expected number of certain sub-colourings is large in order to verify condition \eqref{eq:lowerboundbeta}. This will be established in \S\ref{section:partialcolourings}, after some preparation in \S\ref{section:previousresults}--\ref{section:convergence}.

We will make use of the analysis of optimal, i.e., expectation maximizing, profiles in \S3.1 of \cite{HRHowdoes}, where the optimal $t$-bounded colouring profile was approximated by a continuous version (i.e., where the number of colour classes of each size is a non-negative real number, not necessarily an integer). 
For the purposes of this paper we need tighter bounds compared to what is established in~\cite{HRHowdoes}, showing that the difference in the expected number of colourings between the optimal integer profile and the optimal continuous profile is extremely small. This is achieved in Lemma~\ref{lemma:improvedapproximation} below. The corresponding result in \cite{HRHowdoes} is Lemma 29 there, which would be sufficient to prove concentration on $\bigO(\ln^2 n)$ values in Theorem~\ref{twopointrestricted}, but not for the two-point concentration result that we obtain.

\subsection{Setup and previous results}\label{section:previousresults}

We first collect some notation and results from \cite{HRHowdoes}. For $t=\bigO(\ln n)$, let $E_{n,k,t}$ denote the expected number of unordered $t$-bounded $k$-colourings in $\Gnh$ (\emph{not} in $G_{n,m}$ with $m=\left \lfloor N/2 \right \rfloor$, although by Lemma~\ref{gnmlemma} these two quantities only differ by a factor $\exp({\bigO(\ln^2 n)})$). We will replace $\ln E_{n,k,t}$  by a quantity which is easier to handle.
	
Let $P_{n,k,t}$ be the set of all $t$-bounded $k$-colouring profiles $\bfk=(k_u)_{1 \le u \le t} \in  \mathbb{N}_0^t$  on $n$  vertices, that is, 
\begin{equation}\label{constraint}
    k_u\ge 0,\qquad  \sum_{1 \le u \le t} k_u=k  \qquad \text{and}\qquad \sum_{1 \le u \le t} u k_u = n.
\end{equation}
Given positive reals $k<n$ and a positive integer $t$, we define a real-valued version, letting
	\[
	P^0_{n,k,t} = \bigl\{\  (k_u)_{1 \le u \le t} \in \R^t : \text{\eqref{constraint} holds}\ \bigr\}.
	\]
For $1 \le u \le t$ and $\bfk \in P_{n,k,t}^0$, let
\[
    d_u:= 2^{\binom{u}{2}}u! \quad \text{and} \quad L_{\bfk} := n \ln n - n + k - \sum_{1\le u \le t} k_u \ln (k_u d_u),
\]
and set 
\begin{equation}\label{eq:L0def}
    L_0(n,k,t):=\sup_{\bfk \in P^0_{n,k,t}}L_\bfk.
\end{equation}
In \cite[Lem.~29]{HRHowdoes} it was established for suitable $n,k,t$ that 
\begin{equation} \label{eq:oldapprox}
    \ln(E_{n,k,t}) = L_0(n,k,t) + \bigO(\ln^4 n),
\end{equation} 
which, in many cases, means that $L_0(n,k,t)$ is a good approximation for $\ln(E_{n,k,t})$, while being much easier to analyse. In Lemma~\ref{lemma:improvedapproximation} below we will sharpen~\eqref{eq:oldapprox} for suitable $n, k, t$, replacing  $\bigO(\ln^4 n)$ by  $\bigO(\ln^{3/2} n)$.
Continuing to follow \cite{HRHowdoes}, we rescale by setting $p_u=k_u/k$, which is the fraction of colour-classes of size $u$. So for $t\in\mathbb{N}$ and $\rho\in(1,t)$ a real number let
\[
    \tP_{\rho,t}
    = \left\{ (p_u)_{1\le u \le t} \in [0,1]^t\ :\ \sum_{1 \le u \le t} p_u = 1
    \text{\quad and\quad}
    \sum_{1 \le u \le t} u p_u = \rho
    \right\}.
\]
This is simply the set of probability distributions with support $\{1, \dots, t\}$ and mean $\rho$. 
Moreover, when setting $\rho=n/k$ this is in 1-1 correspondence with  $P^0_{n,k,t}$,  where we substitute $p_u=k_u/k$. Set
\begin{equation*}
    \tL(\rho,k,\bp)
    := \rho\ln(\rho k)-\ln k -\rho +1 -\sum_{1\le u \le t} p_u\ln(p_ud_u),
\quad\bp=(p_i)_{i \le i\le t} \in \tP_{\rho,t},
\end{equation*}
and define the analogue of $L_0(n,k,t)$ by
\begin{equation}\label{eq:L0tildedef}
    \tL_0(\rho,k,t) := \sup_{\bp \in \tP_{\rho,t}} \tL(\rho,k,\bp).
\end{equation}
	It follows immediately from the definitions that for any positive integer  $t$  and positive reals $n$ and $k$ with $1<n/k<t$,
\[
    L_0(n,k,t) = k \tL_0(\rho,k,t),
    \quad\text{where}\quad
    \rho = n/k.
\]
The following lemma from \cite{HRHowdoes} describes the location of the maximum of $\tL(\rho,k,\bp)$ over $\bp$, which is actually independent of $k$.
\begin{lemma}[\cite{HRHowdoes}, Lem.~32]
\label{lem:tLmax}
Let $1<\rho<t$, where $t \in \mathbb{N}$. Then for any real $k>1$,  
there is a unique $\bp=\bp_{\rho,t}\in \tP_{\rho,t}$ maximizing $\tL(\rho,k,\bp)$ which is given by
\begin{equation}\label{pixy}
    p_u = e^{x_t(\rho)+uy_t(\rho)} d_u^{-1}, \quad 1\le u\le t,
\end{equation}
where $x_t(\rho)$ and $y_t(\rho)$ satisfy the equations
		\begin{equation}\label{xy1}
			\sum_{1\le u \le t} e^{x_t(\rho)+uy_t(\rho)} d_u^{-1} =1
		\end{equation}
		and
		\begin{equation}\label{xy2}
			\sum_{1\le u \le t} u e^{x_t(\rho)+uy_t(\rho)} d_u^{-1} = \rho.
		\end{equation}
	\end{lemma}
	We will also need the following results from \cite{HRHowdoes}.
	
\begin{lemma}[\cite{HRHowdoes}, Cor.~37]\label{lemma:ybound}
Suppose that $t=t(n)=\alpha_0(n)-\bigO(1)$ is an integer. Uniformly over all $n$ and all real $\rho\ge 2$ such that $t-\rho=\Theta(1)$ we have
\begin{equation*}
    y_t(\rho) = 2\ln n-\ln\ln n + \bigO(1).
\end{equation*}
\end{lemma}

\begin{lemma}[\cite{HRHowdoes}, Cor.~39]\label{lemma:onemorecolour}
    Suppose that $t=t(n)=\alpha_0(n)-\bigO(1)$ is an integer. Uniformly over all $k\le n/2$ such that $k=n/(t-\Theta(1))$,
\[
    \frac{\partial}{\partial k} L_0(n,k,t)
    = \frac{2}{\ln 2}\ln^2 n+\bigO(\ln n\ln\ln n).
\]
\end{lemma}
The following statement is a direct consequence of Lemma 41 in \cite{HRHowdoes} (noting that in that lemma, we can plug in $\ln \mu_\alpha = \theta \ln n$ by \eqref{eq:defx2}, and that $\alpha_0=\alpha+\theta+o(1)$ by \eqref{eq:theta}).
\begin{lemma}\label{lemma:averagecolourclass}
	Suppose that $t\in \{\alpha(n)-1, \alpha(n)-2\}$, and let $\textbf{k}_t= \min \{ k: E_{n,k,t} \ge 1\}$ denote the $t$-bounded first moment threshold first defined in \eqref{ktdef}. Then
	\[
	\frac n{\textbf{k}_t} = \alpha_0-1-\frac{2}{\ln 2}+o(1).
	\]
\end{lemma}

\subsection{Reparametrisation}

From here on we deviate from~\cite{HRHowdoes}, bringing the characterisation of $\bp$ in Lemma~\ref{lem:tLmax} into a form that we can study more easily in our setting. The advantage of our parametrisation is that we replace $x_t(\rho)$ and $y_t(\rho)$ by quantities $\lambda_n(t, \rho)$ and $\mu_n(t, \rho)$ that are bounded and that we can approximate analytically and also numerically.
Moreover, it will be convenient to replace $u$ (where $1 \le u \le t$) by $\alpha-i$ (where $\alpha-t \le i \le \alpha-1$).	We first define the auxiliary function
	\begin{equation} \label{eq:defh}
	h_n : \mathbb{N} \to \mathbb{R}, \quad  i \mapsto - \frac{\ln 2}{2}i^2 - \ln((\alpha-i)!) + \ln(\sqrt{2\pi})+(\alpha-i+1/2)\ln \alpha - \alpha.
	\end{equation}
	Note that the last three terms are similar to the leading terms in  
the Stirling approximation of $\ln((\alpha-i)!)$ for large $\alpha-i$. 
We state without proof some easy asymptotic properties of $h_n(i)$ for later use.
\begin{lemma}\label{lemma:hasymp}		
Uniformly for all $i = o(\ln^{1/2} n)$,
\[
    h_n(i) = - \frac{\ln 2}2 i^2 +o(1) \quad \text{as} \quad n \rightarrow \infty,
\]
 and uniformly for all $i \le \alpha-1 = \bigO(\ln n)$,
\[
    h_n(i) = - \Big(\frac{\ln 2}2 -o(1) \Big) i^2  \quad \text{as} \quad n \rightarrow \infty. 
\] 
\end{lemma}	
	 Now, letting $i = \alpha-u$, we have
	\begin{equation}\label{eq:logdu}
	-\ln d_u = -{\alpha-i \choose 2}\ln 2 - \ln (\alpha-i)! = h_n(i) 
	+ \Big( \alpha \ln 2 + \ln \alpha - \frac{\ln 2 }{2} \Big)i + f(\alpha),
	\end{equation}
	where $f(\alpha)$ is some function of $\alpha$.
	Thus, if we let $\xi_i = p_{\alpha-i}$, absorbing the terms above other than $h_n(i)$ into $\lambda_n$, $\mu_n$, we can rewrite \eqref{pixy}--\eqref{xy2} as
	\begin{equation}
		\label{eq:xiu}
		\xi_i 
  =\xi_i(n) 
		 = e^{h_n(i) + \lambda_n + \mu_n i}, \quad \alpha-t \le i \le \alpha-1,
	\end{equation}
	where $\lambda_n=\lambda_n(t, \rho)$ and $\mu_n=\mu_n(t,\rho)$ are chosen so that
	\begin{equation}
		\label{eq:constraints_xiu}
		\sum_{i=\alpha-t}^{\alpha-1} \xi_i = 1, \qquad 
		\sum_{i=\alpha-t}^{\alpha-1}  i\xi_i = \alpha-\rho.
	\end{equation}
Note that the dependence on $n$ of all of these quantities is only through $\alpha=\alpha(n)$. The advantage of this new formulation is that (at least for the values of $n,k,t$ which are relevant for Theorems~\ref{twopointrestricted} and \ref{theorem:announcedbounds}) $\lambda_n, \mu_n$ are of order $\bigO(1)$, as we show below. Furthermore, as $n\rightarrow \infty$, $\lambda_n, \mu_n$ converge towards values (or rather, functions) $\lambda, \mu$ which only depend on $\alpha-t$ and $\alpha_0-\alpha$, as we will see in the next section. 
	
\begin{lemma}\label{lemma:mubound} Suppose that $t=\alpha_0(n)-\bigO(1)$ is an integer. Uniformly over all  $n, k \in \mathbb{N}$ so that $2 \le n/k = t- \Theta(1)$, 
	\begin{equation}
		\mu_n= \bigO(1) \quad \text{and} \quad \lambda_n = \bigO(1).
	\end{equation}	
\end{lemma}

\begin{proof}
	Observe that by the definition of $\mu_n$ and \eqref{eq:logdu},
	\[\mu_n  = - y_t(\rho)  + \alpha \ln 2 + \ln \alpha - \frac {\ln 2}{2}. \]
Noting that $\alpha = 2 \log_2 n - 2 \log_2 \log_2 n +\bigO(1)$, the claim for $\mu_n$ now follows directly from Lemma \ref{lemma:ybound}.
Moreover, note that by \eqref{eq:constraints_xiu},  
\[
    \lambda_n
    = -\ln \Big( \sum_{\alpha-t \le i  \le \alpha-1}e^{h_n(i)+\mu_n i}\Big),
\]
and together with Lemma~\ref{lemma:hasymp} this implies the claim for $\lambda_n$.
\end{proof}

	We will need the following useful property of the $\xi_i$'s, which says that they decrease very fast when viewed as a function of $i$.
\begin{lemma} 
	\label{lem:monoxi}
	Under the conditions of Lemma~\ref{lemma:mubound} there is a $C>0$ such that $\xi_{i+1}\le C 2^{-i} \cdot\xi_i$ for all $\alpha-t\le i < \alpha-1$.
\end{lemma}
\begin{proof}
	From the definitions \eqref{eq:defh} and \eqref{eq:xiu} of $h_n$ and $\xi_i$ and from Lemma \ref{lemma:mubound} we obtain uniformly in $i$ that
	\[	
	\ln \frac{\xi_{i+1}}{\xi_i} = h_n(i+1) - h_n(i) + \mu_n = -i \ln 2  +  \ln(\alpha-i) - \ln\alpha + \bigO(1).
	\]
	The claim follows.
\end{proof}

\subsection{Convergence of $\lambda_n$, $\mu_n$, and of the $\xi_i$'s}
\label{section:convergence}

In this section we show that if $t\in \{\alpha-1, \alpha-2\}$ and $n/k = \alpha_0-1- 2/ {\ln 2}+o(1) $ -- that is, in the cases that will turn out to be relevant in the proofs of Theorems~\ref{twopointrestricted} and~\ref{theorem:announcedbounds} -- then, as~$n$ gets large, the quantities $\lambda_n, \mu_n$ and the $\xi_i$'s get close to some functions $\lambda,\mu$ and $\zeta_i$ that only depend on $\alpha-t \in \{1,2\}$ and $ \alpha_0 - \alpha \in [0,1]$. 
We start by defining these `limiting' functions. Let $i_0 \in \{1,2\}$ and $x \in[0,1]$. 
(Later on we will set $i_0=\alpha-t$ and $x=\alpha_0-\alpha$.) 
Let $\lambda = \lambda(i_0,x), \mu = \mu(i_0,x)$ be given in the following way. Set
\begin{equation}
\label{eq:defxi}
	\zeta_i = e^{\lambda+\mu i-\frac{\ln 2}{2}i^2},
 \quad
 i \ge i_0
\end{equation}
and
\begin{equation}
\label{eq:defT}
T=T(x) = 1+\frac{2}{\ln 2} -x,
\end{equation}
and choose $\lambda , \mu$ so that
\begin{align}
\label{eq:deflambdamu}
    \sum_{i \ge i_0} \zeta_i
    = \sum_{i \ge i_0} e^{\lambda +\mu i - \frac{\ln 2}{2} i^2}=1, 
    \qquad
    \sum_{i \ge i_0} i \zeta_i
    = \sum_{i \ge i_0}  i e^{\lambda +\mu i - \frac{\ln 2}{2} i^2} = T(x) . 
\end{align}
Note that $\lambda, \mu$ exist and are unique. Indeed, $\mu$ uniquely solves $\sum_{i \ge i_0} (T(x)-i)e^{\mu i - \frac{\ln 2}{2}i^2}=0$, where, since $T(x) = 1+{2}/{\ln 2} -x > 2 \ge i_0$, the first terms of $(T(x)-i)_{i \ge i_0}$ are positive  and the others are negative. Moreover, if we have determined $\mu$, then $\lambda = -\ln (\sum_{i \ge i_0} e^{\mu i - \frac{\ln 2}{2}i^2})$ is also unique. Furthermore, note that given $i_0$, $\mu(i_0, x)$ and $\lambda(i_0, x)$ are continuous functions of~$x$. The aim of this section is to show the following statement.
\begin{lemma}\label{lemma:convergencemulambda}
Suppose that $t \in \{\alpha-1, \alpha-2\}$, and $k=k(n)$ is an integer so that $n/k = \alpha_0(n)-1- 2 / {\ln 2} + o(1)$. Define $\lambda_n$, $\mu_n$ as in \eqref{eq:constraints_xiu} with $\rho = n/k$, and let $\lambda = \lambda(i_0, x)$, $\mu =\mu (i_0, x)$ as in \eqref{eq:deflambdamu} with $i_0=\alpha-t$ and $x=\alpha_0(n)-\alpha(n)$. Then
\[
    \lim_{n \rightarrow \infty} |\lambda_n - \lambda| =
    \lim_{n \rightarrow \infty} |\mu_n - \mu| =
    \lim_{n \rightarrow \infty}\sum_{i \ge i_0}|\xi_i - \zeta_i| =
    0.
\]
\end{lemma}
\begin{proof}
First note that $
T(x) = 1+ 2 /{\ln 2} - x = \alpha- n/k+o(1).
$
By Lemma~\ref{lemma:mubound} we have $\mu_n, \lambda_n = \bigO(1)$, so together with Lemma~\ref{lemma:hasymp} and~\eqref{eq:constraints_xiu} this implies
\begin{equation*}
	\sum_{i \ge i_0} e^{-\frac{\ln 2}{2}i^2+ \lambda_n + \mu_n i}= 1+o(1), \qquad 
	\sum_{i \ge i_0} ie^{-\frac{\ln 2}{2}i^2+ \lambda_n + \mu_n i}= T(x)+o(1).
\end{equation*}
Comparing to \eqref{eq:deflambdamu}, the claims for $\lambda_n,\mu_n$ follow by continuity. In order to establish the last claim, using again that $\mu_n, \lambda_n = \bigO(1)$,  Lemma~\ref{lemma:hasymp} guarantees that there is an $M \in \mathbb{N}$ such that for sufficiently large $n$,
\[
    |\xi_i-\zeta_i| = o(1) , ~~ i_0 \le i \le \ln\ln n
    \quad
    \text{and}
    \quad
    |\xi_i-\zeta_i| \le \xi_i+\zeta_i \le e^{-i^2/4}, ~~ i \ge M
    \enspace ,
\]
where the first bound is uniform in $i$. Summing the first bound for $i$ up to some $M' \ge M$ and the second bound for all $i > M'$, and then letting $M'$ grow slowly establishes the claim.
\end{proof}

\subsection{The expected number of partial colourings}
\label{section:partialcolourings}
Define $\zeta_i$ as in \eqref{eq:defxi}--\eqref{eq:deflambdamu}. For $i_0 \in \{1,2\}$, $s \ge i_0$ and $x \in [0,1]$, we let
\begin{equation}\label{eq:phisxi}
\varphi(s,x,i_0) = 	- \Big(1-\sum_{i_0 \le i \le s} \zeta_i \Big) \ln  \Big(1-\sum_{i_0 \le i \le s} \zeta_i \Big) + \frac{\ln 2}{2}\sum_{i_0 \le i \le s} \Big( \zeta_i \big(i-T(x)\big) \Big) .
\end{equation}
Comparing to Lemma~\ref{expectationlemma}, we see that for $x=\alpha_0-\alpha$ this expression corresponds to the function $\varphi$	in the logarithm of the expected number of partial colourings with exactly $\zeta_i (\alpha-i)/n$ colour classes of size $\alpha-i$ for $i_0 \le i \le s$. 
The aim of this section is to establish the following two lemmas, which will be proved in Section~\ref{section:proofofpartialprofiles}. They will later help us to bound the expected number of such partial colourings from below.
\begin{lemma}\label{lemma:partialprofiles2}
Let  $i_0\in \{1,2\}$. For every $s \ge 2$,
\begin{equation} \label{eq:goal2}
    \inf_{x \in [0,1]}	\varphi(s,x,i_0)   >0.
\end{equation}
\end{lemma}
For the remaining case $i_0 = s = 1$ such a statement is no longer true. However, we still obtain the following property. 
\begin{lemma}\label{lemma:partialprofiles1}
Let  $i_0=s=1$. Then
\begin{equation} \label{eq:goal1}
    \inf_{x \in [0.04,1]}		\varphi(1,x,1)   >0.
\end{equation}
\end{lemma}	
The awkward constraint $x \ge 0.04$ is unfortunately necessary here, because it turns out that $\varphi(1,x,1)<0$ for $0 \le x < x_0$, where $x_0$ satisfies
\begin{equation}
\label{eq:x0}
    \varphi(1,x_0,1)
    = -(1-\zeta_1(x_0)) \ln(1-\zeta_1(x_0)) + \zeta_1(x_0)\left(\frac{\ln 2}{2}x_0-1\right)
    = 0
\end{equation}
and $\zeta_1$ is, as before, given by \eqref{eq:defxi}--\eqref{eq:deflambdamu}. A numerical approximation reveals that $x_0 \approx 0.02905$. We pick up this fact again after Lemma~\ref{lemma:kstartame}.

Note that for all $s \ge i_0 \in \{1,2\}$, $\varphi(s,x,i_0)$ is a continuous function of $x$ because $\mu(i_0,x)$, $\lambda(i_0,x)$ are. So, by compactness, it suffices to show that $\varphi(s,x,i_0)>0$ for all $0 \le x \le1$, or for $0.04\le x \le 1$ in the case $i_0=s=1$.  We prove this in \S\ref{section:proofofpartialprofiles} after noting some monotonicity properties and numerical approximations for $\lambda(i_0,x), \mu(i_0,x)$ in \S\ref{section:monotonicity}--\ref{section:numerics}. Let us remark at this point that we do not provide an entirely analytical proof of Lemmas~\ref{lemma:partialprofiles2} and~\ref{lemma:partialprofiles1}, as $\lambda$ and $\mu$ are only defined implicitly through~\eqref{eq:deflambdamu} as functions of $x$ and $i_0$. However, what we do show (in the proof of Lemma~\ref{lemma:checksle3}) is that it is enough to the verify the value of $\varphi$ numerically for a finite number of explicit values 
in order to establish Lemmas~\ref{lemma:partialprofiles2} and \ref{lemma:partialprofiles1}, which we proceed to do.

\subsubsection{Monotonicity}\label{section:monotonicity}
	We start by noting some handy monotonicity properties of various quantities that appear in our calculations.
	\begin{lemma}\label{lemma:monomu}
		Let $i_0 \in \{1,2\}$. Then $x \mapsto \mu(i_0, x)$ is strictly decreasing and $x \mapsto \lambda(i_0,x)$	is strictly increasing for $x\in [0,1]$.
	\end{lemma}
	
	\begin{proof}
		By rearranging the equations in \eqref{eq:deflambdamu}, we find that $\mu$ is the solution of
		\begin{equation}\label{eq:defmusimpler}
			\sum_{i \ge i_0} \big(i-T(x)\big)e^{\mu i - \frac{\ln 2}{2} i^2} =0,
		\end{equation}
		or equivalently of 
		\begin{equation}
			\sum_{i \ge i_0} \big(i-T(x)\big)\frac{e^{\mu i - \frac{\ln 2}{2} i^2}}{\sum_{j \ge i_0} e^{\mu j - \frac{\ln 2}{2} j^2}} =0. \label{eq:muroot}
		\end{equation}
		The left-hand side of this equation is a weighted average of the numbers $i - T(x)$, $i \ge i_0$. This weighted average is strictly increasing in $\mu$, since increasing $\mu$ puts more weight on the larger elements of the sequence where $i$ is larger. It is also strictly increasing in $x$ as $T(x)=1+{2}/{\ln 2} -x$. So if we increase $x$, this means that $\mu$, as the root of \eqref{eq:muroot},  decreases.	
		
For the function $\lambda(x)$, we obtain from \eqref{eq:deflambdamu} that
		\[
		\lambda = - \ln  \Big( \sum_{i \ge i_0} e^{\mu i-\frac{\ln 2 }{2}i^2} \Big),
		\]
		and so 
		\begin{equation} \label{eq:dlambdadmu}
			\frac{\text{d} \lambda}{\text{d}  \mu} = -\frac{ \sum_{i \ge i_0} ie^{\mu i-\frac{\ln 2 }{2}i^2}}{ \sum_{i \ge i_0} e^{\mu i-\frac{\ln 2 }{2}i^2}} = -T(x) = x - 1- \frac{2}{\ln2} <0.
		\end{equation}
		So $\lambda$ is strictly decreasing as a function of $\mu$ and strictly increasing as a function of $x$.
	\end{proof}
	
	\begin{lemma}\label{lemma:monoxi}
For $i_0 =1$, $\zeta_1(x)$ is increasing in $x \in [0,1]$. For $i_0 \in \{1, 2\}$, $\zeta_2(x)$ is increasing in $x \in [0,1]$, and $\zeta_3(x)$ is {increasing} for $x \in \left [0, \frac{2}{\ln 2} - 2\right]$ and {decreasing} for $x \in \left [\frac{2}{\ln 2} - 2,1\right]$.
	\end{lemma}
	
	\begin{proof}
For a given $i$, by \eqref{eq:defxi} and \eqref{eq:dlambdadmu},
		\begin{equation}\label{eq:derkappa}
			\frac{\text{d} \zeta_i}{\text{d}  \mu} = \Big(\frac{\text{d} \lambda}{ \text{d}\mu} + i \Big) \zeta_i = (x-1 - \frac{2}{\ln 2}+i) \zeta_i.
		\end{equation}
By Lemma~\ref{lemma:monomu}, $\mu=\mu(x)$ is strictly decreasing as a function of $x$, and since $\zeta_i >0$ for all $i \ge i_0$, we obtain the required monotonicity properties directly from \eqref{eq:derkappa}.
	\end{proof}
	
	\subsubsection{Numerical approximations}\label{section:numerics}
	
	We will need some approximations for $\lambda$, $\mu$ and $\zeta_i$ at certain values $x\in[0,1]$ and $i_0\in\{1,2\}$. For more details on how these approximations were obtained, see Appendix~\ref{app:numApprox}. The following tables indicate intervals in which the respective values are contained.

 \vspace{2mm}
	\noindent \textbf{For $i_0=1$:}
 
	\renewcommand{\arraystretch}{1.2}

\begin{table}[H]
\small
\begin{center}
    \begin{tabular}{|l||l|l|l|l|l|}
        \hline
        $x$  &  $\mu(x)$ & $\lambda(x)$ & $\zeta_1(x)$ & $\zeta_2(x)$ & $\zeta_3(x)$ \\
        \hline
        $0$ &   $[2.6879,2.6880]$ &$[-6.313,-6.311]$ &  $[0.0188,0.0189]$&  $[0.0980, 0.0981]$ & $[0.254, 0.255]$\\
        \hline
        $0.15$ & $[2.5816,2.5817]$&  $[-5.908,-5.906]$ &  $[0.0254,0.0255]$&$[0.118, 0.119]$  &$[0.277, 0.278]$\\
        \hline
        $\frac{2}{\ln 2}-2$ & $[2.0407, 2.0408]$ & $[-4.089, -4.087]$ &  $[0.0912,0.0913]$&  $[0.248, 0.249]$ & $[0.337, 0.338]$\\
        \hline
        $1$ & $[1.9512, 1.9513]$&  $[-3.825, -3.824]$ &  $[0.108,0.109]$&  $[0.270, 0.271]$ & $[0.336, 0.337]$ \\
        \hline
    \end{tabular}
\end{center}
\end{table}

\vspace{-6mm}
\noindent \textbf{For $i_0=2$:}

\begin{table}[H]
\small
\begin{center}
    \begin{tabular}{|l||l|l|l|l|l|}
        \hline
        $x$ & $\mu(x)$   & $\lambda(x)$ &  $\zeta_2(x)$ & $\zeta_3(x)$ \\
        \hline
        $0$ &   $[2.6443,2.6444]$& $[-6.123, -6.122]$& $[0.108, 0.109]$ & $[0.270, 0.271]$ \\
        \hline
        $\frac{2}{\ln 2}-2$ & $[1.8229, 1.8230]$&  $[-3.318, -3.317]$ &   $[0.347,0.348]$&  $[0.379, 0.380]$ \\
        \hline
        $1$ &  $[1.6836, 1.6837]$& $[-2.909, -2.907]$ &   $[0.395,0.396]$&  $[0.376, 0.377]$ \\
        \hline
    \end{tabular}
\end{center}
\end{table}

	\subsubsection{Proof of Lemmas~\ref{lemma:partialprofiles2} and \ref{lemma:partialprofiles1}}\label{section:proofofpartialprofiles}
	
	We start by showing that in Lemma~\ref{lemma:partialprofiles2}, the cases $s \ge 4$ are implied by the case $s=3$. For this we need some notation. For $\ell \ge 0$, let
\begin{align}
\label{eq:Sdef}	
    S_{\ell} = \sum_{j \ge 0} e^{\mu j -\frac{\ln 2}{2} (j^2+2j\ell)} ,
    \qquad
    S_{\ell}' =  \sum_{j \ge 0} je^{\mu j -\frac{\ln 2}{2} (j^2+2j\ell)}
\end{align}
and, recalling the definitions~\eqref{eq:defxi} of the $\zeta_i$'s and~\eqref{eq:defT} of $T$, set
\[
    E_\ell = S_\ell \Big( -\ln(\zeta_\ell S_{\ell}) -\frac{\ln 2}{2}\ell +\frac{\ln 2 }{2} T(x) \Big) - \frac{\ln 2}{2} S_\ell' \enspace.
\]
	Then we have the following lemma giving an alternative condition for Lemma~\ref{lemma:partialprofiles2}.
	\begin{lemma}\label{lemma:toE}
		Let $i_0 \in\{1,2\}$ and $s \ge i_0$. Then $\varphi(s,x,i_0)>0$	is equivalent to $E_\ell >0$, where $\ell=s+1$.
	\end{lemma}
	\begin{proof}
		It follows from \eqref{eq:deflambdamu} that
\[
    1- \sum_{i=i_0}^s \zeta_i
    = \sum_{i \ge s+1} \zeta_i
    = \zeta_{s+1} \sum_{j \ge 0} \frac{\zeta_{j+s+1}}{\zeta_{s+1}}
    = \zeta_{s+1}S_{s+1} \enspace.
\]
		Furthermore, again from \eqref{eq:deflambdamu}, we obtain
{
\small
\[
    \sum_{i=i_0}^s i \zeta_i
    = T - \sum_{i \ge s+1} i \zeta_i
    = T - (s+1)\zeta_{s+1 }\sum_{j\ge 0}  \frac{(j+s+1)\zeta_{j+s+1}}{(s+1)\zeta_{s+1}}
    = T-(s+1)\zeta_{s+1} \Big(S_{s+1}+\frac{S'_{s+1}}{s+1}\Big).
\]
}
Plugging in these relationships, we find that, letting $\ell=s+1$,
		\begin{align}
			\varphi(s,x,i_0)& =- \Big(1-\sum_{i=i_0}^s \zeta_i \Big)   \ln  \Big(1-\sum_{i=i_0}^s \zeta_i \Big) + \frac{\ln 2}{2}\sum_{i=i_0}^s \Big( \zeta_i \big(i-T\big) \Big)\nonumber \\
			&=  -\zeta_{\ell}S_{\ell}  \ln(\zeta_\ell S_{\ell})+ \frac{\ln 2}{2} \Big( T-\ell\zeta_{\ell} \Big(S_{\ell}+\frac{S'_{\ell}}{\ell}\Big) -T (1- \zeta_{\ell}S_{\ell}) \Big) \nonumber \\	
			&= \zeta_\ell S_\ell \Big( -\ln(\zeta_\ell S_{\ell}) -\frac{\ln 2}{2}\ell +\frac{\ln 2 }{2} T\Big) - \frac{\ln 2}{2} \zeta_\ell S_\ell'. \nonumber  \label{eq:transtoS}
		\end{align}
		Since $\zeta_\ell>0$ for all $\ell$, the claim follows.
	\end{proof}
	
With this preparation at hand, we show that we can reduce the analysis for Lemma~\ref{lemma:partialprofiles2} to the case $s \le 3$.

\begin{lemma}\label{lemma:E4}
For $i_0 \in \{1,2\}$ and $x \in [0,1]$, if $E_4 >0$, then $E_\ell >0$ for all $\ell \ge 5$.
\end{lemma}
	
	\begin{proof}
		Suppose that $E_4 >0$, and note that by \eqref{eq:defxi} and \eqref{eq:Sdef},
		\begin{equation}\label{eq:Eequ}
			E_\ell = \sum_{j \ge 0} \Big(e^{\mu j- \frac{\ln 2}{2} (j^2+2j\ell)} \Big( \frac{\ln 2}{2}\ell^2 - \lambda - \mu\ell -\ln S_\ell  - \frac{\ln 2}{2}\ell + \frac{\ln 2}{2} T  -\frac{\ln 2 }{2}j\Big) \Big).
		\end{equation}
		 Further note that (interpreting $\ell$ as a real number)
		\[
		\frac{\text{d}}{\text{d} \ell} \Big(	 \frac{\ln 2}{2}\ell^2 - \mu\ell - \frac{\ln 2}{2}\ell \Big)= \ell \ln 2 - \mu - \frac{\ln 2}{2}.
		\]
		Recall from \S\ref{section:numerics} that for $i_0=1$, $\mu(0) \le2.69$, and for $i_0=2$, $\mu(0) \le 2.65$. By Lemma~\ref{lemma:monomu}, $\mu(x)$ is strictly decreasing for $x \in [0,1]$, so in both cases $\mu(x) \le 2.69$, and we have 
		\[
		\frac{\text{d}}{\text{d} \ell} \Big(	 \frac{\ln 2}{2}\ell^2 - \mu\ell - \frac{\ln 2}{2}\ell \Big) \ge \ell \ln 2 -2.69 - \frac{\ln 2}{2} \ge \ell \ln 2 - 3.04.
		\]
		This is strictly positive for $\ell \ge 4.5$, thus the minimum of the quadratic function $ \frac{\ln 2}{2}\ell^2 - \mu\ell - \frac{\ln 2}{2}\ell $ occurs at $\ell <4.5$, and so for $\ell \ge 5$ we have
		\[
		\frac{\ln 2}{2}\ell^2 - \mu\ell - \frac{\ln 2}{2}\ell  \ge  \frac{\ln 2}{2}4^2 - 4\mu - \frac{\ln 2}{2}\cdot 4. 
		\]	
Furthermore, it is clear from the  definition \eqref{eq:Sdef} that $S_{\ell}$ is decreasing in $\ell$, which means $-\ln S_\ell \ge -\ln S_4$ for $\ell \ge 5$. Plugging these bounds into \eqref{eq:Eequ}, we obtain for $\ell \ge 5$
		\[
		E_\ell \ge \sum_{j \ge 0} \Big(e^{\mu j- \frac{\ln 2}{2} (j^2+2j\ell)} \Big( \frac{\ln 2}{2}4^2 - 4 \mu- \frac{\ln 2}{2}\cdot 4-\ln S_4  -\lambda + \frac{\ln 2}{2} T  -\frac{\ln 2 }{2}j\Big) \Big).
		\]
		Now consider 
		\[ \frac{E_\ell}{S_\ell } \ge \sum_{j \ge 0} \Big( \frac{e^{\mu j- \frac{\ln 2}{2} (j^2+2j\ell)}}{\sum_{i \ge 0}e^{\mu i- \frac{\ln 2}{2} (i^2+2i\ell)} } \Big( \frac{\ln 2}{2}4^2 - 4 \mu- \frac{\ln 2}{2}\cdot 4-\ln S_4  -\lambda + \frac{\ln 2}{2} T  -\frac{\ln 2 }{2}j\Big) \Big).\]
		The expression on the right-hand side is a weighted average of the term in the brackets, which we denote by $R(j)$. Note $R(j)$ is decreasing in $j$ (the only part that depends on $j$ is $-\frac{\ln 2}{2}j$). If we decrease $\ell$, the weights in the weighted sum are shifted towards  larger $j$ (since for larger $j$, $\exp\{- j\ell \ln 2\}$ becomes disproportionally larger) -- so smaller terms $R(j)$ get more weight, and the average decreases. So for $\ell \ge 5$ we obtain 
		\[
		\frac{E_\ell}{ S_\ell } \ge \sum_{j \ge 0} \Big( \frac{e^{\mu j- \frac{\ln 2}{2} (j^2+2j\cdot 4)}}{\sum_{i \ge 0}e^{\mu i- \frac{\ln 2}{2} (j^2+2i\cdot 4)} } \Big( \frac{\ln 2}{2}4^2 - 4 \mu- \frac{\ln 2}{2}\cdot 4-\ln S_4  -\lambda + \frac{\ln 2}{2} T  -\frac{\ln 2 }{2}j\Big) \Big).
		\]
But this last expression is just $E_4/S_4 >0$, and so $E_\ell >0$ as required.
	\end{proof}
	The following corollary is now immediate from Lemmas \ref{lemma:toE} and \ref{lemma:E4}.
\begin{corollary}
Let $i_0 \in \{1,2\}, x \in [0,1]$. If $\varphi(3,x,i_0)>0$, then $\varphi(s,x,i_0)>0$ for $s \ge 4$.
\end{corollary}
It remains to verify Lemmas~\ref{lemma:partialprofiles2} and~\ref{lemma:partialprofiles1} for $s \le 3$. After our preparations this is not very difficult, but it will take some time: 
We study the function $\varphi(s,x,i_0)$ -- which is different for every choice of $s, i_0$ --  and use monotonicity properties from \S\ref{section:monotonicity}  to show that it is positive if it is positive for $s,x,i_0$ in a certain \emph{finite} set. We then check those finitely many points using the numerical approximations from \S\ref{section:numerics}. The  proof is in the appendix.	
\begin{lemma}
\label{lemma:checksle3}
Let either $i_0 = s=1$ and $x \in [0.04, 1]$, or $i_0 \in \{1,2\}$, $s \in \{2,3\}$ and $x \in [0,1]$. Then $\varphi(s,x,i_0)>0$. 
\end{lemma}
With this lemma, the proof of Lemmas~\ref{lemma:partialprofiles2} and~\ref{lemma:partialprofiles1} is complete.\qed

\subsection{Approximation by an integer profile}
Recall the definition of $L_0(n,k,t)$ given in \eqref{eq:L0def}, and of the rescaled version  $\tL_0(\rho,k,t)$ in \eqref{eq:L0tildedef}. We already saw in \eqref{eq:oldapprox} that $L_0(n,k,t)$ is a good approximation for $\ln E_{n,k,t}$. In this section we prove the promised sharpening of this relationship for appropriate values $n,k,t$. The first step is approximating the optimal continuous profile by an integer profile.
\begin{lemma}\label{lemma:integerapprox}
Let $n,k,t$ be integers so that $t =\alpha-\bigO(1)$ and $n/k = t-\Theta(1)$, and let $\boldsymbol{\xi}=(\xi_i(n,t,\rho))_i$ be the continuous profile maximizing $\tL_0(\rho,k,t)$ characterised in \eqref{eq:xiu} and~\eqref{eq:constraints_xiu}, where $\rho = n/k$. Then there is an integer profile
\[
    \bfk^* = (k_u^*)_{1\le u \le t}\in P_{n,k,t}
\]
with the following properties:
\begin{itemize}
    \item[a)] $\sum_u |k_u^* - \xi_{\alpha-u}k|  =\bigO(\sqrt{\ln n})$,
    \item[b)] there is an integer $u^* = \alpha - \bigO(\sqrt{\ln n})$ so that $k_u^*=0$ for all $u<u^*$, and
    \item[c)] $\ln \E_{1/2}[\bar X_{\bfk^*}] = L_0(n,k,t) +\bigO (\ln^{3/2} n)$.
\end{itemize}
\end{lemma}
\begin{proof}
Let $k_u = \xi_{\alpha-u} k \in \mathbb{R}$, where $1\le u \le t$. We will make a series of small changes to the $k_u$'s to obtain an integer profile.
Lemma~\ref{lem:monoxi} implies that $\ln \xi_i \le -\frac{i^2}{2}\ln 2 + \bigO(i)$ as $i$ increases, and so there is some constant $C'>0$ so that
\[
    \sum_{i > C'\sqrt{\ln n} } \xi_i < \frac 1{n^2} .
\]
We let $u^* = \alpha -  \lceil C' \sqrt{\ln n}\rceil$, and set $k^*_u = 0$ for all $u < u^*$, which gives b). In particular, 
\begin{equation}
\label{eq:tailsumk*k}    
    \sum_{1 \le u < u^*} |k_u^*- \xi_{\alpha-u}k|
    = \sum_{1\le u < u^*} \xi_{\alpha-u}k    
    < \frac{k}{n^2} <\frac 1n.
\end{equation}
We will now successively define some auxiliary values $k_u' \in \R$ and $k_u'' \in \mathbb{N}_0$ for all $u^* \le u \le t$ by making small changes to the values $k_u$. First we add what we removed from the $k_u$'s with $u<u^*$ to $k_{u^*}$, setting			
\begin{equation}\label{eq:kustarprime}
    k_{u^*}' = k_{u^*} +  \sum_{1 \le u < u^*} k_u  \in  \Big[k_{u^*},  k_{u^*} + \frac 1n \Big]. 
\end{equation}
To create the integer values  $k_u'' \in \mathbb{N}_0$, we start with~$k'_{u^*}$, round it down to $k''_{u^*} = \left \lfloor k'_{u^*} \right \rfloor \in \mathbb{N}_0$, and add $\delta_{u^*} = k'_{u^*} - k''_{u^*}$ to $k_{u^*+1}$, setting $k'_{u^*+1} = k_{u^*+1} +\delta_{u^*}$. We proceed like this for $u=u^*+1, u=u^*+2, \dots, u=t-1$, always rounding the current value $k_u'$ down to $k''_u = \left \lfloor k'_u \right\rfloor \in \mathbb N_0$ and adding the difference $\delta_u=k_u'-k_u''$ to $k_{u+1}$, setting $k'_{u+1}=k_{u+1}+\delta_u$. Finally, we set $k_t'' = k_t'$. 
Note that for all $u$, $|k_u'' - k_u| \le 1$.
		
We have made $\bigO(t-u^*)=\bigO(\sqrt{\ln n})$ `small changes' so far. By construction, we have not altered the sum $\sum_u k_u'' =\sum_u k_u = k$, and in particular $k''_t$ is also an integer. However, we have changed the sum $\sum_u k_u u = n$: the first of our changes (when we set $k_u^*=0$ for all $u<u^*$ and defined $k_{u^*}'$ via \eqref{eq:kustarprime}) increased this sum by at most $u^* \sum_{1\le u<u^*}k_u < u^*k/n^2 <1$. Each subsequent small change (when moving from $k_u'$ to $k_u''$ and adding the difference $\delta_u$ to $k_{u+1}$) increased the sum by $\delta_u \in [0,1)$. Since we made $t-u^*$ such changes, we have $\sum_u k''_u u  \in [n, n + t-u^*+1) $.
		
It is clear that we can fix this by performing up to $t-u^*+1 = \bigO(\sqrt{\ln n})$ further `small neighbour changes' where we increase some $k_u''$ by $1$ and decrease~$k_{u+1}''$ by $1$, where $u^* \le u < t$. (Since $\sum_u k_u'' = k$, there is certainly a large enough $k_{u+1}''$ so that we can do this while keeping all values non-negative.)  
This defines our final sequence $\bfk^* \in P_{n,k,t}$. As we made at most $2(t-u^*+1)$ small changes in total, we have
\[ \sum_{u^* \le u \le t} | k_u^* - k_u| \le 2(t-u^*+1)= \bigO(\sqrt{\ln n}),\]
which, together with~\eqref{eq:tailsumk*k}, establishes a). It only remains to check  for c) that $\ln \E_{1/2}[\bar X_{\bfk^*}]$ is not too far away from $L_0(n,k,t) =k \tL_0 (\frac nk, k, t)= k \tL (\frac nk, k, \mathbf{p})$, where $p_u = \xi_{\alpha-u}$ for all $u$. For this, first consider how much $L_{\bfk^*}$ has changed from 
\begin{equation}
\label{eq:ltildeterms}
 k  \tL \Big(\frac nk, k, \mathbf{p} \Big) 
    = L_{\mathbf{p}k}
    = n \ln n - n + k - \sum_{1\le u\le t} \xi_{\alpha-u}k \ln (\xi_{\alpha-u}k d_u).
\end{equation}
Since $n$ and $k$ are preserved by our changes, we only need to study how much the terms in the sum change. First of all, since $\sum_{u <u^*} \xi_{\alpha-u} k < {k}/{n^2} <  1/n $, by Jensen's inequality, 
		\[0 \ge \sum_{u<u^*} \xi_{\alpha-u}k \ln (\xi_{\alpha-u}k ) \ge \frac{1}{n} \ln (1/(n u^*)) = - \ln(n u^*) / n > -  \frac{\ln^2 n}{n} .\]
	Note further that each $d_u$ is at most $2^{t^2}t!=\exp(\bigO(t^2))$, and so $0 \le \sum_{u<u^*} \xi_{\alpha-u}k \ln (d_u) \le  \frac{1}{n} \bigO(t^2) = \bigO\big( \frac{\ln^2 n}{n} \big)$. Thus, our first change (setting $k_u^*=0$ for $u<u^*$ and adding the difference to $k_{u^*}$) does not change the sum in \eqref{eq:ltildeterms} by more than $\bigO\big( \frac{\ln^2 n}{n} \big)$. 
		
Next we examine how much the $\bigO(t-u^*+1) = \bigO(\sqrt{\ln n})$ small neighbour changes --- where we change $k_u'$ or $k_u''$ by up to $\pm 1$ and shift the difference to its neighbour $k'_{u+1}$ or $k''_{u+1}$ --- change the sum in \eqref{eq:ltildeterms}. First note that for $1 \le x \le n$, changing $x$ by at most $1$ increases or decreases $x \ln x$ by at most $\ln n$. 
For $x\in [0,1]$, we have $x \ln x = \bigO(1)$.
So the $\bigO(\sqrt{\ln n})$ small neighbour changes can change the value of the sum $\sum_{u^* \le u \le t} \xi_{\alpha-u}k \ln (\xi_{\alpha-u}k ) $ by at most $\bigO(\ln^{3/2} n)$.
Furthermore, for all $1\le u \le t$ we have $\ln d_{u+1} - \ln d_u = \ln(d_{u+1}/d_u) = u\ln 2 + \ln (u+1) = \bigO(u) = \bigO(\ln n)$. So our $\bigO(\sqrt{\ln n})$ small neighbour changes also only change the sum $\sum_{u \ge u^*} \xi_{\alpha-u}k \ln (d_{u} )$ by at most $\bigO(\ln^{3/2} n)$. Overall, we obtain
\begin{equation}\label{eq:Lkstar}
    L_{\bfk^*} = n \ln n - n + k - \sum_{1 \le u \le t} k_u^* \ln (k_u^* d_u) = L_0(n,k,t) + \bigO(\ln^{3/2} n).
\end{equation}
This is not quite what we want yet, as $L_{\bfk^*}$ is only an approximation for $\ln \E_{1/2}[\bar X_{\bfk^*}]$. However, using Stirling's approximation, it is easy to see that $\E_{1/2}[\bar X_{\bfk^*}] = L_{\bfk^*} +\bigO(\ln^{3/2} n) $. This follows by absorbing the logarithm of all factors $\sqrt{2\pi m}$ into the $\bigO(\ln^{3/2} n)$-term (noting that only $\bigO(\sqrt{\ln n})$ such factors exist, as $k_u^* = 0$ for all $u<u^* = \alpha - \bigO(\sqrt{\ln n})$). It follows from \eqref{eq:Lkstar} that
\[
\ln \E_{1/2}[\bar X_{\bfk^*}]=L_0(n,k,t) + \bigO(\ln^{3/2} n).
\]
\end{proof}

	In the next lemma, we show that colouring profiles which contain colour classes of size less than $\alpha-\bigO(\sqrt{\ln n})$ only contribute a negligible amount to $E_{n,k,t}$.	
	\begin{lemma}\label{lemma:smallclasses}
		Let $n,k,t$ be integers so that $t =\alpha-\bigO(1)$ and $n/k = t - \Theta(1)$. Let $P_{n,k,t}' \subset P_{n,k,t}$ be the set of all profiles where $k_s>0$ for some $s< \alpha - 10\sqrt{\ln n} $. Then
		\[
		\sum_{\bfk \in P_{n,k,t}'} \E_{1/2}[\bar X_\bfk] = o (E_{n,k,t}).
		\]
	\end{lemma}
	
	\begin{proof}
We define a map
\[
    \psi: P_{n,k,t}' \mapsto P_{n,k,t}
\]
as follows. Let $\bfk \in P_{n,k,t}'$.
We pick some $s<  \alpha - 10\sqrt{\ln n} $ and some $r> \alpha-\bigO(1)$ such that $k_s, k_r>0$; such an $r$ exists as the average colour class size $n/k$ is $\alpha-\bigO(1)$.  
Then we set 
		\begin{itemize}
			\item $\tilde k_s = k_s -1$ and $\tilde k_r = k_r -1$,
			\item if $r+s$ is even: $\tilde k_{(r+s)/2} = k_{(r+s)/2 } +2$,
			\item  if $r+s$ is odd: $\tilde k_{\left \lfloor (r+s)/2 \right \rfloor} = k_{\left \lfloor (r+s)/2 \right \rfloor} + 1$ and $\tilde k_{\left \lceil (r+s)/2 \right \rceil} = k_{\left \lceil (r+s)/2 \right \rceil} + 1$,
			\item for all other $u$, $\tilde k_u = k_u$,
		\end{itemize}
		and let $\psi(\bfk) = \tilde \bfk = (\tilde k_u)_u$. From the definition it is clear that $\psi(\bfk)  \in P_{n,k,t}$. Now consider
		\[
		\E_{1/2}[\bar X_{\psi(\bfk)}] = \frac{n!}{\prod_u \tilde k_u! u!^{ \tilde k_u}}2^{-\sum_u \tilde k_u {u \choose 2}} \enspace.
		\]
		First note that, as $r= \alpha-\bigO(1)$, $s< \alpha - 10\sqrt{\ln n}$  and $\alpha \le  2 \log_2 n$, for $n$ large enough
\begin{align}
\frac{\prod_u \tilde k_u! u!^{ \tilde k_u}}{\prod_u k_u! u!^{ k_u}}
     &\le k^{4} \cdot \frac{\big\lfloor \frac{(r+s)}{2}  \big \rfloor !\big\lceil \frac{(r+s)}{2}  \big \rceil !}{s!r!}
     \le e^{4 \ln n} \frac{{r+s \choose r}}{{r + s \choose \lfloor \frac{(r+s)}{2}  \rfloor}} \le e^{4 \ln n}2^{r+s}   \le e^{8\ln n}.
     \label{eq:productbound}
\end{align}
Further, we have
		\[
		2^{-\sum_u \tilde k_u {u \choose 2}} = 2^{-\sum_u  k_u {u \choose 2} +{r \choose 2} + {s \choose 2} - {\left \lfloor (r+s)/2 \right \rfloor \choose 2}- {\left \lceil (r+s)/2 \right \rceil \choose 2}} \ge 2^{-\sum_u k_u {u \choose 2} +\frac {r^2}{4} + \frac{s^2}{4} - \frac{rs}{2} -\frac 14}.  
		\]
		As $r= \alpha-\bigO(1)$ and  $s<\alpha-10 \sqrt{\ln n}$, for $n$ large enough
		\[
		\frac {r^2}{4} + \frac{s^2}{4} - \frac{rs}{2} -\frac{1}{4}= \frac{(r-s)^2-1}{4}\ge 20 \ln n.
		\]
		Together with \eqref{eq:productbound}, we obtain for $n$ large enough that
		\begin{equation}
			\E_{1/2}[\bar X_{\psi(\bfk)}]  \ge \E_{1/2}[\bar X_\bfk]  e^{-8 \ln n + (20 \ln 2)\ln n } > n \E_{1/2}[\bar X_\bfk]. \label{eq:boundaftersqrt}
		\end{equation}
Now for every $\tilde \bfk \in P_{n, k, t}$, there are at most $t^2$ profiles $\bfk \in P_{n,k,t}'$ so that $\psi(\bfk) = \tilde \bfk$ --- this is because once we pick $1\le r, s \le t$, for which there are at most $t^2$ possibilities, we can reconstruct $\bfk$ from $\tilde \bfk$. So it follows from  \eqref{eq:boundaftersqrt} that
		\[
		\sum_{\bfk \in P_{n,k,t}'}  \E_{1/2}[\bar X_\bfk] < \frac{t^2}{ n } \sum_{\tilde \bfk \in P_{n,k,t}}  	\E_{1/2}[\bar X_{\tilde \bfk}] = \frac{t^2}{n}E_{n,k,t} = o(E_{n,k,t}).
		\]
	\end{proof}
	
	We now have all the preparations in place to improve upon~\eqref{eq:oldapprox}.
	\begin{lemma}\label{lemma:improvedapproximation}
		Let $n,k,t$ be integers so that $t =\alpha-\bigO(1)$ and $n/k = t - \Theta(1)$.  Then
		\[
		\ln(E_{n,k,t}) = L_0(n,k,t) + \bigO(\ln^{3/2} n).
		\]
	\end{lemma}
	
	\begin{proof}
		Let $u_0 = \min\big(u^*, \,\alpha - 10\sqrt{\ln n} \big) = \alpha - \bigO(\sqrt{\ln n})$, where $u^*$ was defined in Lemma~\ref{lemma:integerapprox}. Let $\tilde E_{n,k,t}$ be the expected number of all unordered colourings with a profile $\bfk \in P_{n,k,t}$ where $k_u=0$ for all $u<u_0$, and let $\tilde P_{n,k,t}$ be the set of all such profiles. Then it follows from Lemma~\ref{lemma:smallclasses} that $E_{n,k,t} - \tilde E_{n,k,t} = o(E_{n,k,t}) $, and so
		\[
		\ln(E_{n,k,t}) = \ln(\tilde E_{n,k,t}) +o(1).
		\]
		Note that $|\tilde P_{n,k,t}| \le n^{\alpha - u_0} = \exp(\bigO(\ln^{3/2} n))$, and so
		\[\ln(\tilde E_{n,k,t}) = \max_{\bfk \in \tilde P_{n,k,t}} \ln \E_{1/2}[\bar X_{\bfk}] + \bigO(\ln^{3/2} n).
		\]
		It is not hard to see that for every  $\bfk \in \tilde P_{n,k,t}$, $\ln \E_{1/2}[\bar X_{\bfk}]$ is within $ \bigO(\ln^{3/2} n)$ of $L_\bfk$: this follows by using Stirling's approximation and absorbing each of the $\bigO(\alpha-u_0) = \bigO{(\sqrt{\ln n})}$ terms of the form $\sqrt{2\pi m}$ into the $ \bigO(\ln^{3/2} n)$-term. So we obtain
		\[
		\ln(E_{n,k,t}) =  \max_{\bfk \in \tilde P_{n,k,t}}  L_\bfk+ \bigO(\ln^{3/2} n).
		\]
		Since ${\tilde P_{n,k,t}} \subset P_{n,k,t}^0$, this is at most $L_0(n,k,t) + \bigO(\ln^{3/2} n)$. On the other hand, from Lemma~\ref{lemma:integerapprox} we know that it is also at least $L_0(n,k,t) + \bigO(\ln^{3/2} n)$, and we are done.
	\end{proof}
	
\subsection{The optimal profile is well-behaved}
We will need the following lemma in order to apply Theorem~\ref{theorem:general}, which will then imply Theorems \ref{twopointrestricted} and \ref{theorem:announcedbounds}. We fix $m= \lfloor N/2  \rfloor$ throughout. 
The lemma implies that, under certain conditions, the (near) optimal profile $\bfk^*$ from Lemma~\ref{lemma:integerapprox} is tame and fulfils conditions~\eqref{eq:lowerboundexpectation} and \eqref{eq:lowerboundbeta} from Theorem~\ref{theorem:general} with room to spare.
\begin{lemma}\label{lemma:kstartame}
	Let $a \in \{\alpha-1, \alpha-2\}$, and suppose that~$k=k(n) = \textbf{k}_a+o(n/\ln^2 n)$, where~$\textbf{k}_a$ is the $a$-bounded first moment threshold  defined in \eqref{ktdef}. Then there is an $a$-bounded $k$-colouring profile $\bfk^*=\bfk^*(n)$ so that all of the following hold.
	\begin{itemize}
 \itemsep1pt
            \item[a)] 	$\ln \E_{1/2}[\bar X_{\bfk^*}] 
		= L_0(n,k,a) +\bigO (\ln^{3/2} n)$.
		\item[b)] There is an increasing function $\gamma: \mathbb N_0 \rightarrow \R$ with $\gamma(y) \rightarrow \infty$ as $y \rightarrow \infty$ so that for all large enough~$n$, and for all $1 \le u \le a$,
		\[
		\kappa_u^* := k_u^* u / n	< b^{-(\alpha-u)\gamma(\alpha-u)}.
		\]
		\item[c)] If 	\[
		\ln \E_{1/2}[\bar{X}_{\bfk^*}] \ge 1.1 \log_2^2 n 
		\]
then 
		\begin{equation} \label{eq:bound1}
			\ln \E_{m}[\bar{X}_{\bfk^*}] \ge \Theta (\ln^2 n),
		\end{equation}
	so in particular \eqref{eq:lowerboundexpectation} holds.
		\item[d)] If $\mu_a \ge n^{1.05}$, then for any ${\delta}>0$ there is a constant $C_\delta>0$ so that if $n$ is large enough, then for all colouring profiles $\boldlambda \le \boldkappa^*=(\kappa_u^*)_{1\le u \le a}$ such that ${\delta} \le \sum_{1\le u \le a} \lambda_u \le 1-{\delta}$, 
\begin{equation}\label{eq:exponentiallower}
	\E_p\left[\bar{X}_{\boldlambda}\right] \ge \exp \big( C_\delta n \big) \prod_{1 \le u \le a}{k_u^* \choose \ell_u}^2,   \end{equation}
where $\ell_u = \lambda_u n/u$. In particular, \eqref{eq:lowerboundbeta} holds.

	\end{itemize}
\end{lemma}

The constraint $\mu_a \ge n^{1.05}$ in d) may seem strange at first, but it is not simply an artefact of our proofs. It turns out that \eqref{eq:exponentiallower} (or the weaker~\eqref{eq:lowerboundbeta}) does \emph{not} hold if we allow $\mu_a < n^{1+x_0-\varepsilon}$, where $x_0 \approx 0.02905$ is the constant given by \eqref{eq:x0}. 
More specifically, in this case the profile $\bfk^*$ optimising the expected number of colourings contains an unrealizable sub-profile: the expected number of partial unordered colourings consisting of exactly $k_a^*$ disjoint $a$-sets is $\exp(-\Theta(n))$, so in particular we do not have the required lower bound \eqref{eq:lowerboundbeta}.

\begin{proof}
By Lemma~\ref{lemma:averagecolourclass} and since $k=\boldsymbol{k}_a+o(n /\ln^2 n)$, we have
\begin{equation} \label{eq:rationk}
\frac{n}k
= \alpha_0-1-\frac{2}{\ln 2}+o(1) = \alpha - 2 - \Theta(1),
\end{equation}
so we may apply Lemma~\ref{lemma:integerapprox} with $t=a$. Let $\bfk^*$ be the near optimal $a$-bounded $k$-colouring profile defined in (the proof of) that lemma. Then we obtain immediately that $\bfk^*$ fulfills a). 

We continue with the proof of b). Letting $u^*$ as Lemma~\ref{lemma:integerapprox}b), for $u < u^* = \alpha - \bigO(\sqrt{\ln n})$ we know that $k^*_u=0$, so these~$u$ certainly satisfy b). For $u^* \le u \le a$, by Lemma~\ref{lemma:integerapprox}a), uniformly
	\[
	\kappa_u^* = \frac{u k_u^*}{n} =\frac{u \xi_{a-u}k}{n} + \bigO((\ln^{3/2} n) / n).
	\]
Recall the definition~\eqref{eq:xiu} of $\xi_{i}$, where, by Lemma~\ref{lemma:mubound}, $\mu_n, \lambda_n = \bigO(1)$. Since $h_n(i)=-\big(\frac{\ln 2}{2}+o(1)\big)i^2$ by Lemma~\ref{lemma:hasymp} we infer that $\xi_i = \exp\big(-\frac{\ln 2}2i^2+o(i^2) + \bigO(i+1)\big)$ uniformly in $i \le \alpha-u^*$ for $n \rightarrow \infty$. As $i=\alpha-u=\bigO(\sqrt{\ln n})$,  there is an increasing (actually, linear) function $\gamma$ with $\gamma(y) \rightarrow \infty$ as $y \rightarrow \infty$ so that for all large enough $n$,
\[
    \kappa_u^* 	< b^{-(\alpha-u)\gamma(\alpha-u)},
\]
establishing b).  
	
Let us turn to the proof of c). Note that $k \sim {n}/{2 \log_2 n}$ by Lemma~\ref{lemma:averagecolourclass}, and so the number~$f$ of `forbidden edges' of the profile $\bfk^*$ (edges within the parts of a vertex partition with profile $\bfk^*$, which may not be present if the partition defines a colouring) satisfies 
	\[
	f = \sum_u {u \choose 2} k^*_u \sim \alpha^2 k /2 \sim n \log_2 n,
	\]
	and so by Lemma~\ref{gnmlemma},
	\[
	\ln \E_{m}[\bar{X}_{\bfk^*}]= \ln \E_{1/2}[\bar{X}_{\bfk^*}]- \log_2^2 n + o(\ln^2 n) \ge  0.1 \log_2^2 n + o(\ln^2 n),
	\]
	giving \eqref{eq:bound1}.
It remains to show d), so assume $\mu_a \ge n^{1.05}$ and let $\delta>0$. Suppose that $\boldlambda \le \boldkappa^*$ such that $ \delta \le \lambda \le 1-{\delta}$, where $\lambda = \sum_{u^* \le u \le a} \lambda_u$. Lemma~\ref{expectationlemma} gives
	\begin{equation}\label{eq:expectationpartialprofile}
		\ln \E_{1/2}\left[\bar{X_{\boldlambda}}\right]
  =\varphi(\boldlambda)n + o(n),
	\end{equation}
	where
 \[
	\varphi(\boldlambda)=-\parenth{1-\lambda} \ln\parenth{1-\lambda} +\frac{\ln 2}{2} \sum_{u^* \le u \le a} \lambda_u \parenth{\alpha_0-1-\frac{2}{\ln 2}-u}.
	\]
It clear from this expression that if we fix $\lambda \in [\delta,1-\delta]$, then $\varphi(\boldlambda)$ is minimised if we choose~$\lambda_u$ as large as possible for larger values~$u$, that is, if we pick $\lambda_u=\kappa_u^*$ for all $\alpha- s< u \le a$ for some integer $s=s(\lambda)$, and set $\lambda_{\alpha-s} = \lambda -\sum_{u > \alpha-s} \kappa^*_u \in[0, \kappa_{\alpha-s}^*)$.
Note that it is certainly possible that $s(\lambda) = \alpha-a$, in which case 
there is no $u$ such that $\lambda_u = \kappa_u^*$ and $\lambda_a = \lambda <\kappa_a^*$.
The next step is to show that if $n$ is large enough, then $s(\lambda) \le S(\delta)$ for some $S(\delta)$ that only depends on~$\delta$. 
To see this, first note that by Lemma~\ref{lemma:integerapprox}, using also again that $\xi_i = \exp\big(-\frac{\ln 2}2i^2+o(i^2)+\bigO(i+1)\big)$ and $u^* = \alpha - \bigO(\sqrt{\ln n})$,
\begin{equation*}
\label{eq:kappaconvergence1}
    \sum_u |\kappa_u^* - \xi_{\alpha-u}|
    \le
    \sum_{u^*\le u\le a} \Big|\frac{k_u^* u }{n} - \xi_{\alpha-u} \Big| + o(1)
    = \frac 1k \sum_{u^*\le u\le a} \Big|k_u^*(1+o(1)) - \xi_{\alpha-u} k \Big|
    + o(1)
    = o(1).
\end{equation*}
Recall the definition of the $\zeta_i$'s in~\eqref{eq:defxi}, where we set $x=\alpha_0-\alpha$ and $i_0=\alpha-a \in \{1,2\}$ as before. 
Using  Lemma~\ref{lemma:convergencemulambda} and the triangle inequality we obtain that
	\begin{equation}\label{eq:convergencekappaxi}
		\sum_u |\kappa_u^* - \zeta_{\alpha-u}| = o(1).	
	\end{equation}
For every fixed $r$ and $i_0$, the sum $\sum_{i_0 \le i \le r} \zeta_i(x)$ is continuous in $x \in [0,1]$, as $\mu(i_0,x), \lambda(i_0,x)$ are, and tends to $1$ as $r \rightarrow \infty$ for every fixed $x \in [0,1]$. By applying Lemma~\ref{lemma:technicalsequence} we deduce that 
\[
	\min_{x \in [0,1 ], i_0 \in \{1,2\}} \left(\sum_{i_0 \le i \le r} \zeta_i(x) \right) \rightarrow 1
 \quad\text{as}\quad r \rightarrow \infty,
\]
and so there is an integer $S(\delta)$ such that  $\sum_{i_0 \le i \le r} \zeta_i(x) \ge 1-{\delta}/{2}$ for all $x \in [0,1]$ and $r \ge S(\delta)-1$.
Together with~\eqref{eq:convergencekappaxi}, if $n$ is large enough, then $\sum_{i_0 \le i \le r} \kappa^*_{\alpha-i}(x) > 1-\delta$ for all $x\in [0,1]$ and $r \ge S(\delta)-1$.
Since for an integer $s = s(\lambda)$ as above we know that $\sum_{i_0 \le i \le s(\lambda)-1}\kappa^*_{\alpha-i}(x)\le \lambda \le 1-\delta$, we have $s(\lambda)-1 < S(\delta)-1$, so $s(\lambda) \le S(\delta)$.

We continue with estimating the value of $\varphi(\boldlambda)$ from~\eqref{eq:expectationpartialprofile}.
Let us first consider the case $s(\lambda) = i_0 = \alpha-a$, that is, we consider the profile $\boldlambda$ such that $\lambda_a = \lambda < \kappa_a^* = \zeta_{i_0} + o(1)$, and $\lambda_u = 0$ for all other $u$. Then
\[
    \varphi(\boldlambda)
    =-(1-\lambda)\ln(1-\lambda) + \frac{\ln 2}{2}\lambda\left(x + i_0 -1-\frac2{\ln 2}\right)
    =: \widetilde{\varphi}(\lambda) ,
\]
which is obviously a continuous and concave function of $\lambda$. Moreover, $ \widetilde{\varphi}(0)=0$, and by continuity and~\eqref{eq:convergencekappaxi},
\[
    \widetilde{\varphi}(\kappa_a^*)
    = \widetilde{\varphi}(\zeta_{i_0}) + o(1)
    \stackrel{\eqref{eq:phisxi}}{=} \varphi(i_0,x,i_0) + o(1).
\]
If $i_0 = 2$, then we obtain by applying Lemma~\ref{lemma:partialprofiles2} 
that $\widetilde{\varphi}(\kappa_a^*)$ is bounded away from 0 and so {as $\delta \le \lambda < \kappa_a^*$}, there is a constant $c_\delta > 0$ such that $\varphi(\boldlambda) > c_\delta$ for sufficiently large $n$. On the other hand, if $i_0 = 1$, since we assumed that $\mu_a = \mu_{\alpha-i_0} \ge n^{1.05}$, by applying Lemma~\ref{prop:basicboundsisets} we infer that $x \ge 0.05+o(1) \ge 0.04$ for $n$ large enough. As before, it follows now from Lemma~\ref{lemma:partialprofiles1} that there is a constant $c'_\delta > 0$ such that $\varphi(\boldlambda) > c'_\delta$ for sufficiently large $n$, and we thus have established that for $C'_\delta = \min\{c_\delta, c'_\delta\}$,
\begin{equation}
\label{eq:s(lambda)=i_0}
    \varphi(\boldlambda) \ge C'_\delta > 0, \quad s(\lambda) = i_0.
\end{equation}
It remains to give a lower bound for $ \varphi(\boldlambda)$ in the case $i_0 < s(\lambda) \le S(\delta)$, where we consider the profile $\boldlambda$ defined by $\lambda_u=\kappa_u^*$ for all $\alpha-s+1 \le u \le a$, and $\lambda_{\alpha-s} = \lambda -\sum_{u > \alpha-s(\lambda)} \kappa^*_u \in[0, \kappa_{\alpha-s(\lambda)}^*)$. Given any $i_0 \le r\le S(\delta)$,  let $\boldkappa^{*, (r)}$ be the partial colouring profile obtained by truncating $\boldkappa^*$ at $\alpha-r$, that is, $\kappa^{*, (r)}_u = \kappa^*_u \mathbbm{1}_{u \ge \alpha-r}$ for $1 \le u \le a$. 
Now for $i_0 \le r \le S(\delta)-1$, consider the mapping
\[
    z
    \mapsto
    \varphi\big(\boldkappa^{*, (r)} + z(\boldkappa^{*, (r+1)}-\boldkappa^{*, (r)})\big),
    \quad
    z\in[0,1],
\]
which, as $z$ increases from 0 to 1, transforms $\boldkappa^{*, (r)}$ to $\boldkappa^{*, (r+1)}$ by increasing the $(r+1)$st entry from 0 to $\kappa^*_{\alpha-r-1}$. 
This mapping is concave and so
\[
    \min_{z \in [0,1]} \varphi\big(\boldkappa^{*, (r)} + z(\boldkappa^{*, (r+1)}-\boldkappa^{*, (r)})\big)
    \ge 
    \min \{ \varphi(\boldkappa^{*, (r)}), \varphi(\boldkappa^{*, (r+1)}) \},
\]
from which we immediately obtain that
\begin{equation*}\label{eq:boundphilambda1}
    \varphi(\boldlambda)
    \ge \min \{ \varphi(\boldkappa^{*, (s(\lambda)-1)}), \varphi(\boldkappa^{*, (s(\lambda))}) \}.
\end{equation*}
Define $\boldsymbol{\zeta}^{(s)}$ by $\zeta^{(s)}_u = \zeta_{\alpha-u} \mathbbm{1}_{u \ge \alpha-s}$. Then by \eqref{eq:convergencekappaxi} and continuity,
\begin{equation*} \label{eq:boundphi2}
    \varphi(\boldlambda)
    \ge \min \{ \varphi(\boldsymbol{\zeta}^{(s(\lambda)-1)}), \varphi(\boldsymbol{\zeta}^{(s(\lambda))}) \} + o(1)
    \stackrel{\eqref{eq:phisxi}}{=} \min\{\varphi(s(\lambda)-1,x,i_0), \varphi(s(\lambda),x,i_0)\} + o(1).
\end{equation*} 
We have already established that $\varphi(i_0,x,i_0) \ge C'_\delta$ for some constant $C'_\delta > 0$. Moreover, whenever $2 \le i_0+1 \le s(\lambda) \le S(\delta)$, by applying Lemma~\ref{lemma:partialprofiles2} we obtain that $\varphi(\boldlambda) \ge c_{s(\lambda)}$ for some constant $c_{s(\lambda)} > 0$ and sufficiently large $n$, and so $\varphi(\boldlambda) \ge c_{s(\lambda)} \ge C''_\delta$ for some $C''_\delta > 0$ and all $\delta \le \lambda \le 1-\delta$.

Overall, recalling~\eqref{eq:expectationpartialprofile} and setting $C_\delta := \frac 12 \min\{C'_\delta, C''_\delta\}$, we have just established that
\[
    \ln \E_{1/2}\left[\bar{X}_{\boldlambda}\right] =\varphi(\boldlambda)n + o(n) \ge 2 C_\delta n + o(n)
\]
	for all $\boldlambda \le \boldkappa^*$ with $\delta \le \lambda \le 1-\delta$. The rest is routine. As $a - u^* = o(\ln n)$,
\[
    \E_{1/2}\left[\bar{X_{\boldlambda}}\right]
    \ge \exp \big( 2C_\delta n + o(n)\big)
    = \exp \big( 2C_\delta n + o(n)\big) \cdot 2 ^{2k (a-u^*+1)}
    \ge
    \exp \big( 2C_\delta n + o(n)\big)\prod_{u}{k_u \choose \ell_u}^2,
\]
and~\eqref{eq:exponentiallower} is established for large enough $n$. 
\end{proof}

\section{Proof of Theorems \ref{twopointrestricted}, \ref{theorem:announcedbounds} and \ref{twopointequitable}}
\label{sec:proofsmainthms}

\subsection{Proof of Theorem  \ref{twopointrestricted} }
Let $p=1/2$ and set $a=a(n)=\alpha(n)-2$. Note that by \eqref{eq:muinfinity} and \eqref{expectationu},
\begin{equation}\label{eq:mualpha-2}
\mu_a \gg \frac{n^2}{ \ln^ 2 n}.
\end{equation}
Recall the definition \eqref{ktdef} of the 
$a$-bounded first moment threshold $\textbf{k}_a=\mathrm{min_k}  (E_{n,k, a} \ge 1)$. 
By applying Lemmas~\ref{lemma:improvedapproximation},~\ref{lemma:averagecolourclass} and~\ref{lemma:onemorecolour} we can pick a sequence
\[
    k=k(n) \in \{\textbf{k}_{a}, \textbf{k}_{a}-1\},
\]
so that
\begin{equation}\label{eq:propertyofk}
	E_{n,k+1,a}\ge \exp\Big(\frac{2}{\ln 2}\ln^2 n +o(\ln^2 n)\Big) \rightarrow \infty, \quad  \quad \text{but } 	E_{n,k-1,a} \rightarrow 0 \quad \text{ as $n \rightarrow \infty$.}
\end{equation}
By Lemma~\ref{gnmlemma}, $\E_m(X_\bfk) \lesssim \E_{1/2}(X_\bfk)$ for any $a$-bounded $(k-1)$-colouring profile $\bfk$. In particular by the second property, whp $G_{n, m}$ contains no $a$-bounded $(k-1)$-colouring, so whp
\[
    \chi_{a} (G_{n, m}) \ge k.
\]
Let $\bfk^*$ be the $a$-bounded {$(k+1)$-colouring} profile from Lemma~\ref{lemma:kstartame}. To complete the proof of Theorem~\ref{twopointrestricted} we show that whp $X_{\bfk^*}>0$ {in $G_{n,m}$, which implies $\chi_{a} (G_{n, m}) \le k+1$.} For this, note that from Lemmas~\ref{lemma:kstartame} and \ref{lemma:improvedapproximation}, we obtain for large $n$ that
\begin{equation}\label{eq:expinf}
	\E_{1/2}[X_{\bfk^*}] 	\ge \exp\Big(\frac{2}{\ln 2}\ln^2 n +o(\ln^2 n)\Big) \ge \exp(1.1 \log_2^2 n).
\end{equation}
Note that by \eqref{eq:mualpha-2}, $\mu_a \gg n^2 / \ln^2 n$. So it follows from Lemma~\ref{lemma:kstartame} that $\bfk^*$ is tame and fulfils \eqref{eq:lowerboundexpectation} and \eqref{eq:lowerboundbeta}. We may apply  Theorem~\ref{maintheorem}, giving that whp $X_{\bfk^*}>0$.
\qed

\subsection{Proof of Theorem  \ref{theorem:announcedbounds} }
First note the following lemma.
\begin{lemma}\label{lemma:lowerbound}
Let $p=1/2$, and $a=a(n) \in \{\alpha(n)-2, \alpha(n)-1\}$. Then, whp, 
   \begin{equation}\label{eq:chilower}
	\chi_a(G_{n,1/2}) \ge \textbf{k}_a-1.
\end{equation} 
\end{lemma}
\begin{proof}
This follows by the first moment method, noting that by the definition of the first moment threshold, $E_{n,\textbf{k}_a-1,a} < 1$, and that by Lemmas~\ref{lemma:onemorecolour}, ~\ref{lemma:averagecolourclass} and~\ref{lemma:improvedapproximation}, $E_{n,\textbf{k}_a-2,a} \le \exp(-\Theta(\ln^2 n)) \rightarrow 0$.   
\end{proof}
Now let $p=1/2$ and $a=a(n)$ so that $n^{1.1} \le \mu_a \le n^{2.9}$. In particular, 
by \eqref{expectationu} and the definition of $\alpha$, this implies
\begin{equation*}
a(n) \in \{\alpha(n)-2, \alpha(n)-1\}.
\end{equation*}
So by Lemma~\ref{lemma:lowerbound}, whp $\chi_a(G_{n,1/2}) \ge \textbf{k}_a-1$. Similarly, by setting $k= \textbf{k}_{a}+1$, we find analogously to \eqref{eq:expinf} that for $n$ large enough
\[\E_{1/2}[\bar{X}_{\bfk^*}] \ge\exp\Big(\frac{2}{\ln 2}\ln^2 n +o(\ln^2 n)\Big) > \exp(1.1\log_2^2 n),\]
where $\bfk^*$ is the $a$-bounded $k$-colouring profile from Lemma~\ref{lemma:kstartame}.  As we also have $\mu_a \ge n^{1.1} > n^{1.05}$, by Lemma~\ref{lemma:kstartame}, we may apply Theorem~\ref{theorem:general}  to $\bfk^*$, and as $k=\bigO(n/\ln n)$ (see Lemma~\ref{lemma:averagecolourclass}) we obtain that
\begin{equation*}
\Pb_m (\bar{X}_{\bfk^*} >0) \ge \exp \big(   - \bigO \big ( k^2 / \mu_a\big)+ o(1)\big) \ge \exp(-n^{0.9})
\end{equation*}
for $n$ large enough. Thus we have, for $n$ large enough,
\[
\Pb_{m}(\chi_a (G_{n,m}) \le \textbf{k}_a+1) \ge  \exp(-n^{0.9}).
\]
A simple coupling argument (pointed out by Alex Scott in the context of \cite{heckel2019nonconcentration}) shows that this is also true in $G_{n,1/2}$ (rather than $G_{n,m}$ with $m=\lfloor N/2 \rfloor$) if we increase the number of colours slightly. Namely, start with $G \sim G_{n,m}$ and independently sample $E \sim \mathrm{Bin}(N, 1/2)$, and then either add $E-m$ uniform edges to $G$ or remove $m-E$ edges from $G$ uniformly. Then the resulting graph $G'$ has the distribution $G_{n, 1/2}$. Now suppose that in $G$, $\bar{X}_{\bfk^*}>0$. Then $G$ has an $a$-bounded $(\textbf{k}_a+1)$-colouring consisting of $\bigO(n/\ln n)$ colour classes of size $\bigO(\ln n)$ each. If $E - m <n \ln \ln n$, say, which holds whp, then whp at most $\ln n (\ln \ln n)^2 \le \ln^2 n$ added edges `spoil' a colour class (i.e., both endpoints are in the same colour class) in this colouring. This can be fixed by introducing at most $\ln^2 n$ additional colours. Thus we have for $n$ large enough
\begin{equation}\label{eq:lowerboundprob}
\Pb_{1/2}(\chi_a (G_{n, 1/2}) \le \textbf{k}_a+1+\ln^2 n) \ge  \exp(-n^{0.89}).
\end{equation}
Note that the resulting chromatic number of the coupled $G_{n, 1/2}$ is at most $\textbf{k}_a+1+\ln^2 n$ whp for every fixed choice of $G_{n,m}$ with $\bar{X}_{\bfk^*}>0$, so to get the lower bound \eqref{eq:lowerboundprob} it is enough to establish this whp-statement; we do not need to show we only introduce at most a certain number of additional colours with probability $1 - o(\exp(-n^{0.9}))$.
We now employ a standard technique to show that if we slightly increase $k$ further, we can boost the lower bound in \eqref{eq:lowerboundprob} to $1-o(1)$. Namely, note that if $G$ is a graph on $n$ vertices and we obtain $G'$ by changing some or all of the edges incident with one particular vertex, then $|\chi_a(G)-\chi_a(G')| \le 1$ (because we can always just give that vertex its own additional colour). Therefore, we can apply the Azuma-Hoeffding inequality to the vertex exposure martingale (see \cite{janson:randomgraphs}, \S2.4) to obtain for all $t \ge0$ that
\begin{equation}\label{eq:azuma}
\Pb_{1/2}(|\chi_a(G_{n,1/2}) - \E_{1/2}[\chi_a]| \ge t) \le 2 \exp(-t^2/(2n)).
\end{equation}
Applying \eqref{eq:azuma} with $t=n^{0.96}$ and comparing to \eqref{eq:lowerboundprob}, we find that for $n$ large enough we must have 
\[
\textbf{k}_a \ge \E_{1/2}[\chi_a] - n^{0.96}.
\]
Thus, setting $k'=\textbf{k}_a+2 n^{0.96}$ and applying  \eqref{eq:azuma} again with $t=n^{0.96}$, we obtain
\begin{equation*}
	\Pb_{1/2}(\chi_a(G_{n, 1/2}) > k') \le P(|\chi_a - \E_{1/2}[\chi_a]| > n^{0.96}) \le \exp(n^{-0.91}) \rightarrow 0, 
\end{equation*}
and so whp $\chi_a \le k' = \textbf{k}_a+2n^{0.96}$. Together with \eqref{eq:chilower}, this gives (with room to spare) $\chi_a (G_{n, 1/2}) = \textbf{k}_{a-1}+\bigO(n^{0.99})$ whp as required. \qed
~\\

\begin{remark}
A closer inspection of the proof reveals that the error term $\bigO(n^{0.99})$ could be sharpened to $\bigO(n^{3/2} / (\sqrt{\mu_a} \ln n){+n^{1/2} \omega(n)})$ {where $\omega(n)\rightarrow \infty$ is arbitrary,} and of course in the case $\mu_a \gg n^2 / \ln^2 n$ we have the much sharper Theorem~\ref{twopointrestricted} anyway.
\end{remark}

\subsection{Proof of Theorem~\ref{twopointequitable}}
Let $G \sim G(n, \roundDown{pN})$ with $0<p<1-1/e$ constant. Let ${Z}^{=}_k$ and $\bar{Z}^{=}_k$ denote the number of ordered and unordered equitable $k$-colourings of $G$, respectively. To prove Theorem~\ref{twopointequitable}, we need to find a sequence $k=k(n)$ so that whp $\sum_{\ell \le k-1} Z^=_{\ell}=0$ and $Z^=_{k+1} >0 $. Let 
\[
y=y(n) = \alpha_0-1-\frac 2 {\ln 2} = 2 \log_bn -2 \log_b \log_b n-2 \log_b 2, 
\]
and
\[
k^*=k^*(n)=\min \left\{ k \mid k \ge \frac{n}{y+1} \text{ and } \E \left[ \bar{Z}^{=}_k \right] > \frac{1}{\ln n} \right\}.
\]
We will show that the sequence $k^*$ has the desired properties in Lemmas \ref{lemma:nosmaller} and \ref{lemma:colouringexists} below. 
We will need the following result from \cite{heckel:equitable}. 
\begin{lemma}[\cite{heckel:equitable}, Lemmas~3.1 and~3.3]\label{lemmafrompaper}
Let $k=k(n)$ be a sequence of integers such that
\[
k= \frac{n}{y+z}
\]
with $z=\bigO(1)$. Then
\[ \E \left[ \bar{Z}^=_{k} \right] = b^{-\frac{z}{2}n+o(n)}
\quad
\text{and}
\quad
\frac{ \E \left[ \bar{Z}^{=}_{k+1}\right]}{\E \left[ \bar{Z}^{=}_{k}\right]} \ge \exp \left( \Theta \left(\ln n \ln \ln n \right)\right).
\]
\end{lemma}
From the first part of Lemma \ref{lemmafrompaper} we obtain directly that
\begin{equation}\label{eq:kstarorder}
    k^* = \frac{n}{y+o(1)}.
\end{equation}
\begin{lemma}With high probability, \label{lemma:nosmaller}
	\[\sum_{\ell \le k^*-1} Z^=_{\ell}=0.\]
\end{lemma}
\begin{proof}
Using the first moment method, it is not hard to show that there is no equitable colouring with fewer than ${n}/(y+1)$ colours --- and this also follows from the main result in \cite{heckel2018chromatic}, where it was shown that the (normal) chromatic number of $G_{n,p}$ is ${n}/(y+o(1))$ whp (and as the property of having a colouring with a certain number of colours is monotone, this result can be transferred easily to $G_{n,m}$). 
	So in particular, whp
	\[
	\sum_{\ell < {n}/(y + 1)} Z^=_{\ell}=0.
	\]	
By the definition of $k^*$ and \eqref{eq:kstarorder} we have $\E[\bar{Z}^=_{k^*-1}] < {1}/{\ln n}$. As  $k^* = \frac{n}{y+o(1)}$, Lemma \ref{lemmafrompaper} implies that
\begin{align*}
    \E \Big[\sum_{\frac{n}{y+ 1 } \le \ell \le k^*-1 } \bar{Z}^=_{\ell} \Big] &\le  \E \left[\bar{Z}^=_{k^*-1}\right] \sum_{\frac{n}{y + 1 } \le \ell \le k^*-1 }\exp{\left(-\Theta\left((k^*-1-\ell) \ln n \ln \ln n\right)\right)} \\
    &\sim \E \left[\bar{Z}^=_{k^*-1} \right] < \frac{1}{\ln n} =o(1),
\end{align*}
and so whp no equitable colouring with at most $k^*-1$ colours exists.
\end{proof}

\begin{lemma} \label{lemma:colouringexists}
	With high probability, ${Z}^{=}_{k^*+1}>0$.
\end{lemma}
\begin{proof}
It suffices to check that equitable $(k^*+1)$-colourings are $(\alpha-2)$-bounded, tame (as in Definition \ref{deftame} with $a=\alpha-2$) and fulfil \eqref{eq:lowerboundexpectation} and \eqref{eq:lowerboundbeta}. Then Theorem~\ref{maintheorem} implies that whp there exists an equitable $(k^*+1)$-colouring.  For the remainder of this proof, let
	\[
	k=k^*+1= \frac{n}{y+o(1)} \enspace .
	\]
Given $k$ colours, the colouring profile of an equitable $k$-colouring is uniquely defined: letting $\Delta={n}/{k}-\roundDown{n/k}$, an equitable $k$-colouring of $n$ vertices consists of exactly 
\[
    k_\text{L}=\Delta k
\]
larger colour classes of size $\roundUp{n/k}$ and 
\[
    k_\text{S}=(1-\Delta) k
\]
smaller colour classes of size $\roundDown{n/k}$. Note that if $k$ divides $n$, $\Delta=0$ and all $k$ colour classes are of size ${n}/{k}$. Now observe  that
\[\frac nk =y(n)+o(1) = \alpha_0 - 1 -\frac {2}{\ln b}+o(1).\]
Since $p<1-1/e$, we have $1 + {2}/{\ln b} >3$. So for $n$ large enough,
	\[
	\alpha- \left \lceil \frac nk \right \rceil = \roundDown{\alpha_0}- \roundUp{\alpha_0-1-\frac{2}{\ln b}+o(1)} \ge 2,
	\]
	and so an equitable $k$-colouring is $a$-bounded for $a=a(n)=\alpha-2$ as required. Next we check that $\bfk$ is tame (with $a=\alpha-2$) and fulfils \eqref{eq:lowerboundexpectation} and \eqref{eq:lowerboundbeta}. For this, note that
	\[
	\alpha- \left \lfloor \frac nk \right \rfloor  \le \alpha_0 - \left(\alpha_0-1-\frac{2}{\ln b}-2\right)= 3+\frac{2}{\ln b}. 
	\]
	So, writing $k_u$ for the number of colour classes of size $u$ in an equitable $k$-colouring, part a) of Definition~\ref{deftame} certainly holds since $k_u=0$ for $u < \alpha-3-{2}/{\ln b}$.  For \eqref{eq:lowerboundexpectation} (which implies part b) of Definition~\ref{deftame}), by Lemma~\ref{lemmafrompaper} we have
	\[
	\ln \E_m[\bar{Z}_k^=] \ge \ln \E_m[\bar{Z}_{k-1}^=] + \Theta(\ln n \ln \ln n) > \ln \Big(\frac{1}{\ln n } \Big)+ \Theta(\ln n \ln \ln n) \gg \ln n.
	\] 
	as required. It only remains to show \eqref{eq:lowerboundbeta}. 
So fix $\delta>0$ and let $\lS \le \kS$, $\lL \le \kL$ such that
\[
    \delta \le \lambdaS + \lambdaL \le 1-{\delta}
\]
where $\lambdaS= \lS \floornk /n$ and $\lambdaL= \lL \ceilnk /n$. Denote by $\bfl$ the colouring profile consisting of $\lS$ colour classes of size 
\[
    \floornk=\frac{n}{k}-\Delta = y-\Delta+o(1) = \alpha_0-1-\frac{2}{\ln b}-\Delta+o(1)
\]
and $\lL$ colour classes of size
\[
    \ceilnk = \frac nk -\Delta+1 = \alpha_0-\frac{2}{\ln b}-\Delta+o(1).
\]
By Lemma \ref{expectationlemma}, letting $\lambda= \lambdaS+\lambdaL$,
\begin{align}
    \ln \E_p [\bar{X}_\bfl] 	& = \Big(-(1-\lambda)\ln(1-\lambda)+\frac{\ln b}{2}(\Delta\lambdaS+(\Delta-1)\lambdaL )\Big)n+o(n) \nonumber \\
    &= \left(-(1-\lambda)\ln(1-\lambda)+\frac{\ln b}{2} \Delta \lambda -\frac{\ln b}{2} \lambdaL \right)n+o(n) \label{jty}.
\end{align}
Suppose first that $\lambda\le \Delta$. In that case, bounding $\lambda_L \le \lambda$, we get
\begin{align}
    \ln \E_p [\bar{X}_\bfl] &\ge \left(-(1-\lambda)\ln(1-\lambda)+\frac{\ln b}{2} \Delta \lambda -\frac{\ln b}{2} \lambda \right)n+o(n) \nonumber\\
    &\ge (1-\lambda)\Big(-\ln(1-\lambda)-\frac{\ln b}{2} \lambda\Big)n+o(n). \label{eq:bound1lambda}
\end{align}
Now suppose that $\lambda> \Delta$, then continuing from (\ref{jty}) and bounding $\lambda_L \le k_L \lceil n/k \rceil / n =\Delta+o(1) $, 
	\begin{align}
	\ln \E_p [\bar{X}_\bfl] &\ge \left(-(1-\lambda)\ln(1-\lambda)+\frac{\ln b}{2} \Delta \lambda -\frac{\ln b}{2} \Delta \right)n+o(n) \nonumber \\
	&\ge (1-\lambda)\Big(-\ln(1-\lambda)-\frac{\ln b}{2} \lambda\Big)n+o(n). \label{eq:bound2lambda}
\end{align}
So we get the same lower bound in both cases \eqref{eq:bound1lambda} and \eqref{eq:bound2lambda}. Now note that that $-\ln(1-\lambda)-\frac{\ln b}{2} \lambda>0$ for all $\lambda \in (0,1)$, since $\frac{\ln b}{2} < \frac 12 <1$ as $p<1-1/e$. In particular, $-\ln(1-\lambda)-\frac{\ln b}{2}\lambda $ is bounded away from $0$ for $\delta \le \lambda \le 1-\delta$, and so as $1-\lambda \ge \delta$,
	\begin{align*}
		\ln \E_p [\bar{X}_\bfl] &\ge \Theta (n) .
	\end{align*}
	On the other hand,
	$
	{\kS \choose \lS}^2{\kL \choose \lL}^2 \le 2^{2\kS+2\kL} = 2^{2k}= \exp \big( \bigO\left( n / \ln n \right) \big).
	$
	In particular, for $n$ large enough,
		\[
	\E_p [\bar{X}_\bfl] \ge \exp \left( \ln^6 n \right)  {\kS \choose \lS}^2{\kL \choose \lL}^2, \]
 giving \eqref{eq:lowerboundbeta}.
\end{proof}

	\bibliographystyle{plainnat}

\pagebreak

\appendix
\section{Appendix: Omitted proofs}
	
\subsection{Proof of Lemmas~\ref{prop:basicboundsisets} and~\ref{prop:muurelativetomua}}

Recall that $\alpha_0= 2 \log_b n -2 \log_b \log_b n +2 \log_b (e/2) + 1$. Let $a=a(n)=\alpha_0-\bigO(1)$, and write $a=\alpha_0-y$ for some $y=y(n)= \bigO(1)$. Then we obtain that
\begin{equation}
\label{eq:qa-1}
    q^{a-1} = b^{-a+1} = \Big(\frac{2\log_b n}{en}\Big)^2 b^y.
\end{equation}
Let us begin with establishing~\eqref{eq:muinfinity}. By applying Stirling's formula we obtain that
\begin{align*}
    \mu_a
    &= \binom{n}{a}q^{\binom{a}{2}}
    \sim \frac{n^a}{a!}q^{a \choose 2}
    \sim \frac{1}{\sqrt{2 \pi a}}
        \left(\frac{en q^{(a-1)/2}}{a}\right)^a.
\end{align*}
Plugging in~\eqref{eq:qa-1} gives
\begin{equation}
\label{eq:muaasympt}
    \mu_a
    \sim \frac{1}{\sqrt{2 \pi a}} \left( \frac{2 b^{y/2} \log_b n }{a} \right)^a
    \sim \Big(\frac{n}{\log_bn}\Big)^y \exp \Big( 2\log_b \log_b n -\frac12 \ln\log_b n + \bigO(1)\Big). 
\end{equation}	 
If  $p < 1 - e^{-4}$, we have $2/\ln b - 1/2 > 0$. If we set $a=\alpha=\lfloor \alpha_0 \rfloor$, we have $y= \alpha_0-\alpha \ge 0$, so it follows that $\mu_\alpha = \omega(1)$, giving~\eqref{eq:muinfinity}. To see the other claims, write $u = a - x'$. Then, for $0.1 \alpha \le u \le 10\alpha$,
\[
    \frac{\mu_u}{\mu_a}
    = \binom{n}{u}\binom{n}{a}^{-1}  q^{\binom{u}2 - \binom{a}2}
    = \Theta\Big(\frac{\ln n}n\Big)^{x'} q^{-x'a + (x'+1)x'/2}.
\]
Since $q^{-a} = \Theta(n^2/\ln^2n)$, the claim~\eqref{expectationu} follows immediately and Lemma~\ref{prop:muurelativetomua} is established. In order to see~\eqref{eq:mut}, again set $a=\alpha$ and note that for $x=\alpha_0-u = \bigO(1)$ and so $x' = \alpha-u= \bigO(1)$, we readily obtain from~\eqref{expectationu} and~\eqref{eq:muaasympt} that
\[
    \mu_u
    = \mu_\alpha \cdot \Theta\left(\frac{n}{\ln n}\right)^{x'}
    = \Theta\left(\frac{n}{\ln n}\right)^{x'+y}\exp(\bigO(\ln \ln n)),
\]
and the claim follows from observing that $x'+y= x$.

\subsection{Proof of Lemmas \ref{expectationlemma}, \ref{lemmaupperbound} and \ref{expectationlowerbound}}
		
We begin with some general estimates. Let $U = \{u \mid \kappa_u \neq 0 \}$, then $U \subseteq [0.1 \alpha,  10\alpha]$. In particular, $u=\Theta(\alpha)=\Theta(\ln n)$ for $u \in U$. This implies $k = \Theta(n / \ln n)$. Furthermore, if we fix any partition $\Pi$ with profile $\boldkappa$ and consider the event that $\Pi$ is a colouring, then this forbids $f = \sum_u k_u{u \choose 2}= \bigO(n \ln n)$ edges. It follows from Lemma \ref{gnmlemma} that $\E_m[\bar X_\boldkappa]$ and $\E_p[\bar X_\boldkappa]$ agree up to a factor $\exp(\bigO(\ln^2 n))$, so it suffices to prove the three lemmas for $\E_p[\bar X_\boldkappa]$ only.

We begin with an auxiliary estimate. First, note that for any $1 \le u \le 10\alpha$,
\[
\frac{n!}{(n-u)!} = \prod_{i=0}^{u-1}(n-i) = n^{u} \exp\left(\bigO\Big(\frac{u^2}{n} \Big) \right)= n^{u} \exp \left( \bigO \Big( \frac{\ln^ 2n}{n} \Big)\right).
\]
As usual we let $\kappa=\sum_u \kappa_u$ and $k_u = \kappa_un / u$ so that $\sum_u u k_u = \kappa n$. Then
\begin{align*}
    \prod_{u \in U} \left(\frac{n!}{(n-u)!}\right)^{k_u} &= n^{\sum_{u \in U} uk_u} \exp \left( \bigO \left(  \frac{k \ln^ 2 n}{n} \right)\right)= n^{\kappa n} e^{\bigO( \ln n )}.
\end{align*}
Using this and writing as usual $\mu_u = {n \choose u}q^{u \choose 2}$, we obtain that
\begin{align}
\label{eq:wert}
    \E_p[\bar{X}_\bfk] =   \frac{n!}{(n-\kappa n)!\prod_{u \in U} (u!^{k_u}k_u!)}q^{\sum_{u \in U} k_u {u \choose 2}}  = \frac{n!}{(n-\kappa n)!n^{\kappa n}} \prod_{u \in U}\frac{\mu_u^{k_u}}{k_u!}   e^{\bigO( \ln n )}.
\end{align}
By applying Stirling's formula $n! \sim \sqrt{2 \pi} n^{n+1/2} / e^n$, 
\begin{equation}
\label{eq:wert2}
    \frac{n!}{(n-\kappa n)!n^{\kappa n}} = (1-\kappa)^{-(1-\kappa)n}e^{-\kappa n} e^{\bigO( \ln n )}; 
\end{equation}
	here we let $0^0:=1$ and note that the factor $\exp (\bigO(\ln n))$  takes care of the case $(1-\kappa)n \nrightarrow \infty$. As $|U|=\bigO(\ln n)$,  we again have by Stirling's formula (with $\Theta(1)$ taking care of the case $k_u \nrightarrow \infty$), and  as $k_u \le k \le n$,
	\[
	\prod_{u \in U} k_u! = \prod_{u \in U} \left( \Theta(1) \frac{k_u^{k_u + 1/2}}{e^{k_u}}\right)= \prod_{u\in U} \left( \frac{k_u}{e}\right)^{k_u} \exp \left(\bigO(\ln^2 n)\right).
	\]
	Suppose that $d=d(n)$ is an integer function such that $d(n)=\alpha_0-\bigO(1)$. Then using \eqref{eq:wert}, \eqref{eq:wert2} and \eqref{expectationu}, we have
	\begin{align}
		\E_p[\bar{X}_\bfk]
		&= (1-\kappa)^{-(1-\kappa)n}e^{-\kappa n}
		\prod_{u \in U}\left(\frac{\mu_d e}{k_u} \left(\Theta\left(\frac{n}{ \ln n} b^{-\frac{d-u}{2}}\right)\right)^{d-u}\right)^{k_u}  \exp \left(\bigO( \ln^2 n) \right).   \label{split}
	\end{align}
	Now set $d(n)=\alpha(n)$ and let $R=\sum_u (\alpha-u)k_u + \frac{n \ln \ln n}{\ln n}$, then with \eqref{eq:defx2} and \eqref{eq:theta} this gives	
\begin{align}
\label{cont}
    \E_p[\bar{X}_\bfk]
    = (1-\kappa)^{-(1-\kappa)n}e^{-\kappa n} \prod_{u \in U}\left(k_u^{-1} n^{\alpha_0-u} (\ln n )^{-\alpha+u}b^{-\frac{(\alpha-u)^2}{2}}\right)^{k_u}  \exp \left(\bigO (R) \right) .
\end{align}
As $k_u= \kappa_u  n/u = \kappa_u \Theta({n}/{\ln n} ) $ for $u \in U$ and $k = \Theta(n/\ln n)$,
\begin{align*}
    \prod_{u \in U} k_u^{-k_u}
    = \Theta \left( \frac{n}{\ln n} \right)^{-k} \prod_{u \in U} \kappa_u^{-k_u}
    = n^{-k} \exp\left[-\Theta \left( \frac{n}{\ln n}\right) \sum_{u \in U} \kappa_u \ln \kappa_u+\bigO\left( \frac{n\ln \ln n}{\ln n}\right)\right].
\end{align*}
By Jensen's inequality, since $-x \ln x$ is concave, $\sum_u \kappa_u =\kappa \in [0,1]$ and $-\kappa \log \kappa <1$, 
\[
    0 \le -\sum_{u \in U} \kappa_u \ln \kappa_u \le -|U| \frac{\kappa}{|U|} \ln \left( \frac{\kappa}{|U|} \right)= \kappa \left(\ln |U|-\ln \kappa\right) \le \ln|U|+1 = \bigO\left(  \ln \ln n\right),
\]
and so
\begin{align*}
    \prod_{u \in U} k_u^{-k_u} 
    = n^{-k} \exp\left(\bigO \Big(\frac{n \ln \ln n}{\ln n}\Big)\right).
\end{align*} 
Plugging this into~\eqref{cont} gives
\begin{align}
\label{cont2}
    \E_p[\bar{X}_\bfk]
    &= (1-\kappa)^{-(1-\kappa)n}e^{-\kappa n} \prod_{u \in U}\left( n^{\alpha_0-u-1}(\ln n)^{-\alpha+u}b^{-\frac{(\alpha-u)^2}{2}}\right)^{k_u}  \exp(\bigO(R) ). 
\end{align}
With this at hand we can prove Lemma \ref{expectationlemma}. By assumption $\sum (\alpha-u)^2\kappa_u$ is bounded, and so
\[
    \sum_u (\alpha-u)k_u
    = \bigO \left( n/\ln n\right)\quad \text{ and } \quad \sum_u (\alpha-u)^2k_u = \bigO \left( n/\ln n\right).
\] 
Plugging this into \eqref{cont2}, and using that $\alpha_0= \alpha+\bigO(1) = 2 \log_b n + \bigO(\ln\ln n)$ and thus \[\sum_{u \in U} (\alpha-u)k_u \ln n = \frac{\ln b}{2} \sum_{u \in U} (\alpha-u) u k_u + \bigO\Big(\sum_{u \in U} (\alpha-u)^2k_u + n  \frac{ \ln\ln n}{\ln n}\Big) ,\] we obtain
\begin{align*}
	\E_p[\bar{X}_\bfk]
	& = (1-\kappa)^{-(1-\kappa)n}e^{-\kappa n}
	\prod_{u \in U}\left( n^{\alpha_0-u-1}\right)^{k_u}  \exp \left(\bigO\Big( \frac{n \ln \ln n}{\ln n}\Big) \right) \\
	&= \exp\Big[(-(1-\kappa)\ln(1-\kappa) -\kappa)n +\sum_{u \in U} (\alpha_0-u-1)k_u  \ln n+\bigO\Big(\frac{n \ln \ln n}{\ln n}\Big)\Big] \\
	&= \exp\Big[\Big(-(1-\kappa)\ln(1-\kappa) -\kappa +\frac{\ln b}{2}\sum_{u \in U} (\alpha_0-u-1)\kappa_u \Big)n+\bigO\Big(\frac{n \ln \ln n}{\ln n}\Big)\Big]\\
	&= \exp\Big[\varphi(\bfk)n+\bigO\Big(\frac{n \ln \ln n}{\ln n} \Big)\Big],
\end{align*}
as claimed in Lemma \ref{expectationlemma}.  

We continue with the quick proof of Lemma \ref{lemmaupperbound}, where we assume that $\kappa=1$ and that the average colour class size $n/k$ is larger than $\alpha_0-1-{2}/{\ln b} + C$ for some $C>0$. Then
\[
    \sum_u k_u u
= n
>  k\left( \alpha_0-1-\frac2{\ln b} + C\right) 
= \sum_u k_u \left(\alpha_0-1-\frac2{\ln b} + C \right)  .
\]
This implies that
\[
\sum_u k_u \left(\alpha_0-u-1 \right)   \le \left(\frac{2}{\ln b} - C \right) k,
\]
and in particular $\sum_{u \in U} (\alpha-u)k_u  \le \bigO(n/ \ln n)$. So, continuing from (\ref{cont2}), as $\kappa=1$ and since $b^{-(\alpha-u)^2/2}\le 1$,
\begin{align*}
	\E_p[\bar{X}_\bfk]
	&\le e^{-n} n^{ \left(\frac{2}{\ln b} - C \right) k}  \exp \left(\bigO\Big( \frac{n \ln \ln n}{\ln n}\Big) \right).
\end{align*}
Using once more that
$k < n / (\alpha_0-1-2/{\ln b} + C)
= \frac{n}{2 \log_b n} \left(1+\bigO\left( \frac{\ln \ln n}{\ln n} \right) \right)
$, this gives for sufficiently large $n$ that
\begin{align*}
	\E_p[\bar{X}_\bfk]&\le n^{ - C  k}  \exp \left(\bigO\Big( \frac{n \ln \ln n}{\ln n}\Big) \right)
    \le \exp\left(-\frac{C\ln b}{2}n +o(n)\right),
\end{align*}
and Lemma \ref{lemmaupperbound} is established with room to spare.

We move on to Lemma  \ref{expectationlowerbound}. Recall that we assume that $k_u=0$ for all $u<u^* \sim 2 \log_b n \sim \alpha$ and for all $u>a$, where $a(n) = \alpha_0 - \bigO(1)$ is such that $\mu_{a}\ge n^{1+\epsilon}$.
Using~\eqref{eq:mut} together with the latter condition we obtain that $a \le \alpha_0-1-\epsilon+o(1)$. Set $\delta' = \epsilon - \delta >0$. We continue from \eqref{split}, setting $d= a+1$. By \eqref{eq:mut}, this implies that $\mu_d \ge n^{\epsilon - \delta'/2}$ if $n$ is large enough. Note that $k_u \le k= \bigO(n/\ln n)$, so we obtain
\begin{align*}\E_p[\bar{X}_\bfk]&\ge (1-\kappa)^{-(1-\kappa)n}e^{-\kappa n} \prod_{u \in U}\left(n^{\epsilon-\delta'/2}\Theta\Big(\frac{n}{\ln n}  \Big) ^{a-u} b^{-\frac{(a+1-u)^2}{2}}\right)^{k_u}   \exp \left(\bigO ( \ln^2 n)\right).
\end{align*}
For all $u \in U$, $u \sim 2 \log_b n \sim a$, so $a-u=o(\ln n)$. Whenever $0\le r = o(\ln n)$, we obtain the uniform estimate
\[
\Theta\Big(\frac{n}{\ln n}  \Big) ^{r} b^{-\frac{(r+1)^2}{2	}} = n^{r+o(r+1)}. 
\]
Therefore,
\begin{align*}
\ln	\E_p[\bar{X}_\bfk]
	& \ge -(1-\kappa) \ln (1-\kappa) n - \kappa n + \ln n  \sum_u \Big(\epsilon - \frac {\delta'} 2 + a-u+o(a-u+1)\Big) k_u  + \bigO(\ln^ 2n)  .
\end{align*}
Note that $k_u =\kappa_u n / u \ge \kappa_u n /{2 \log_ b n}$. Furthermore, as $u \le a$, we have (for $n$ large enough) $a-u +o(a-u+1)\ge- \delta'/4$, and so as $\epsilon -{3\delta'}/{4} >\delta$,
\begin{align*}
	\E_p[\bar{X}_\bfk]
	&\ge \exp  \Big(\big(-(1-\kappa) \ln (1-\kappa)  - \kappa + \frac{\delta \ln b}{2} \kappa \big)n + \bigO(\ln^ 2n)  \Big), 
\end{align*} 
completing the proof of Lemma~\ref{expectationlowerbound}.
\qed

\subsection{Proof of Lemma \ref{lemma:smallsubprofiles}}
Let $U= \{u: k_u \neq 0\}$. Our assumptions guarantee that $u \sim 2 \log_b n$ for all $u \in U$. Let $\boldlambda \le \boldkappa$. By applying Lemma \ref{expectationlowerbound} to $\boldlambda$ with $\delta=\epsilon/2$ we obtain, uniformly in $\boldlambda$,
\[
\E_p[\bar{X}_\bfl] \ge \exp \left( \tilde\varphi_{\epsilon/2}(\lambda) n + \bigO(\ln^ 2 n) \right).
\]
Note that $\tilde \varphi_{\epsilon/2}(0)=0$. Moreover, $\tilde \varphi_{\epsilon/2}'(x) = \ln (1-x) + (\epsilon \ln b)/4$
is strictly positive for $x\in[0, 1-b^{-\epsilon/4})$ and monotone decreasing. Thus, 
setting $C=C(\epsilon) = (1-b^{-\epsilon/4})/2 > 0$, we obtain that $\tilde \varphi_{\epsilon/2} (\lambda) n = \Theta(\lambda n)$ uniformly for  $\lambda \in (0, C)$. By assumption $\lambda \ge \ln^{-3} n$ and we obtain 
\begin{equation} \label{eq:lowerboundforlambda}
	\E_p[\bar{X}_\bfl] \ge   \exp \left( \Theta(\lambda n) \right) \ge \exp \left( \Theta(n/\ln^3 n) \right).
\end{equation}
Note that $k \le n/u^* \sim n / (2\log_b n)$, so letting $r=n/u^*$,
\[
\prod_u {k_u \choose \ell_u}^2 \le {k \choose \ell}^2 \le \left( \frac{ e k}{\ell} \right) ^{2\ell}\le \left( \frac{ e r}{\ell} \right) ^{2\ell}.
\]
Let $\lambda' = \ell/r = \sum_u \ell_u/r  = \sum_u n\lambda_u/(ur) \sim  \sum_u \lambda_u = \lambda$. Then
\[
    \prod_u {k_u \choose \ell_u}^2
    \le \exp \left(- 2r\lambda' \ln \lambda' + \bigO(r\lambda' ) \right).
\]
As $\lambda' \sim \lambda \ge \ln^{-3} n$, we have $0 < -\log \lambda' = \bigO(\ln \ln n)$, and so as $r = \bigO(n / \ln n)$,
\[
    \prod_u {k_u \choose \ell_u}^2 \le \exp \left(\bigO \Big(\lambda \frac{n \ln \ln n}{\ln n} \Big) \right). 
\]
Together with \eqref{eq:lowerboundforlambda},
\[
\frac{\prod_u {k_u \choose \ell_u}^2}{\E_p\left[ {\bar{X}_{\boldsymbol{\bfl}}} \right]} \le  \exp\left(-\Theta \big(\lambda n\big) \right)  \le \exp\left(-\Theta \Big(\frac{n}{\ln^3 n}\Big) \right).
\] \qed

\subsection{Proof of Lemma \ref{gnmlemma}} 

By Stirling's formula,
\begin{align*}
	\frac{{N-x \choose m}}{{N \choose m}} = \frac{(N-x)!(N-m)!}{(N-m-x)!N!} \sim \frac{(N-x)^{N-x}(N-m)^{N-m}}{(N-m-x)^{N-m-x}N^N}= \left(\frac{N-m}{N} \right)^x\frac{(1-\frac{x}{N})^{N-x}}{(1-\frac{x}{N-m})^{N-m-x}}.
\end{align*}
As $m= pN+\bigO(1)$ 
and $x=o(n^{4/3})=o(N)$, the first term is asymptotically equal to $q^x$. For the second term, we calculate slightly more carefully using $1-y=e^{-y-{y^2}/{2}+\bigO(y^3)}$; then $(1-\frac{x}{N})^{N-x}$ is asymptotically equal to $e^{-x+{x^2}/{2N}}$ and similarly $(1-\frac{x}{N-m})^{N-m-x} \sim e^{-x+{x^2}/{2(N-m)}}$. So the fraction is asymptotically
\begin{align*}
	\exp \left( \frac{x^2}{2N}-\frac{x^2}{2(N-m)}\right)  \sim \exp \left( -\frac{(b-1)x^2}{n^2}\right),\end{align*}
using $N-m=qN+\bigO(1)$, $b=1/q$ and $x=o \left(n^{4/3} \right)$.
\qed

\subsection{Proof of Lemma \ref{ablemma}} By Stirling's formula and as $1+x \le e^x$,
\[
\frac{(b-a)!}{b!} \lesssim \frac{(b-a)^{b-a}e^a}{b^b} = b^{-a}e^a\left(1-\frac ab \right)^{b-a} \le b^{-a}e^{a^2/b}.
\]

\qed 

\subsection{Proof of Lemma \ref{lemma:technicalsequence}}

First of all, note that $\min_{x \in [0,1]} s_i(x)$ exists because $[0,1]$ is compact and $s_i(x)$ is continuous. 
Let $f:[0,1]^2 \rightarrow [0,1]$ be such that $f(x,1)=1$ and $f(x, 1-1/i) = s_i$ for all $x \in [0,1]$ and $i \in \mathbbm{N}$, and interpolating linearly between these points for every $x$. Then $f(x,y)$ is continuous in both $x$ and $y$, and since $[0,1]^2$ is compact, $f$ is uniformly continuous. 
Let $\epsilon >0$, then there exists a $\delta >0$ so that for all $x$, and all $y>1-\delta$, 
\[
|f(x,1) - f(x, y)| = |f(x,y) -1| < \epsilon.
\]		
Thus, for all $y>1-\delta$,
\[
|\min_{x \in [0,1]} f(x, y) - 1|  < \epsilon
\]		
(this minimum exists because $f( \cdot, y)$ is continuous and $[0,1]$ compact).
So $\min_{x \in [0,1]} f(x, y) \rightarrow 1$ as $y \rightarrow 1$. It follows that
\[
\min_{x \in [0,1]} s_i(x)
= \min_{x \in [0,1]} f(x, 1-1/i) \rightarrow 1 \text{ as  $i \rightarrow \infty$.}
\]
\qed

\subsection{Proof of Lemma \ref{easylemma1}}

Part \eqref{lemparta} follows immediately from the facts that $g, \gid \le f = \bigO(n \ln n)$ and $\ntr, r_1 \le n$. For \eqref{lempartb1}, as $u^* \sim 2\log_b n \sim a$, we have $\ntr=\sum_u ut_u  \sim at$, and as $\eta = (\ntr - r_1)/\ntr \le 1$,
\[
T(2) = \frac{e^{2 \eta} \left(\sum_u t_u{u \choose 2} \right)^2}{{\ntr\choose 2} q} = \bigO \left(\frac{a^4 t^2}{a^2 t^2} \right)=\bigO(\ln^2 n).
\]
For \eqref{lempartb}, first note that
\begin{equation*}
	T(x) = \frac{e^{x \eta} \left(\sum_u t_u{u \choose x} \right)^2}{{\ntr\choose x} q^{x \choose 2}} \le \frac{e^{x \eta} \left(t\frac{a^x}{x!} \right)^2}{{\ntr\choose x} q^{x \choose 2}} =:T'(x).
\end{equation*}
Then, as $\ntr = (1-\lambda)n \ge n^{1-c}$ and $t \sim \ntr/a = \Theta(\ntr / \ln n)$ and $\eta<1$,
\[
T'(3) =  \bigO \left(\frac{t^2 a^6}{n_{\mathrm{tr}}^3} \right)= \bigO \left( \frac{\ln^3 n}{t}\right)=o(1).
\]
For $3 \le x \le 4(\alpha-u^*)=o(\log n)$, as $\ntr\ge n^{1-c}$ and $\alpha-u^*=o(\ln n)$, if $n$ is large enough, then
\[
\frac{ T'(x+1) }{T'(x)}=\frac{e^\eta a^2}{q^x (\ntr-x)(x+1)}= O \left( \frac{ \ln^2 n }{q^{4(\alpha-u^*)}\ntr}\right) \le n^{-1+2c}=o(1),
\]
and so for all $3 \le x \le 4(\alpha-u^*)$,
\[
T(x) \le T'(x) \le T'(3) =  \bigO \left( \frac{\ln^3 n}{t}\right)=o(1)
\]
and further
\[
\sum_{x=3}^{4(\alpha-u^*)} T(x) \le \sum_{x=3}^{4(\alpha-u^*)} T'(x) = \bigO(T'(3))=o(1).
\]
Finally, for \eqref{lempartc}, note that by Lemma~\ref{lem:ra-1},
\begin{equation}\label{eq:fewterms}
	\sum_{x > 4(\alpha-u^*), u,v}r^{u,v}_x = \sum_{x \ge u^*-1}r_x  \le 2 \ln^3 n.
\end{equation}
Furthermore, as $x,u,v \le a=\bigO(\ln n)$ and $\eta \le 1$,
\[
T(x,u,v)= \frac{{t_ut_v{{u \choose x}}{v \choose x}}e^{x \eta}}{{\ntr \choose x}q^{x \choose 2}} \le t^2 a^{2x}e^x q^{-{x \choose 2}}\le \exp \left( \bigO (\ln^2 n)\right).
\]
So together with \eqref{eq:fewterms}, we get
\begin{align*}\prod_{x > 4(\alpha-u^*), u,v}\frac{T(x,u,v)^{r^{u,v}_x}}{r^{u,v}_x!} \le \exp \left( \bigO (\ln^2 n) \right) ^ {\sum_{x \ge u^*-1, u,v}r^{u,v}_x}= \exp\left(\bigO(\ln ^5 n) \right).
\end{align*} 
\qed

\subsection{Proof of Lemma \ref{easylemma2}}
As $\lambda,\eta < {\ln^{-3}n}$, we have
\begin{equation}\label{eq:ntrbound2}
    \ntr = (1-\lambda)n > n - {n}/{\ln^ 3n}
    \quad
    \text{and}
    \quad
    {\ntr-r_1}=\eta \ntr \le \eta n < {n}/{\ln^3 n}.
    \end{equation}
Furthermore, as $u^* \sim a$, we have $\lambda n = \sum_u \ell_u u \sim a \ell$, and with $\lambda < {\ln^{-3} n}$,
\begin{align}
	{\gid} &= \sum_{u^* \le u \le a} \ell_u {u \choose 2} \le \ell a^2 = \bigO(\lambda na) = \bigO\left(\frac{n}{\ln^2 n}\right),\nonumber \\
	{\gt} &= \sum_{2 \le x \le a} r_x{x \choose 2} \le a \sum_{2 \le x \le a} x r_x=a(\ntr-r_1)= \bigO\left(\frac{n}{\ln^2 n}\right). \label{eq:gidbound}
\end{align}
So $g=\gid+\gt = \bigO ({n}/{\ln^2 n} )$, and as $f=\Theta(n \ln n)$, we obtain
\begin{align*}
    F_3 (\bfl, \bfr)
    &= -\frac{2(b-1)f \big(f-\bigO \big(\frac{n}{\ln^2 n}\big)\big)+2 \big(f - \bigO \big(\frac{n}{\ln^2 n}\big)\big)^2}{n^2}
    =-\frac{2bf^2}{n^2} + \bigO \left(\frac{1}{\ln n}\right), 
\end{align*}
establishing~\eqref{teil1}. For \eqref{teil2}, note that as $e^{2\eta}=1+\bigO(\eta)=1+\bigO(\ln^{-3} n)$, together with \eqref{eq:ntrbound2},
\[
T(2) =  \frac{ \left(\sum_u t_u{u \choose 2} \right)^2}{\frac{n^2}{2} q} \left(1+\bigO(\ln^{-3 }n)\right)=\frac{2b}{n^2}(f- \gid)^2 \left(1-\bigO(\ln^{-3 }n)\right).
\]
Using \eqref{eq:gidbound}, $\eta< {\ln^{-3} n} $ and $f= \Theta(n \ln n)$, we obtain
\begin{align*}
	T(2) = \frac{e^{\bigO (\ln^{-3}n)} \big(\sum_u t_u {u \choose 2}\big)^2}{{\ntr \choose 2}q} = \frac{2b e^{\bigO (\ln^{-3}n) }\big(f - \gid \big)^2  }{\ntr (\ntr-1)} = \frac{2bf^2}{n^2}+o(1),
\end{align*}
which is \eqref{teil2}. For \eqref{teil3}, by \eqref{eq:ntrbound2} and  as $a=O(\ln n)$ and $\eta< \ln^{-3} n$, we have
\begin{equation}
	\label{eq:Tbound}
	T(x) \le \frac{e^{\bigO \big(\frac{1}{\ln^2 n}\big)} \left(t {a \choose x} \right)^2}{{\ntr\choose x} q^{x \choose 2}} \sim  \frac{ t^2 {a \choose a-x}^2}{{n\choose x} q^{x \choose 2}} \le \frac{k^2 a^{2a-2x}}{\mu_x}.
\end{equation}
By \eqref{expectationu}, for any $u^*-1 \le x \le a-3$, 
\[  \mu_{x}= \mu_a \left(\Theta\Big( \frac{n}{\ln n}b^{-\frac{a-x}{2}}\Big)   \right)^{a-x} \ge \mu_a n^{(1-o(1))(a-x)},\]
using $a-x \le a-u^*= o (\ln (n))$. Plugging this into \eqref{eq:Tbound} we obtain
\[ \sum_{x= u^* -1}^{a-3} T(x) \lesssim \frac{k^2}{\mu_a} \sum_{x={u^*}-1}^{a-3} \left( n^{-1+o(1)}\right)^{a-x} = \frac{n^{-1+o(1)}}{\mu_a},\]
giving \eqref{teil3} --- note that this expression is $o(1)$ because $a \le \alpha$ and so $\mu_a \ge n^{o(1)}$ by \eqref{eq:mut}. Similarly as in \eqref{eq:Tbound},
\begin{align*}
	T(a-2)&\le \frac{e^{\bigO\left({1}/{\ln^ 2n}\right)} \left(t_{a-2} +(a-1) t_{a-1}+{a \choose a-2}t_a\right)^2 }{{\ntr \choose a-2} q^{{a-2 \choose 2}}} = \bigO \left( \frac{ (k_{a-2} + k_{a-1}\ln n+k_a \ln ^2 n)^2}{\mu_{a-2}} \right) \\
	&= \bigO \left(  \frac{ ( k_{a-2} \ln n + k_{a-1} \ln^2 n+k_a \ln^3 n)^2 }{\mu_{a}n^2} \right) ,
\end{align*}
noting that $\mu_{a-2} = \mu_a \Theta(n^2 /\ln^2 n) $ by \eqref{expectationu}, which is \eqref{teil4a}. Finally for \eqref{teil4}, 
\[
T(a-1) \le \frac{e^{\bigO\left(\frac{1}{\ln^ 2n}\right)} \left(a ^2 t_a^2 + 2at_at_{a-1} \right) }{{\ntr \choose a-1} q^{{a-1 \choose 2}}} \lesssim \frac{ a ^2 k_a^2 + 2ak_ak_{a-1} }{\mu_{a-1}} = \bigO \left(  \frac{  k_a^2 \ln^ 3n + k_ak_{a-1} \ln^2 n }{n\mu_{a}} \right) .
\]
\qed

\subsection{Proof of Lemmas~\ref{lemma:newlemma1} and~\ref{lemma:newlemma2}}
We prove Lemma~\ref{lemma:newlemma1} first. Let $C(\epsilon)$ be the constant from Lemma~\ref{lemma:smallsubprofiles}. Then, by that lemma we know that uniformly for all $\bfl \le \bfk$ such that ${\ln^{-3} n} \le \sum_u \lambda_u \le C(\epsilon)$,
		\begin{equation} 
	\E_p \left[ \bar{X}_\bfl\right]	\ge \exp \big(  \Theta \left(  n / \ln^ 3 n \right) \big) \prod_u {k_u \choose \ell_u}^2. \label{eq:bound1Xell}
\end{equation}
The claim now follows from Lemma~\ref{contributionlemma}, \eqref{eq:umform} and the fact that there at most $k^a= \exp(\bigO(\ln^2 n))$ possible sequences $\bfl \le \bfk$ in total (as there are at most $k^a$ ways to write $k$ as an ordered sum with at most $a$ summands). 

The proof of Lemma~\ref{lemma:newlemma2} is very similar, we only replace \eqref{eq:bound1Xell} by \eqref{eq:betadelta}. Indeed, the claim  follows with Lemma~\ref{contributionlemma}, \eqref{eq:umform}, \eqref{eq:betadelta} and the fact that there at most $k^a= \exp(\bigO(\ln^2 n))$ possible sequences $\bfl \le \bfk$ in total. \qed

\subsection{Proof of Lemma~\ref{lemma:newlemma3}} 
\label{ssec:proofmiddle3}
Let $\bfl \le \bfk$ be a sequence so that $1-\delta_0 \le \lambda \le 1-n^{-c_0}$, where $\delta_0 >0$ is a constant we will choose later. Set $\tau=1-\lambda$, so that
\begin{equation}\label{eq:eta}
 n^{-c_0} \le \tau \le \delta_0.
\end{equation} 
Our aim is to show that, if we choose $\delta_0$ appropriately, then uniformly for all such sequences~$\bfl$,
	\begin{equation}\label{eq:partialsequences3}
	\E_p \left[ \bar{X}_\bfl\right] \ge \exp \Big(\Theta(n^{1-c_0})\Big) \prod_u {k_u \choose \ell_u}^2,
\end{equation}
and then the claim will follow exactly as in the proof of Lemma~\ref{lemma:newlemma1} (or Lemma~\ref{lemma:newlemma2}), only replacing \eqref{eq:bound1Xell} (or \eqref{eq:betadelta}) by \eqref{eq:partialsequences3}.

Recall that we let
\[
    \ntr = n - \sum_{u}u\ell_u = (1-\lambda)n = \tau n.
\]
We start by rewriting $\E_p[X_\bfk]$ by first counting the partial colourings $X_\bfl$, then multiplying this with the expected number of colourings of the remaining $\ntr$ vertices with profile $\bfk - \bfl$:
\begin{align}
    \E_p \left[ {X}_\bfk \right] &= \frac{n!}{\prod_u u!^{k_u}}q^{\sum_u {u \choose 2}k_u} = \frac{n!}{\ntr!\prod_u u!^{\ell_u}}q^{\sum_u {u \choose 2}\ell_u} \cdot \frac{{\ntr}!}{\prod_u u!^{k_u-\ell_u}}q^{\sum_u {u \choose 2}(k_u-\ell_u)} \nonumber \\
	&= \E_p \left[ {X}_\bfl \right]\E_{\ntr,p} \left[ {Y}_{\bfk-\bfl}\right]\nonumber,
\end{align}
where $\E_{\ntr,p}$ denotes the expectation in $G_{\ntr,p}$ and ${Y}_{\bfk-\bfl}$ is the number of ordered colourings in $G_{\ntr,p}$ with exactly $k_u-\ell_u$ colour classes of size $u$ for all $u^* \le u \le a$; note that this is a complete colouring profile for $G_{\ntr,p}$. Letting 
\[
    \bar{Y}_{\bfk-\bfl} =  \frac{{Y}_{\bfk-\bfl}}{\prod_u (k_u-\ell_u)!},
\]
we obtain that
\[
\E_p \left[ \bar{X}_\bfl \right] = \frac{\E_p \left[ {X}_\bfk \right]}{\E_{\ntr,p} \left[ {Y}_{\bfk-\bfl}\right] \prod_u \ell_u!} = \frac{\E_p \left[ \bar{X}_\bfk \right]}{\E_{\ntr,p} \left[ \bar{Y}_{\bfk-\bfl}\right] } \prod_u {k_u \choose \ell_u}.
\]
As $\kp$ is tame, by  Definition \ref{deftame} we have $\E_p \left[ \bar{X}_\bfk \right] \ge \exp(-n^{1-c})$. We will show that, for an appropriate choice of $\delta_0$,
\begin{equation}\label{eq:lemmaaim}
	\E_{\ntr,p} \left[ \bar{Y}_{\bfk-\bfl}\right] \prod_u {k_u \choose \ell_u} \le \exp\big(-n^{1-c_0} \ln b \big)\end{equation} 
for large enough $n$, which implies \eqref{eq:partialsequences3}. To prove \eqref{eq:lemmaaim}, first note that
\begin{align}
\label{eq:boundbino}
    \prod_u {k_u \choose \ell_u} =\prod_u {k_u \choose k_u -\ell_u} \le \prod_u k_u^{k_u-\ell_u} \le n^{k-\ell}=n^{(1+o(1))\ntr / (2 \log_b n)}=b^{(\frac 12+o(1))\ntr}. 
\end{align}
The difficult part is to bound $\E_{\ntr,p} \left[ \bar{Y}_{\bfk-\bfl}\right] $. The idea here is that the average colour class size of a colouring with profile $\bfk - \bfl$ in $G_{\ntr,p}$ is significantly `too large', so that we can use Lemma \ref{lemmaupperbound} to obtain an upper bound on $\E_{\ntr,p} \left[ \bar{Y}_{\bfk-\bfl}\right]$ of the form $b^{-C \ntr }$ for any  constant $C$  we like by choosing $\delta_0$ appropriately. To check that we can apply Lemma \ref{lemmaupperbound}, let
\begin{align*}
	\bar{\alpha}_0 = \alpha_0(\ntr)= 2 \log_b \ntr -2 \log_b \log_b \ntr  + 2 \log_b (e/2)+1
        ~\text{ and }~
	\bar{\alpha} =\alpha(\ntr)= \left \lfloor \bar{\alpha}_0 \right \rfloor.
\end{align*} 
Then, as $\ntr = \tau n$, using $1 \ge \tau \ge n^{-c_0}$, for sufficiently large $n$,
\begin{equation}
	\alpha_0-\bar{\alpha}_0 = 2 \big((\log_b n -  \log_b \ntr) - (\log_b \log_b n-\log_b \log_b \ntr)\big) =-2 \log_b \tau +\bigO(1) \le 3c_0  \log_b n. \label{eq:alpha0diff}
\end{equation}
Our assumptions guarantee that $k_u = 0$ for all $u<u^*$ and $u>a$, where $u^* \sim a \sim \alpha$. Since $\alpha_0 \sim 2 \log_b n$ and as $c_0=c/3 < 1/3$, by \eqref{eq:alpha0diff} clearly $k_u - \ell_ u =k_u=0$ for all $u<0.1 \bar{\alpha}$ and $u>10 \bar{\alpha}$ if $n$ is large enough, meeting the first two conditions of Lemma~\ref{lemmaupperbound}. It remains to bound the average colour class size ${\ntr}/(k-\ell)$ of the colouring profile $\bfk-\bfl$ from below, i.e., we show that it is significantly larger than $\bar{\alpha}_0-1-{2}/{\ln b}$.

Recall $t_u = k_u - \ell_u$ and let $\boldsymbol{t} = \bfk-\bfl$. Given $\bfk$ and $\ntr$, what choice of $\boldsymbol{t}$ minimises the average colour class size? Clearly we want to pick the colour classes as small as possible, but subject to our constraints $t_u \le k_u$ for all $u$ and $\sum_u t_u u =\ntr$. Recall that $k_u = \kappa_u n/u \le b^{-(\alpha-u)\gamma(\alpha-u)}n/u^*$, where $\gamma$ is the function from Definition~\ref{deftame}. The idea is that the average colour class size in $\boldsymbol{t}$ is at least as much as that of a (hypothetical) colouring profile where we take exactly $b^{-(\alpha-u)\gamma(\alpha-u)}n/u^*$ colour classes of size $u$ for $u= u^*, u^*+1, ...$ until $\ntr$ vertices are reached. To formalise this idea, let 
\[\tilde k_u :=  b^{-(\alpha-u)\gamma(\alpha-u)} n / u^* \ge k_u,\]
and define a sequence of real numbers
\[0=n_{u^*-1}< n_{u^*}< n_{u^*+1} < \dots < n_a\]
by setting
\[
    n_u :=\sum_{u^* \le v \le u}  v \tilde k_v .
\]
As $\tilde k_v \ge k_v$ for all $v$, we clearly have $n_a \ge n > \ntr$. 
So let $u_0$ be the unique index such that $n_{u_0} \le \ntr < n_{u_0+1}$. 
With these definitions at hand, consider first the case $u_0 = u^*-1$. Then
\begin{equation}\label{eq:ustarbound}
	b^{-(\alpha-u^*)\gamma(\alpha-u^*)} n = n_{u^*}> \ntr = \tau n,
\end{equation}
and we simply bound the average colour class size in $\boldsymbol{t}$ from below by $u^*$. Recall that we defined $u^*$ so that $\gamma(\alpha-u^*) \ge 10$ (see the remarks at the end of \S\ref{section:tamecolourings}). Therefore \eqref{eq:ustarbound} implies that $\log_b \tau < -10 (\alpha-u^*)$, and so the average colour class size of $\boldsymbol t$ is at least
\begin{equation} \label{eq:bound1stcase}
	u^* \ge \alpha+\frac{\log_b \tau}{10}. 
\end{equation}
In the second case, if $u_0 \ge u^*$, we first want to show that, if we pick $\delta_0>0$ small enough, then we can ensure that 
	\begin{equation}\gamma(\alpha-u_0-1) \ge \max(\log_b 10, 10).
		\label{eq:gammabound}
			\end{equation}
For this, first note that as $u^* k_{u_0} \le u_0 k_{u_0} \le n_{u_0}\le \ntr = \tau n$, we have  $b^{-(\alpha-u_0)\gamma(\alpha-u_0)} \le \tau \le \delta_0$. Now, since $\gamma(x) \rightarrow \infty$ as $x \rightarrow \infty$, there is a constant\footnote{which only depends on the function $\gamma(x)$, i.e.\ on the choice of the tame sequence $\bfk$} $C_\gamma$  so that either $\alpha - u_0 \le C_\gamma$, or~\eqref{eq:gammabound} holds. But as $\gamma(x)$ is increasing, $\alpha - u_0 \le C_\gamma$ implies $\delta_0 \ge  b^{-(\alpha-u_0)\gamma(\alpha-u_0)} \ge b^{-C_\gamma \gamma(C_\gamma)}$. So by picking $\delta_0>0$ small enough we can exclude this possibility, and \eqref{eq:gammabound} follows.

Continuing with the second case $u_0 \ge u^*$, the average colour class size in $\boldsymbol t$ is bounded from below by 
\[
    \sum_{u^* \le u \le u_0} \left( u \cdot \frac{\tilde k_u}{\sum_{u^* \le v \le u_0}\tilde k_v} \right) = \sum_{u^* \le u \le u_0} \left(u \cdot \frac{x_u}{\sum_{u^* \le v \le u_0}x_v} \right),
    \quad
    \text{where}
    \quad
    x_u=b^{-(\alpha-u)\gamma(\alpha-u)}.
\]
Setting $y=b^{-\gamma(\alpha-u_0)}$, as $\gamma$ is increasing this is at least
\begin{equation*}
	\sum_{i=0}^{u_0-u^*} (u_0-i) \cdot \frac{x_{u_0-i}}{\sum_{v=u^*}^{u_0}x_v}= u_0-\sum_{i=0}^{u_0-u^*} i \cdot \frac{x_{u_0-i}}{\sum_{v=u^*}^{u_0}x_v} \ge u_0 - \sum_{i=0}^{u_0-u^*} i \cdot \frac{x_{u_0-i}}{x_{u_0}} \ge u_0 - \sum_{i=0}^{u_0-u^*} i y^{i}. \label{eq:adjks}
\end{equation*}
As $\gamma(x)$ is increasing, by  \eqref{eq:gammabound}  we have $\gamma(\alpha-u_0) \ge \log_b 10$ and so the average colour class size of $\boldsymbol{t}$ is at least
\begin{equation}\label{eq:gkjhjkdfgh}
	u_0 -\sum_{i \ge 0} i \cdot 10^{-i}> u_0-1.
\end{equation}
So now we need a lower bound on $u_0$. For this, note that as $n_{u_0+1} > \ntr = \tau n$, again letting $x_u=b^{-(\alpha-u) \gamma(\alpha-u)}$ and setting $z=b^{-\gamma(\alpha-u_0-1)}$, 
we have
\begin{align*}
	\tau &< \sum_{v=u^*}^{u_0+1}  \frac{\tilde k_v v}{n} \le \frac{u_0+1}{u^*}\sum_{u=u^*}^{u_0+1}x_u \le  \frac{u_0+1}{u^*}\sum_{u=u^*}^{u_0+1}z^{\alpha-u} \le 2 \cdot b^{-10(\alpha-u_0-1)}
\end{align*}
if $n$ is large enough, using in the last inequality that $u_0 \sim \alpha \sim u^*$ and that $z= b^{-\gamma(\alpha-u_0-1)} \le \min( b^{-10}, 1/{10})$ by \eqref{eq:gammabound}. It follows that
\[
    u_0 \ge \alpha+\frac{\log_{b}(\tau/2)}{10}+\bigO(1).
\]
Continuing from \eqref{eq:gkjhjkdfgh}, the average colour class size is bounded from below by
\[
u_0 - 1 \ge\alpha+\frac{\log_{b}(\tau/2)}{10}+\bigO(1).
\]
Putting this together with \eqref{eq:bound1stcase}, we see that in both cases the lower bound for the average colour class size in $\boldsymbol{s}$  is $ \alpha_0+{(\log_{b}\tau)}/{10}+\bigO(1)$. Comparing to \eqref{eq:alpha0diff}, we have $\bar \alpha_0 = \alpha_0+2 \log_b \tau +\bigO(1)$. Note that $\tau < \delta_0$, and so we now pick the constant $\delta_0 >0$ small enough so that the average colour class size is bounded from below by
\[
    \bar \alpha_0+(-2 + 1/ {10})\log_b \tau+\bigO(1) \ge \bar \alpha_0+10.
\]
Thus we may apply Lemma~\ref{lemmaupperbound}. Together with \eqref{eq:boundbino}, as by \eqref{eq:eta} $\ntr =\tau n \ge n^{1-c_0}$ we obtain
\[
\E_{\ntr,p} \left[ \bar{Y}_{\bfk-\bfl}\right] \prod_u {k_u \choose \ell_u}  \le b^{-2\ntr+(1/2+o(1))\ntr}<\exp\left( - n^{1-c_0} \ln b\right)
\]
if $n$ is large enough, which is \eqref{eq:lemmaaim} and implies \eqref{eq:partialsequences3}. As in the proofs of Lemma~\ref{lemma:newlemma1} and \ref{lemma:newlemma2}, the claim now follows from  \eqref{eq:partialsequences3} and Lemma~\ref{contributionlemma}, \eqref{eq:umform} and the fact that there at most $k^a= \exp(\bigO(\ln^2 n))$ possible sequences $\bfl$ in total (as there are at most $k^a$ ways to write $k$ as an ordered sum with at most $a$ summands).
\qed

\subsection{Proof of Lemma~\ref{lemma:checksle3}} 
We will distinguish several cases. In each case we derive a lower bound for $\varphi(s,x,i_0)$ that is a function of $(\zeta_i)_{i_0\le i \le s}$ and then use monotonicity and concavity properties and the approximations from \S\ref{section:numerics}  to show that we only need to check non-negativity of this function at a finite number of values.

\paragraph{Case 1: $i_0=s=1$ and $x \in [0.04,1]$.}
Note that
\[
\varphi(1,x,1) = - \Big(1-\zeta_1 \Big) \ln  \Big(1- \zeta_1 \Big)  + \zeta_1 \left(-1+ \frac{\ln 2}{2} x \right).
\]
By Lemma~\ref{lemma:monoxi}, $\zeta_1(x)$ is increasing for $x \in [0,1]$. We distinguish three cases.

\paragraph{Case 1.1: $x \in [0.04, 0.15]$.}
As $x \ge 0.04$ we obtain that $\varphi(1,x,1) \ge h_1(\zeta)$, where
\[
    h_1(y) =  - (1-y ) \ln (1- y)  +y \left(-1+ \frac{\ln 2}{2} \cdot 0.04 \right).
\]
Using the approximations in \S\ref{section:numerics} we have
\[
0 <	\zeta_1(0.04) \le \zeta_1(x) \le \zeta_1(0.15) < 0.026.
\]
As $h_1$ is concave, it takes its minima on $[0, 0.026]$  at the border points, so as $h_1(0)=0$ and $h_1(0.026)\approx 0.000019$, we have $h_1(\zeta_1)>0$.

\paragraph{Case 1.2: $x \in (0.15, {2}/{\ln 2}-2]$.}
As $x > 0.15$ we obtain that $\varphi(1,x,1) \ge h_2(\zeta)$, where
\[
    h_2(y) =  - (1-y ) \ln (1- y)  +y \left(-1+ \frac{\ln 2}{2} \cdot 0.15 \right).
\]
From the approximations in \S\ref{section:numerics} we know that $0< \zeta_1 <0.092$. As $h_2$ is concave and $h_2(0)=0$, and $h_2(0.092) \approx 0.00041$, we have $h_2(\zeta_1)>0$.

\paragraph{Case 1.3: $x \in ({2}/{\ln 2}-2, 1]$.}
As $x > 0.88$ we obtain that $\varphi(1,x,1) \ge h_3(\zeta)$, where
\[
    h_3(y) =  - (1-y ) \ln (1- y)  +y \left(-1+ \frac{\ln 2}{2} \cdot 0.88 \right).
\]
The approximations in \S\ref{section:numerics} guarantee that $0< \zeta_1 <0.11$. As $h_3$ is concave and $h_3(0)=0$, and $h_3(0.11) \approx 0.027$, we have $h_2(\zeta_1)>0$.

\paragraph{Case 2: $i_0=s=2$ and $x \in [0,1]$.}
Note that
\[
\varphi(2,x,2) = - (1-\zeta_2 ) \ln  (1- \zeta_2 )  + \frac{\ln 2}{2} \zeta_2 \big(1-\frac{2}{\ln 2 }+x \big) \ge - (1-\zeta_2 ) \ln  (1- \zeta_2 )  +  \zeta_2 \Big(\frac{\ln 2}{2}-1 \Big).
\]
It is not hard to check that this last expression is positive if $0<\zeta_2<0.5$. Note that by Lemma~\ref{lemma:monoxi} and the approximations in \S\ref{section:numerics},
\[
0<\zeta_2(x) \le \zeta_2(1) <0.4,
\]
and so $\varphi(2,x,2) > 0$ for all $x \in [0,1]$.

\paragraph{Case 3: $i_0=1$, $s=2$ and $x \in [0,1]$.}
Note that
{
\small
\[
\varphi(2,x,1) = 	- (1-\zeta_1-\zeta_2) \ln  (1-\zeta_1-\zeta_2)  + \frac{\ln 2}{2}\Big( \zeta_1 \big(-\frac{2}{\ln2}+x\big) + \zeta_2 \big(1-\frac{2}{\ln2}+x\big)\Big) \ge h_4(\zeta_1, \zeta_2),
\]
}
where
\[
h_4(y, z) = -(1-y-z)\ln(1-y-z) +\frac{\ln 2}{2} \Big( -y\frac{2}{\ln2} + z \big(1-\frac{2}{\ln2}\big)\Big).
\]
The function $h_4$ is concave --- all four of its partial second derivatives are $-\frac{1}{1-y-z} <0$, and in particular its Hessian matrix is negative semidefinite. Recall that by Lemma~\ref{lemma:monoxi}, $\zeta_1, \zeta_2$ are both increasing for $x \in [0, 1]$, so by the approximations in \S\ref{section:numerics}, $ 0.018 < \zeta_1 <  0.11$ and $0.09< \zeta_2 < 0.28 $. So $h_4(\zeta_1, \zeta_2)$ is positive if $h_4$ is non-negative on all four `border points' of the possible range of $(\zeta_1, \zeta_2)$, which is indeed the case: 
\begin{align*}
	h_4(0.018,0.09) &\approx 0.025,
	~~h_4(0.11, 0.09) \approx 0.0097, \\
	h_4(0.018, 0.28) & \approx 0.047, 
	~~h_4(0.11, 0.28) \approx 0.0086. 
\end{align*}

\paragraph{Case 4: $i_0=2$, $s=3$ and $x \in [0,1]$.}
Note that
\[
\varphi(3,x,2) = 	- (1-\zeta_2-\zeta_3) \ln  (1-\zeta_2-\zeta_3)  + \frac{\ln 2}{2}\Big( \zeta_2 \big(1-\frac{2}{\ln2}+x\big) + \zeta_3 \big(2-\frac{2}{\ln2}+x\big)\Big).
\]
We distinguish two cases.

\paragraph{Case 4.1: $x \in [0, 2/{\ln 2}-2]$.}
From the definition we obtain that $\varphi(3,x,2) \ge h_5 (\zeta_2, \zeta_3)$, where
\[
h_5 (y,z) = - (1-y-z) \ln  (1-y-z)  + \frac{\ln 2}{2}\Big( y  \big(1-\frac{2}{\ln2}\big) +z \big(2-\frac{2}{\ln2}\big)\Big).
\]
This function is concave --- all four of its partial second derivatives are $-\frac{1}{1-y-z} <0$, so in particular its Hessian matrix is negative semidefinite. Recall that by Lemma~\ref{lemma:monoxi}, $\zeta_2, \zeta_3$ are both increasing for $x \in [0,  2/{\ln 2}-2]$, so by the approximations in \S\ref{section:numerics}, $ 0 < \zeta_2 <  0.35$ and $ 0< \zeta_3 < 0.39 $. So $h_5 (\zeta_2, \zeta_3)$ is positive if $h_5$ is non-negative on all four `border points' of the possible range of $(\zeta_2, \zeta_3)$, which is indeed the case: clearly $h_5(0,0) = 0$, and
\begin{align*}
	h_5(0.35, 0) \approx 0.05,
	~~h_5(0, 0.39) &\approx 0.18,
	~~h_5(0.35, 0.39) \approx 0.002. 
\end{align*}

\paragraph{Case 4.2: $x \in (2/{\ln 2}-2, 1]$.}
As $x>\frac 2 {\ln 2}-2$, we bound $\varphi(3,x,2) \ge h_6 (\zeta_2, \zeta_3)$, where
\[
    h_6 (y,z) = - (1-y-z) \ln  (1-y-z)  - \frac{\ln 2}{2} y.
\]
As in the former cases, this function is concave (all four partial second derivatives are $-\frac{1}{1-y-z} <0$). By Lemma~\ref{lemma:monoxi}, $\zeta_2$ is increasing and $\zeta_3$ is decreasing for $x \in  [\frac 2{\ln 2}-2, 1]$, so by the approximations in \S\ref{section:numerics}, $ 0.34 < \zeta_2 <  0.4$ and $ 0.37< \zeta_3 < 0.39 $. So again to check that $h_6 (\zeta_2, \zeta_3)$ is positive, we only need to make sure $h_6$ is non-negative on all four `border points', which it is:
\begin{align*}
	h_6(0.34,0.37) &\approx 0.24,
	~~h_6(0.4, 0.37) \approx 0.20, \\
	h_6(0.34, 0.39) &\approx 0.24,
	~~h_6(0.4, 0.39) \approx 0.19. 
\end{align*}

\paragraph{Case 5: $i_0=1$, $s=3$ and $x \in [0,1]$.}
Note that
{
\small
\[
    \varphi(3,x,1)
    =
    - (1-\zeta_1-\zeta_2-\zeta_3) \ln  (1-\zeta_1-\zeta_2-\zeta_3)  + \frac{\ln 2}{2}\Big( \zeta_1 \big(-\frac{2}{\ln2}+x\big) + \zeta_2 \big(1-\frac{2}{\ln2}+x\big)+ \zeta_3 \big(2-\frac{2}{\ln2}+x\big)\Big).
\]
}
We distinguish two cases.

\paragraph{Case 5.1: $x \in [0,  2/{\ln 2}-2]$.}
We obtain that $\varphi(3,x,1) \ge h_7 (\zeta_1,\zeta_2, \zeta_3)$, where
{
\small
\[
    h_7 (y,z,v) = - (1-y-z-v) \ln  (1-y-z-v)  + \frac{\ln 2}{2}\Big( -\frac{2}{\ln2} y + z \big(1-\frac{2}{\ln2}\big)+ v \big(2-\frac{2}{\ln2}\big)\Big).
\]
}
As in previous cases, it is easy to establish this function is concave. By Lemma~\ref{lemma:monoxi}, $\zeta_1, \zeta_2, \zeta_3$ are all increasing for $x \in [0,  2/{\ln 2}-2]$, so by the approximations in \S\ref{section:numerics}, $0.018 < \zeta_1 <  0.092$, $ 0.098 < \zeta_2 <  0.25$ and $ 0.25< \zeta_3 < 0.34 $. So $h_7 (\zeta_1, \zeta_2, \zeta_3)$ is positive if $h_7$ is non-negative on all eight `border points', which it is: 
\begin{align*}
	h_7(0.018, 0.098, 0.25) &\approx 0.130,
	~~h_7(0.018, 0.098, 0.34) \approx 0.14, \\
	h_7(0.018, 0.25, 0.25) &\approx 0.094,
	~~h_7(0.018, 0.25, 0.34) \approx 0.081, \\
	h_7(0.092, 0.098, 0.25) &\approx 0.092,
	~~h_7(0.092, 0.098, 0.34) \approx 0.094, \\
	h_7(0.092, 0.25, 0.25) &\approx 0.034,
	~~h_7(0.092, 0.25, 0.34) \approx 0.0046.
\end{align*}

\paragraph{Case 5.2: $x \in (2/{\ln 2}-2, 1]$.} 
For $x >  2/{\ln 2}-2$ we obtain that $\varphi(3,x,1) \ge h_8 (\zeta_1,\zeta_2, \zeta_3)$, where
\[
    h_8 (y,z,v) = - (1-y-z-v) \ln  (1-y-z-v)  + \frac{\ln 2}{2}( -2y - z ).
\]
As in previous cases, this function is concave. By Lemma~\ref{lemma:monoxi} and the approximations in \S\ref{section:numerics} we obtain $0.091 < \zeta_1 <  0.11$, $ 0.24 < \zeta_2 <  0.28$ and $ 0.33 < \zeta_3 < 0.34 $, so for the final time, all that's left is to check positivity at the border points:
\begin{align*}
	h_8(0.091, 0.24, 0.33) &\approx 0.22, 
	~~h_8(0.091, 0.24, 0.34) \approx 0.22, \\
	h_8(0.091, 0.28, 0.33) &\approx 0.20, 
	~~h_8(0.091, 0.28, 0.34) \approx 0.20, \\
	h_8(0.11, 0.24, 0.33) &\approx 0.21, 
	~~h_8(0.11, 0.24, 0.34) \approx 0.20, \\
	h_8(0.11, 0.28, 0.33) &\approx 0.18, 
	~~h_8(0.11, 0.28, 0.34) \approx 0.18.
\end{align*}
\qed

\section{Appendix: Numerical approximations}
\label{app:numApprox}

In this section we explain how we obtained the numerical approximations from \S\ref{section:numerics}. The R code with the necessary computations and checks can be found here:

\vspace{4pt}

\url{https://gist.github.com/annikaheckel/8dd24d33f5e780c09feffcfdb13f5407}

\vspace{4pt}

\noindent Recall from \eqref{eq:defmusimpler} that $\mu=\mu(x)$ is the solution of
\begin{equation}
	f(\mu, x, i_0) :=	\sum_{i \ge i_0} \big(i-T(x)\big)e^{\mu i - \frac{\ln 2}{2} i^2} =0,
\end{equation}
where $T(x)=1 + {2}/{\ln 2} - x$. Of course, $\mu(x)$ can be approximated with computer assistance by simply truncating the infinite series after the first couple of terms --- since they decrease as $\exp\{-\frac{\ln 2}{2}i^2\}$, they become extremely small very quickly and one can obtain very accurate approximations. 
However, the error term --- the infinitely many terms of the series that were left out --- depends on $\mu$ itself, so making these approximations mathematically rigorous takes some care.
In order to achieve this, note that, if $\mu_2= \mu_2(x)$ satisfies
\begin{equation}\label{eq:mutrunc}
	\sum_{i_0 \le i \le 20} \big(i-T(x)\big)e^{\mu_2 i - \frac{\ln 2}{2} i^2} >0,
\end{equation}
then also $f(\mu_2, x, i_0)>0$. As argued in the proof of Lemma~\ref{lemma:monomu} (just after \eqref{eq:muroot}), this implies that  also $f(\tilde\mu, x, i_0)>0$ for all $\tilde \mu > \mu_2$. It follows that $\mu < \mu_2$. So to prove the upper bounds on $\mu(x)$ given in \S\ref{section:numerics}, we only need to check~\eqref{eq:mutrunc}.

Let us turn to the lower bound, where we need to consider the error term. Suppose that we have some $\mu_1<3$ so that
\begin{equation}\label{eq:mutrunc2}
	\sum_{i_0 \le i \le 20} \big(i-T(x)\big)e^{\mu_1 i - \frac{\ln 2}{2} i^2} < -e^{-83}.
\end{equation}
Then note that
\[
\sum_{i \ge 21} (i-T(x)) e^{\mu_1 i - \frac{\ln 2}{2} i^2} <  \sum_{i \ge 21} i e^{3i - \frac{\ln 2}{2} i^2} <  \sum_{i \ge 21} e^{(3.1- \frac{\ln 2}{2}i) i} \le \sum_{i \ge 21} e^{-4i}=\frac{e^{-84}}{1-e^{-4}}<e^{-83},
\]
and so  from \eqref{eq:mutrunc2} we infer that $f(\mu_1, x, i_0)<0$. Similarly to the previous arguments, this implies $f(\tilde \mu, x, i_0)<0$ for all $\tilde \mu \le \mu_1$, so we must have $\mu>\mu_1$. So to prove the lower bounds on $\mu(x)$ given in \S\ref{section:numerics}, we only need to check for that \eqref{eq:mutrunc2} holds (and that $\mu_1<3$, which is the case for all the lower bounds in in \S\ref{section:numerics}). 

Given upper and lower bounds $\mu_1<\mu <\mu_2<3$, we can deduce bounds on $\lambda=\lambda(x)$. From \eqref{eq:deflambdamu} we obtain that 
\[
    \lambda = - \ln  \Big( \sum_{i \ge i_0} e^{\mu i-\frac{\ln 2 }{2}i^2} \Big).
\]
Bounding $\sum_{i \ge 21}  e^{\mu' i-\frac{\ln 2 }{2}i^2}  < e^{-83}$ as above, we have
\[
- \ln  \Big( \sum_{i_0 \le i \le 20} e^{\mu_2 i-\frac{\ln 2 }{2}i^2} +e^{-83}\Big) < \lambda < - \ln  \Big( \sum_{i_0 \le i \le 20} e^{\mu_1 i-\frac{\ln 2 }{2}i^2}\Big).
\]
Finally, given that $\mu \in(\mu_1,\mu_2)$ and $\lambda \in (\lambda_1,\lambda_2)$, we readily obtain the following bounds on $\zeta_i=\zeta_i(x)$ from the definition \eqref{eq:defxi} of $\zeta_i$:
\[
e^{\lambda_1+\mu_1 i-\frac{\ln 2}{2}i^2}< \zeta_i < e^{\lambda_2+\mu_2 i-\frac{\ln 2}{2}i^2} \enspace.
\]

\end{document}